\documentclass{cocv}

\usepackage{amsfonts}
\usepackage{amsmath}
\usepackage{amssymb}
\usepackage{graphicx}
\usepackage{epstopdf}
\usepackage{algorithm}
\usepackage{algorithmic}
\numberwithin{equation}{section}
\usepackage{float}
\usepackage{diagbox}

\usepackage{hyperref}
\usepackage[capitalise]{cleveref}

\usepackage{subfigure}
\usepackage{dsfont}
\usepackage{enumitem}

\newcommand{\closure}{\operatorname{cl}}
\newcommand{\var}{\operatorname{var}}
\newcommand{\dd}{\mathrm{d}}
\newcommand{\spann}{\operatorname{span}}
\newcommand{\interior}{\operatorname{int}}
\newcommand{\dist}{\operatorname{dist}}
\newcommand{\supp}{\operatorname{supp}}
\newcommand{\conv}{\operatorname{conv}}

\DeclareMathOperator*{\argmin}{arg\,min}
\DeclareMathOperator*{\esssup}{ess\,sup}

\newcommand{\strict}{\mathrm{str}}
\newcommand{\crit}{\mathrm{crit}}
\newcommand{\red}{\mathrm{red}}
\newcommand{\rad}{\mathrm{rad}}

\renewcommand{\tan}{\mathrm{tan}}
\newcommand{\ptw}{\mathrm{ptw}}

\renewcommand{\AA}{\mathcal{A}}
\newcommand{\BB}{\mathcal{B}}

\newcommand{\DD}{\mathcal{D}}

\newcommand{\FF}{\mathcal{F}}

\newcommand{\HH}{\mathcal{H}}
\newcommand{\II}{\mathcal{I}}
\newcommand{\JJ}{\mathcal{J}}
\newcommand{\KK}{\mathcal{K}}

\newcommand{\NN}{\mathcal{N}}

\newcommand{\PP}{\mathcal{P}}

\renewcommand{\SS}{\mathcal{S}}

\newcommand{\ZZ}{\mathcal{Z}}

\newcommand{\N}{\mathbb{N}}

\newcommand{\R}{\mathbb{R}}

\newcommand{\veps}{{\varepsilon}}

\newcommand{\weakly}{\rightharpoonup}

\newtheorem{theorem}{Theorem}[section]
\newtheorem{corollary}[theorem]{Corollary}
\newtheorem{lemma}[theorem]{Lemma}
\newtheorem{assumption}[theorem]{Assumption}
\newtheorem{example}[theorem]{Example}
\newtheorem{proposition}[theorem]{Proposition}
\newtheorem{definition}[theorem]{Definition}
\newtheorem{remark}[theorem]{Remark}

\newcommand{\dimm}{\operatorname{dim}}

\begin{document}

\title{Directional differentiability for solution operators of sweeping processes
with convex polyhedral admissible sets}
\runningtitle{Directional differentiability for sweeping processes}

\author{Martin Brokate}\address{Technische Universit\"{a}t M\"{u}nchen, CIT, 
Department of Mathematics,
Boltzmannstra{\ss}e 3,
85748 Garching, Germany;
Weierstrass Institute for Applied Analysis and Stochastics, 
Anton-Wilhelm-Amo-Stra{\ss}e 39,
10117 Berlin, Germany;
Faculty of Civil Engineering, Czech Technical University in Prague, 
Th\'{a}kurova 7, 16629 Praha 6, Czech Republic, 
\email{brokate@ma.tum.de}}
\author{Constantin Christof}\address{
Technische Universit\"{a}t Darmstadt, Department of Mathematics, 
Dolivostra{\ss}e 15,
64293 Darmstadt, Germany,\\ \email{christof@mathematik.tu-darmstadt.de}}


\begin{abstract}
We study directional differentiability properties of solution operators of 
rate-independent evolution variational inequalities 
with full-dimensional convex polyhedral admissible sets. 
It is shown that, if the space of continuous functions of bounded variation 
is used as the domain of definition, 
then the most prototypical examples of such solution operators---the vector play and stop---are 
Hadamard directionally differentiable in a pointwise manner 
\emph{if and only if} the  admissible set is 
non-obtuse. 
We further prove that, 
in those cases where they exist, the directional derivatives of the vector play and stop are uniquely characterized 
by a system of projection  identities and variational inequalities
and that directional differentiability cannot be expected in the obtuse case even if 
the solution operator is restricted to the space of Lipschitz continuous functions. 
Our results can be used, for example, to 
formulate Bouligand stationarity conditions 
for optimal control problems involving sweeping processes. 
\end{abstract}
%
%
\maketitle

\section{Introduction}
This paper is concerned with 
directional 
differentiability properties of 
solution maps 
$(u, x_0) \mapsto x$
of sweeping processes of the form 
\begin{equation}
\label{eq:sweep_problem}
x(0) = x_0,
\qquad
x(t) \in Z - u(t)
~\forall t \in [0,T],
\qquad 
\dot x(t) \in - \NN_{Z - u(t)}(x(t))
\text{ for a.a.\ }t \in (0,T),
\end{equation}
involving 
a 
full-dimensional convex polyhedron $Z \subset \R^d$,
$d \in \N$,
an initial value 
$x_0 \in Z-u(0)$,
a terminal time 
$T>0$,
a forcing term 
$u\colon [0,T] \to \R^d$,
and the outward normal cone 
$\NN_{Z - u(t)}(x(t))$
of the set $Z - u(t)$ at $x(t)$. 
We study the problem \eqref{eq:sweep_problem} in 
its evolution variational inequality (EVI) formulation
\begin{equation}
\label{eq:sweep_EVI_strong}
y(0) = y_0,
\qquad
y(t) \in Z~\forall  t \in [0,T],
\qquad
\left \langle 
\dot y(t) - \dot u(t), v - y(t)
\right \rangle 
\geq 0~\forall v \in Z
\text{ for a.a.\ }t \in (0,T),
\end{equation}
that is obtained by 
introducing the variable $y := x + u$,
by defining $y_0 := x_0 + u(0) \in Z$,
and by rewriting the normal-cone inclusion 
in \eqref{eq:sweep_problem}
in terms of the Euclidean scalar product
$\langle \cdot, \cdot \rangle\colon 
\R^d \times\R^d \to \R$.
More precisely, 
we consider 
the following weak formulation 
of \eqref{eq:sweep_EVI_strong} 
that is not only sensible 
for  
absolutely continuous inputs $u$
but also for all $u$ in the 
space $CBV([0,T];\R^d)$ of 
continuous functions of bounded variation
with values in $\R^d$:
\begin{equation*}
\label{eq:EVI_again}
\tag{V}
y \in CBV([0,T]; \R^d),
~~~\quad 
y(0) = y_0,
~~~\quad 
y(t) \in Z~\forall t \in [0, T],
~~~\quad
\int_0^T \left \langle v - y, \dd (y - u) \right \rangle \geq 0~\forall v \in C([0, T]; Z).
\end{equation*}
Here, $C([0, T]; Z)$ denotes the 
set of all continuous functions 
$v\colon [0,T] \to \R^d$ with values in 
$Z$ and the integral has to be understood
in the sense of 
Kurzweil-Stieltjes (see \cref{subsec:2.4}).
Recall that, in the context of 
the EVI-formulations
\eqref{eq:sweep_EVI_strong} 
and
\eqref{eq:EVI_again},
the solution mapping 
$\SS\colon (u, y_0) \mapsto y$
is also referred to as the 
\emph{stop operator}. 
Its twin, the function 
$\PP(u, y_0) := u - \SS(u, y_0)$,
is called the \emph{play operator}.
The main result of this paper 
characterizes precisely 
for which 
convex polyhedral admissible sets $Z$
the solution map $\SS\colon (u, y_0) \mapsto y$ of \eqref{eq:EVI_again} 
possesses a pointwise 
(or, more precisely,
$BV([0,T];\R^d)$-weak-star)
directional derivative.

\subsection{Main result}

The main result of this paper is the following theorem:

\begin{theorem}
\label{th:main_summary}
Let $T>0$ be given and 
let $Z \subset \R^d$, $d \in \N$, be a 
full-dimensional convex 
polyhedron.
Then the variational inequality
\eqref{eq:EVI_again} 
possesses a unique solution 
$y \in CBV([0,T]; \R^d)$
for all $(u, y_0) \in CBV([0,T]; \R^d) \times Z$.
The associated solution operator 
$\SS\colon CBV([0,T];\R^d) \times Z \to CBV([0,T];\R^d)$,
$(u, y_0) \mapsto y$, 
is globally Lipschitz continuous 
in the sense that there exists a 
constant $L > 0$
satisfying 
\begin{equation}
\label{eq:Lipschitz_estimate_42}
\|\SS(u, y_0) - \SS(\tilde u, \tilde y_0)\|_{CBV([0,T];\R^d)}
\leq 
L
\left (
\|u - \tilde u\|_{CBV([0,T];\R^d)}
+
|y_0 - \tilde y_0|
\right )
~
\forall 
(u, y_0),(\tilde u, \tilde y_0) \in 
CBV([0,T];\R^d) \times Z.
\end{equation}
Further,  the following 
statements hold regarding the directional differentiability properties of $\SS$.
\begin{enumerate}[label=(\Roman*)]
\item\label{main_sum:item:I}
 If $Z$ is non-obtuse 
(see \cref{def:properties_descript}),
then the map $\SS$ is directionally differentiable 
in the sense that, for all 
$(u, y_0) \in CBV([0,T]; \R^d) \times Z$
and all
$(h, h_0) \in CBV([0,T]; \R^d) \times \R^d$ 
satisfying $y_0 + \tau_0 h_0 \in Z$ for 
some $\tau_0 >0$,
there exists a unique 
$\delta := \SS'((u, y_0);(h, h_0))  \in BV([0,T];\R^d)$
satisfying
\begin{equation}
\label{eq:pointwise_dir_diff}
\lim_{(0,\tau_0) \ni \tau \to 0}
\frac{\SS(u + \tau h, y_0 + \tau h_0)(t) - \SS(u, y_0)(t)}{\tau}
= \delta(t)\qquad \forall t \in [0,T].
\end{equation}
This directional derivative $\delta$
of $\SS$ at $(u, y_0)$
in direction $(h, h_0)$
is uniquely characterized by the system%
\begin{equation}
\label{eq:char_sys_1}
\begin{aligned}
\qquad\qquad\quad&\delta \in BV([0,T];\R^d),
\qquad
\delta_+ \in \KK_{G_r}^{\crit}(y,u),
\qquad
\delta(0) = h_0,\\
&\delta(t) = \pi_{\ZZ(y(t))}(\delta(t-))
~\forall t \in [0,T],
\qquad
\delta(t+) = \pi_{V^+(t)}(\delta(t))
~\forall t \in [0,T],\\
&\int_{s_1}^{s_2}
\langle  z , \dd (\delta-h) \rangle
+
\int_{s_1}^{s_2}
\langle \delta, \dd h \rangle
- \frac12 | \delta(s_2) |^2
+ \frac12 | \delta(s_1) |^2
\geq 0
\quad 
\forall 0\leq s_1 < s_2 \leq T
\quad
\forall z \in \KK_{G_r}^{\crit}(y,u).
\end{aligned}
\end{equation}
Here, 
$\delta_+$ denotes the 
right-limit function of $\delta$;
$y$ is defined by  $y:= \SS(u,y_0)$;
$\KK_{G_r}^{\crit}(y,u)$ 
is given by 
\[
\qquad\qquad\quad
\KK_{G_r}^{\crit}(y,u)
:=
\left \{
z \in G_r([0,T]; \R^d)~
\left  |~
z(t) \in \ZZ(y(t))
~\forall t \in [0,T]
\text{ and }
\int_0^T \langle z, \dd (y-u) \rangle = 0\right.
\right \};
\]
the symbol $G_r([0,T]; \R^d)$
denotes the space of right-continuous 
regulated functions $v\colon[0,T] \to \R^d$; 
$\ZZ(x)$ is the tangent cone 
of $Z$ at $x \in Z$;
$\delta(t-)$ and $\delta(t+)$
denote a left and right limit, respectively;
$\pi$ denotes a Euclidean projection;
and $V^+(t)$ is a suitably defined 
time-dependent 
subspace (see \cref{def:temporal_indices}). 
\item\label{main_sum:item:II}
If $Z$ is not non-obtuse, 
then the map
$\SS$
is not  directionally 
differentiable in the pointwise 
sense of \ref{main_sum:item:I},
not even if
its domain of definition 
$CBV([0,T]; \R^d) \times Z$ 
is replaced 
by $W^{1,\infty}((0,T);\R^d) \times Z$.
\end{enumerate}
\end{theorem}

For details on the unique solvability 
part of \cref{th:main_summary}
and the Lipschitz estimate \eqref{eq:Lipschitz_estimate_42},
we refer the reader to 
\cref{prop:unique_solvability_stop_play}
and \cref{th:CBV-Lipschitz}.
(These two results rely heavily on the preliminary work of \cite{Krejci1999,Krasnoselskii2012,Dupuis1991,Desch1999}.)
The proofs of the statements 
\ref{main_sum:item:I} and \ref{main_sum:item:II},
which are the main contributions of this paper, 
can be found in \cref{sec:3,sec:4}; see in particular \cref{th:dir_diff_final_contra,th:non_dir_diff}. 
Before we discuss the 
role and background of 
\cref{th:main_summary} in more detail,
we give some remarks on its mathematical scope.\medskip

\begin{remark}~
\begin{itemize}
\item The assumption that 
$Z$ is full-dimensional 
(i.e., has a nonempty interior)
can be made without loss of generality (w.l.o.g.) in \cref{th:main_summary}. 
If $Z$ is nonempty, then 
full-dimensionality can always 
be achieved by reformulating \eqref{eq:EVI_again} 
as an EVI in 
the affine hull of $Z$. 
For the same reason, 
we may assume w.l.o.g.\
that the underlying vector space is $\R^d$.
Indeed,
if we start with a problem in an arbitrary 
Hilbert space with a polyhedral admissible set $Z$, 
then we can always reformulate \eqref{eq:EVI_again} 
as a problem in the 
finite-dimensional linear hull of 
the normal vectors of $Z$ and consider this space instead; see \cite[Remark~4.5]{Krejci1999}.

\item Due to the Lipschitz continuity of 
the function 
$\SS$, the difference quotients 
in \ref{main_sum:item:I} also converge weakly-star 
in $BV([0,T];\R^d)$. 
In particular, 
the map $\SS$ is directionally differentiable 
as a function 
from $CBV([0,T]; \R^d) \times Z$
to $L^p((0,T);\R^d)$ for all $1 \leq p < \infty$ in \ref{main_sum:item:I}.
The Lipschitz continuity of 
$\SS$ also implies 
that $\SS$ is
pointwisely Hadamard directionally 
differentiable in \ref{main_sum:item:I}; see \cite[Proposition 2.49]{BonnansShapiro2000}.
\item Due to the evolutionary nature of 
\eqref{eq:EVI_again},
the nonexistence of a pointwise directional 
derivative in \ref{main_sum:item:II} implies that 
$\SS$ cannot be directionally 
differentiable as a map
from $W^{1,\infty}((0,T);\R^d) \times Z$ to $L^1((0,T);\R^d)$
(or, more generally, 
from $W^{1,\infty}((0,T);\R^d) \times Z$
to any sensible 
function space); see \cref{th:non_dir_diff}.
Note that it could still be the case here that $\SS$ is directionally differentiable 
when restricted to a suitable subset of $W^{1,\infty}((0,T);\R^d) \times Z$, e.g., to
$\{u \in W^{1,\infty}((0,T);\R^d) \mid \dot u \in BV([0,T];\R^d)\} \times Z$; cf.\ \cref{sec:4}.
\end{itemize}
\end{remark}

\subsection{Background and relations to prior work}

Sweeping processes of the type 
\eqref{eq:sweep_problem} have first been 
investigated by Moreau in the nineteen-seventies 
to study problems arising in elastoplasticity; 
see \cite{Moreau2011,Moreau1977,Moreau1973}. 
On the most fundamental level, 
they can be understood as models for the 
dynamical behavior of 
an infinitesimally small body that 
is contained in a hollow shape 
which moves
through the Euclidean space $\R^d$
and which forces the body to move along when 
it comes into contact with the shape's boundary
(similarly to a coin on a table under
an upside-down drinking glass). 
In this interpretation, 
the function $x\colon [0,T] \to \R^d$
in \eqref{eq:sweep_problem} 
describes the position of the body, 
$Z$ is the shape, 
$u\colon [0,T] \to \R^d$
models the motion of $Z$, 
the condition $x(t) \in Z - u(t)$
reflects that the body cannot leave the 
set $Z - u(t)$,
and the inclusion 
$\dot x(t) \in - \NN_{Z - u(t)}(x(t))$
expresses that the body
is ``swept''
into an inward normal direction when 
in contact with the boundary $\partial(Z-u(t))$. 
Since their initial introduction, 
sweeping processes have 
been used to model
phenomena in various application areas
(mechanics \cite{Krasnoselskii2012,Krejci1996}, 
electrical engineering \cite{Acary2010},
and 
crowd control \cite{Cao2022-2} to name a few).
They moreover often serve as
prototypical examples for the class of energetic
rate-independent systems \cite{Mielke2015}.
Due to this prominent role,
problems of the type \eqref{eq:sweep_problem}  have been subject to extensive research 
during the last fifty years; see, e.g., 
\cite{Krejci1999,Moreau2011,Siddiqi2002,Mielke2015,Krasnoselskii2012,Krejci1996,Adam2014,Desch1999,Dupuis1991}
and the references therein. 

Despite the considerable 
attention that sweeping processes have received 
in the past,
the question of whether their
solution mappings are 
directionally 
differentiable or not 
is still severely underinvestigated.
At least to the best of our knowledge, the 
only contributions on this topic currently found in the literature are 
\cite{Brokate2015}, \cite{Brokate2020},
and \cite{Brokate2021}, 
which all deal 
with the EVI \eqref{eq:sweep_EVI_strong}
in the case $d=1$ and $Z = [-r,r]$, $r>0$.
In the first of these papers, 
\cite{Brokate2015},
it is shown that the scalar 
stop operator 
$\SS\colon (u, y_0) \mapsto y$
possesses a pointwise Hadamard 
directional derivative when considered as 
a function from 
$C([0,T]) \times Z $ to $C([0,T])$;
see \cite[Proposition 5.3, Corollary 5.4]{Brokate2015}. 
The proof of this result 
relies heavily on the fact that,
in the one-dimensional setting, 
it is possible to derive a semi-explicit solution
formula for the play operator $\PP$
in terms 
of the cumulated maximum function; see 
\cite[Lemma~5.1]{Brokate2015}.
Based on the same observation, 
Newton- and Bouligand differentiability properties of 
the scalar stop and play  
are investigated in \cite{Brokate2020}. 
The third paper, \cite{Brokate2021},
finally establishes
an auxiliary variational inequality
that uniquely characterizes the 
derivatives of the scalar play and stop 
into directions $(h, h_0) \in CBV([0,T]) \times \R$;
see \cite[Theorem 2.1]{Brokate2021}. 
At least to the best of our knowledge, 
directional differentiability results 
for vectorial sweeping processes like \eqref{eq:sweep_EVI_strong} 
can currently not be found in the literature.
One of the reasons for the absence of
such results is that, 
although simple at first glance, 
the structure of \eqref{eq:sweep_EVI_strong}
makes establishing the (directional) differentiability 
of the mappings $\SS$ and $\PP$ rather
difficult.
The discontinuity of the 
normal-cone mapping in \eqref{eq:sweep_problem},
for example, causes pointwise limits of 
difference quotients
of $\SS$ and $\PP$ to be typically neither left-
nor right-continuous 
and, thus, creates serious regularity problems; 
see
 \cref{sec:3,sec:4} and \cref{ex:left_righ_jump}.

It should be noted that this lack of 
directional differentiability results for solution
operators is not limited to sweeping processes 
but also extends to the larger class of ``true'' 
evolution variational inequalities 
(where with ``true'' we mean that the 
EVI cannot be recast as a PDE by means of a viscous regularization or similar tricks; cf.\ \cite{Betz2019}). 
Even in this wider field,
we are only aware of two additional contributions 
on the directional differentiability properties of 
solution maps. 
The first one, \cite{Jarusek2003}, 
is concerned with the directional differentiability of the 
solution operator of the scalar parabolic Signorini problem. 
The main idea of this paper is to interpret the EVI
at hand as a problem in a fractional Sobolev space 
and to then employ an argumentation that 
has originally been developed in \cite{Haraux1977,Mignot1976} for elliptic obstacle
 problems. The second paper, \cite{Christof2019parob}, 
 is concerned with general parabolic 
 obstacle-type evolution variational inequalities. 
 It establishes by means of pointwise-a.e.\ 
 convexity properties that solution operators 
 of this type of EVI are directionally differentiable as functions 
 into the Lebesgue spaces and demonstrates that the 
 arguments of \cite{Jarusek2003} fail for
 all obstacle-type EVIs that involve pointwise
 constraints on a set of positive measure;
 see \cite[Section 3, Theorem 4.1]{Christof2019parob}.

That so little is currently known about the 
differential sensitivity analysis of solution 
maps of EVIs is 
rather unsatisfying---in particular since,
for elliptic variational inequalities, 
differentiability results 
have been established in a multitude of papers 
and the subject is nowadays 
almost completely understood. 
Compare, e.g., 
with the seminal works 
\cite{Zarantonello1971,Haraux1977,Mignot1976} on this topic,
with \cite{BonnansCominettiShapiro1998,FitzpatrickPhelps1982,Noll1995,Adly2018,Shapiro1994}, and with the 
results in \cite[Theorem~4.3]{Do1992}, \cite[Theorem 1.4.1]{ChristofPhd2018}, and \cite[Theorem 4.1]{ChristofWachsmuth2020}
which characterize precisely
by means of the concept of second-order epi-differentiability 
in which situations the solution map of 
an elliptic variational inequality of the first
or the second kind is directionally differentiable 
and in which situations it is not. 

We would like to emphasize at this point that the question 
of whether directional differentiability results 
can be derived for problems like \eqref{eq:sweep_EVI_strong} is
far from academic. Results on this topic are
of fundamental importance when optimal control problems 
governed by systems involving the 
play and stop operator are 
considered (e.g., PDE-systems involving hysteresis)
as they form the basis for 
the formulation of first-order necessary optimality conditions 
and the design of derivative-based solution algorithms. 
In the field of optimal control of sweeping processes,
the lack of knowledge about the 
differentiability properties of solution maps has
caused essentially all authors concerned with the matter 
to resort to approximation or regularization 
approaches that replace the original problem formulation with 
a sequence of discretized and/or regularized problems. 
Compare, e.g., with 
\cite{Adam2014,Cao2016,Colombo2020,Cao2019,Colombo2016,Henrion2023,Cao2022} in this context, 
which employ discrete approximations,
and 
with \cite{Arroud2018,dePinho2019,Stefanelli2017,Brokate2013,Herzog2022,Nour2022}, which utilize 
various mollification techniques. 
Again, this state of affairs is rather unsatisfactory, 
in particular when compared with the 
field of optimal control of elliptic variational 
inequalities for which 
the available directional differentiability results 
allow to derive so-called 
``strong stationarity''-systems that
are equivalent to the notion of 
Bouligand stationarity (i.e., the condition that the 
directional derivatives of the reduced objective function
are nonnegative in all admissible 
directions) and thus do not suffer from the loss of information that regularization and 
discretization approaches are typically subject to; 
see \cite{MignotPuel1984,Harder2017,Christof2022,Hintermueller2009}. 
That it is possible 
to obtain stationarity systems 
of the same strength for evolution variational 
inequalities once directional differentiability results 
are available has 
recently been demonstrated in \cite{BrokateChristof2023} 
for the one-dimensional play and stop based on the findings of \cite{Brokate2015,Brokate2020,Brokate2021};
see \cite[Theorem 7.1]{BrokateChristof2023}
and \cref{sec:5}. To our best knowledge,
the  result in \cite{BrokateChristof2023} 
is the first on strong stationarity conditions 
in the field of optimal control of true EVIs;
cf.\ the comments in \cite[Section 1.1]{BrokateChristof2023}. 
In particular, strong stationarity conditions 
for problems governed by sweeping processes with vector-valued in- and output functions seem to be 
unknown at the moment. 

The principal aim of the present paper is
to shed some further 
light on the directional differentiability properties of 
solution maps of sweeping processes 
and evolution variational inequalities in general.
By doing so, we also hope to lay the foundation 
for the study of stationarity conditions 
for optimal control problems governed by EVIs 
that do not rely on an approximation or regularization procedure and go beyond the one-dimensional 
setting of \cite{BrokateChristof2023}. 
\Cref{th:main_summary} serves these two goals
because it answers conclusively the question 
under which conditions the vector play and stop are directionally differentiable 
when the admissible set $Z$ is 
a convex polyhedron and 
the input space is chosen as $CBV([0,T];\R^d)$.
At least to the best of our knowledge,
\cref{th:main_summary} is the first result 
on the directional differentiability properties 
of solution maps of vectorial 
sweeping processes (and, more generally, 
vector-valued EVIs). 
It is also of broader interest 
because it highlights some
quite peculiar effects that emerge 
in the differential sensitivity analysis 
of evolution variational inequalities but are 
absent in the study of elliptic problems.  
Some of the main takeaways are the following points.\medskip

\begin{itemize}
\setlength\itemsep{0.2cm}
\item 
\Cref{th:main_summary} shows that, 
even when we restrict the attention 
to polyhedra, it is very well possible that 
some admissible sets behave completely 
differently than others in an EVI like \eqref{eq:sweep_EVI_strong}. 
In particular, whether the solution mapping of an EVI 
is directionally differentiable or not may hinge on geometric 
properties---in the situation of \cref{th:main_summary}
the interior angles of $Z$---that are completely irrelevant for
questions of well-posedness and Lipschitz stability; 
cf.\ \cref{prop:unique_solvability_stop_play,th:CBV-Lipschitz}.

\item \Cref{th:main_summary} demonstrates that 
solution maps of EVIs may fail to be directionally differentiable 
even if the admissible set and the considered input space are 
very well behaved. In particular, directional non-differentiability 
may also occur when the solution mapping of an EVI like \eqref{eq:sweep_EVI_strong} is considered
on the space $H^1((0,T);\R^d)$, which is, to our knowledge, 
the most popular choice for the control space in the 
field of optimal control of rate-independent EVIs;
see \cite[Section 4]{Herzog2022}, \cite[Section 5]{BrokateChristof2023}.

\item The system \eqref{eq:char_sys_1} 
and the fact that the directional derivatives $\delta$
in case \ref{main_sum:item:I} of \Cref{th:main_summary} can be discontinuous 
from both the left and the right at some times $t\in (0,T)$ 
(see \cref{ex:left_righ_jump})
indicate that, for the 
evolution variational inequality \eqref{eq:sweep_EVI_strong}, 
it is in general not possible to write down a ``classical''
auxiliary EVI which characterizes the 
directional derivatives of the solution map $\SS$. 
A unique characterization of the directional derivatives 
is apparently only possible in \ref{main_sum:item:I} by means of 
the  combination of 
the ``very weak'' variational inequality in \eqref{eq:char_sys_1}
and the separate jump conditions
$\delta(t) = \pi_{\ZZ(y(t))}(\delta(t-))$
and
$\delta(t+) = \pi_{V^+(t)}(\delta(t))$.
This is a major difference to both the results 
for the one-dimensional play and stop in \cite{Brokate2021}
(for which the directional derivatives of the solution map
satisfy $\delta(t) \in \{\delta(t-), \delta(t+)\}$ for all $t \in [0,T]$)
and elliptic variational inequalities; see \cite[Theorem 1.4.1]{ChristofPhd2018}
and \cite[Theorem 4.1]{ChristofWachsmuth2020}.\medskip
\end{itemize}

We would like to point out that 
\Cref{th:main_summary} contains the directional differentiability 
and characterization results of
\cite{Brokate2015,Brokate2021} as a special case
(modulo some minor differences in the regularity assumptions).
For more details on this topic, see 
\cref{subsec:3.7}. As the argumentation that we use for 
the derivation of \cref{th:main_summary} is quite different 
from that in \cite{Brokate2015,Brokate2021}, this 
also means that our analysis provides 
 alternative
proofs of the theorems in these two papers. 
Regarding the distinction between the cases \ref{main_sum:item:I} and \ref{main_sum:item:II} 
in \cref{th:main_summary}, 
it should be noted that 
non-obtuse polyhedra also play a special role in the context 
of the question of whether different techniques for
extending the solution mapping $\SS\colon (u, y_0) \mapsto y$
of \eqref{eq:sweep_EVI_strong} 
to the set $BV([0,T];\R^d) \times Z$
give rise to the same operator; see \cite{Krejci2014Compare}.
If this is just a coincidence or has some deeper mathematical 
meaning 
is an open question.

\subsection{Structure of the remainder of the paper}
\label{sec:1.3structure}

We conclude this introduction with an overview of the content and structure 
of the remainder of the paper.
\Cref{sec:preliminaries} is concerned with preliminaries. 
Here, we introduce basic notation and recall
results from the theory of metric projections, 
convex polyhedra, Kurzweil-Stieltjes integration, and 
sweeping processes that are needed for the derivation of 
\cref{th:main_summary}. 
This section also contains the proofs of the unique solvability of \eqref{eq:EVI_again}
and the Lipschitz estimate \eqref{eq:Lipschitz_estimate_42};
see \cref{prop:unique_solvability_stop_play}
and \cref{th:CBV-Lipschitz}. 
As we aim to keep our analysis as self-contained as 
possible,
we are very thorough in our discussion of preliminaries.
\Cref{sec:3} contains the proof of
statement \ref{main_sum:item:I} of \cref{th:main_summary}.
The basic idea of the analysis 
in this section is to consider 
for a fixed choice of the tuple
$(u, y_0)$ and the direction $(h, h_0)$
the set $\DD$ of all $BV([0,T];\R^d)$-weak-star limits of 
difference quotients of the mapping $\SS$
and to then derive in various steps 
more and more information about the elements 
of $\DD$ until it becomes clear that this
set is a singleton---similarly to the approaches for 
elliptic problems in \cite{ChristofPhd2018,ChristofWachsmuth2020}. 
In combination with \eqref{eq:Lipschitz_estimate_42}
and the Helly selection theorem, the directional 
differentiability of $\SS$ then follows from a 
simple contradiction argument. 
A main difficulty in this argumentation
(and also a major difference to the elliptic setting)
is that, in the case of the EVI \eqref{eq:EVI_again}, 
one has to identify precisely how the 
elements of the set $\DD$ jump at their points of discontinuity. 
We solve this problem by exploiting that 
EVIs \eqref{eq:EVI_again} with 
a non-obtuse convex polyhedral admissible set $Z$
can be studied by means of a vector-valued variant of 
Stampacchia's lemma; cf.\ \cite[Theorem~5.8.2]{Attouch2006}
and  \cite[Lemma 6.3]{ChristofQVI2023}. 
For a more detailed overview of 
the proof of statement \ref{main_sum:item:I}, see \cref{subsec:3.1}. 
At the end of \cref{sec:3}, we also demonstrate 
that part \ref{main_sum:item:I} of \cref{th:main_summary}
indeed implies the results on the existence
and characterization of directional derivatives 
of the scalar play and stop in \cite{Brokate2015,Brokate2021}. (Note that,
in the case $d=1$, all convex polyhedra $Z \subset \R^d$ 
are non-obtuse.)
In \cref{sec:4}, we prove statement \ref{main_sum:item:II} of 
\cref{th:main_summary} by explicitly constructing---for each 
convex polyhedron that is not non-obtuse---a situation
in which the difference quotients of $\SS$ diverge. 
\Cref{sec:5} of the paper finally contains some 
remarks on how \cref{th:main_summary} can be used 
in the field of optimal control of sweeping processes. 
Here, we demonstrate how  
Bouligand stationarity conditions 
for optimization problems governed by EVIs 
can be established
based on the directional differentiability 
results in \cref{th:main_summary}
and sketch some ideas on how 
the system \eqref{eq:char_sys_1}
and the analysis of \cref{sec:3}
can be used to derive strong stationarity conditions
for this problem class.

\section{Preliminaries and preparatory results}
\label{sec:preliminaries}

This section contains preliminary results 
that are needed for the proof of \cref{th:main_summary}.
In \cref{subsec:2.1,subsec:2.2},
we begin with basic notation and 
some results on metric projections. 
\Cref{subsec:2.3} is then concerned with 
preliminaries on convex polyhedra and their description. 
The results on non-obtuse and not non-obtuse 
polyhedra in \cref{subsec:2.3} may also be of independent 
interest---in particular as the property 
of non-obtuseness is also crucial in areas other than 
the differential sensitivity analysis
of EVIs; see \cite{Krejci2014Compare}.
In \Cref{subsec:2.4}, we collect preliminaries 
on the theory of Kurzweil-Stieltjes integration.
For a detailed introduction to this type of integral, 
we refer to \cite{Monteiro2019}. 
\Cref{subsec:2.5} finally contains some preliminaries 
on the vector play and stop.
Here, we in particular address the unique solvability part 
of \cref{th:main_summary} and the Lipschitz estimate 
\eqref{eq:Lipschitz_estimate_42}.

\subsection{Basic notation}
\label{subsec:2.1}

In this paper, 
we write $\R$, $\N$, and $\N_0$ for the real numbers, the natural numbers, and the nonnegative integers, respectively.
We further denote norms and inner products 
defined on 
a real vector space $X$ by $\|\cdot\|_X$ and 
$\langle \cdot, \cdot  \rangle_X$, respectively. 
For the topological dual space of a normed space 
$(X, \|\cdot\|_X)$, we use the symbol $X^*$,
and for a dual pairing, the brackets
$\langle \cdot, \cdot\rangle_{X^*, X}$.
A closed ball of radius $r>0$ in a normed space $(X, \|\cdot\|_X)$
centered at $x \in X$ is denoted by $B_r^X(x)$. 
In the special case $X = \R^d$, $d \in \N$,
we write $|\cdot|$ for the Euclidean norm, 
$\langle\cdot, \cdot \rangle$ for the Euclidean scalar product, 
$B_r(x)$ for the ball $B_r^{\R^d}(x)$ defined with respect to (w.r.t.) $|\cdot|$,
and $S^{d-1} := \{x \in \R^d \mid |x| = 1\}$ for the unit sphere.
If $D$ is a subset of a normed space,
then we denote by $\partial D$ the boundary of $D$,
by $\interior(D)$ the interior of $D$, by $\closure(D)$ the closure of $D$, by $\conv(D)$ the convex hull of $D$,
and by $\spann(D)$ the linear hull of $D$.
For the empty set, we use the 
convention $\spann(\emptyset) := \{0\}$.
Analogously, we also interpret sums over empty index sets as zero. 
Recall
that a set $D \subset X$ is called a cone
if $s x \in D$ holds for all $x \in D$ and all $s \in (0,\infty)$.
For a cone $D$ in a normed space $(X, \|\cdot\|_X)$, we define 
$D^\circ := \left \{ z^* \in X^* \mid \langle z^*, z\rangle_{X^*,X} \leq 0 ~\forall z \in D \right \}$
to be the polar cone of $D$. In the case $X = \R^d$, 
we always identify $X^*$ with $X$ in this definition, 
i.e., we use the convention
$D^\circ := \left \{ z^* \in\R^d \mid \langle z^*, z\rangle \leq 0 ~\forall z \in D \right \}$.
If $D$ is a closed convex nonempty subset of a 
normed space $(X, \|\cdot\|_X)$ and $x \in D$, then we denote by 
$K_\rad(x;D) := \{ z \in X \mid z = s(y - x), s > 0, y \in D\}$
the radial cone of $D$ at $x$,
by 
$K_\tan(x; D):= \closure(K_\rad(x;D))$ the tangent cone of 
$D$ at $x$, and 
by 
$\NN_D(x) := K_\tan(x; D)^\circ$ the (outward)
normal cone of $D$ at $x$. If 
$V$ is a closed subspace of a Hilbert space, 
then we write $V^\perp$ for the orthogonal 
complement of $V$. 
For the cardinality of a set $D$, we use the symbol $|D|$.
The modes of strong and weak convergence 
in a normed space are denoted by $\to$ and $\weakly$,
respectively. If $F\colon X \to Y$ is a mapping 
and $ \emptyset \neq D \subset X$,
then we write $F|_D$ for the restriction of $F$ to $D$
and $\mathds{1}_D\colon X \to \{0,1\}$
for the $\{0,1\}$-indicator 
function of $D$.
For set-valued mappings, we use the notation 
$F\colon X \rightrightarrows Y$.
Directional, G\^{a}teaux, and Fr\'{e}chet derivatives 
are denoted by a prime in this paper. 
If  $F$ is multivariate, 
then the symbol $\partial_k$
indicates that we take 
the partial Fr\'{e}chet derivative 
w.r.t.\ the $k$-th component of the argument of $F$. 
For weak/strong derivatives that are taken w.r.t.\ a time-variable, 
we also use a dot instead of a prime. 
Gradients are denoted by $\nabla$.

If $-\infty < a < b < \infty$ and $d \in \N$ are given,
then we use the standard symbols
$L^{p}((a,b);\R^d)$, $H^1((a,b);\R^d)$,
and $W^{1,p}((a,b);\R^d)$, $1 \leq p \leq \infty$,
for the vectorial Lebesgue and first-order Sobolev spaces,
respectively, endowed with their canonical norms,
i.e.,\vspace{-0.1cm}
\[
\|v\|_{L^{p}((a,b);\R^d)}
:=
\left (
\int_a^b |v(t)|^p  \dd t
\right )^{1/p}
\quad \forall 1 \leq p < \infty,
\qquad
\|v\|_{L^{\infty}((a,b);\R^d)}
:=
\esssup_{t \in [a,b]}
|v(t)|,
\qquad 
\text{ etc.}
\]
The space of smooth, 
real-valued functions on $\R$ is denoted by $C^\infty(\R)$
and the space of 
all functions $v \in C^\infty(\R)$
that possess a compact support  $\supp(v)$ by 
$C_c^\infty(\R)$. 
We further write $C^\infty([a,b]; \R^d)$ for
the space of all functions 
$v\colon [a,b] \to \R^d$ whose components 
can be written as restrictions of 
elements of $C^\infty(\R)$ to the interval $[a,b]$. 
For the space of all
regulated functions $v\colon [a,b] \to \R^d$,
we use the notation $G([a,b]; \R^d)$.
Recall that a function is called regulated 
if it is the uniform limit of a sequence of step functions
or, equivalently, if it possesses 
left and right limits at all 
points $t \in [a,b]$;
see \cite[Definition 4.1.1, Theorem 4.1.5]{Monteiro2019}.
The space  $G([a,b]; \R^d)$
is a Banach space when endowed with the supremum norm 
$
\|v\|_\infty
:= 
\|v\|_{G([a,b]; \R^d)}
:=
\sup_{t \in [a,b]} |v(t)|
$; see \cite[Theorem 4.2.1]{Monteiro2019}.
When we want to emphasize which interval we refer to, we also write $\|\cdot\|_{\infty, [a,b]}$
instead of $\|\cdot\|_{\infty}$. 
For the left and right limit 
of a function $v \in G([a,b]; \R^d)$ at $t \in [a,b]$,
we use the notation $v(t-)$ and $v(t+)$, respectively, 
with the usual conventions at the interval endpoints $a$ and $b$, i.e.,\vspace{-0.1cm}
\[
v(t-) :=
\lim_{(a,t) \ni s \to t} v(s)
~~
\forall t \in (a,b],
\qquad
v(a-) := v(a),
\qquad 
v(t+) :=
\lim_{(t,b) \ni s \to t} v(s)
~~
\forall t \in [a,b),
\qquad
v(b+) := v(b).
\]
We further define $v_\pm(t) := v(t\pm)$
for all $t \in [a,b]$.
If the underlying interval is important,
then we highlight this by writing 
$v_{-,[a,b]}$ and $v_{+,[a,b]}$. 
Note that, for $a < s_1 < s_2 < b$,
one typically has 
$v_{-,[s_1,s_2]} \neq (v_{-,[a,b]})|_{[s_1,s_2]}$
and
$v_{+,[s_1,s_2]} \neq (v_{+,[a,b]})|_{[s_1,s_2]}$
due to the different conventions at the 
interval endpoints. 
Recall that
the set of points of discontinuity 
$[a,b] \setminus \{t \in [a,b] \mid v(t) = v(t+) = v(t-)\}$
of a regulated function $v\colon [a,b] \to \R^d$
is at most countable;
see \cite[Theorem 4.1.8]{Monteiro2019}. 
For every 
$v \in G([a,b]; \R^d)$, 
we further have $v_-, v_+ \in G([a,b]; \R^d)$ and 
\begin{equation}
\label{eq:Monteiro4.1.9}
\begin{aligned}
&v_-(t-) = v(t-)
~\forall t \in [a,b],
\qquad
&&
v_-(t+) = v(t+)
~\forall t \in [a,b),
\\
&v_+(t-) = v(t-)
~\forall t \in (a,b],
\qquad 
&&v_+(t+) = v(t+)
~\forall t \in [a,b];
\end{aligned}
\end{equation}
see \cite[Corollary 4.1.9]{Monteiro2019}.
We define 
$G_{rl}([a,b];\R^d) := \{v \in G([a,b];\R^d) \mid 
v(t) \in \{v(t-),v(t+)\}~ \forall t \in [a,b]\}$,
$G_r([a,b];\R^d) := \{v \in G([a,b];\R^d) \mid v = v_+\}$,
and 
$
G_l([a,b];\R^d) := \{v \in G([a,b];\R^d) \mid v = v_-\}$.
Note that 
$G_r([a,b];\R^d)$
and
$G_l([a,b];\R^d)$
 are  
closed subspaces of $G([a,b];\R^d)$
and, thus, Banach.
The same is true for the space 
$C([a,b];\R^d)$ of continuous functions 
$v\colon [a,b] \to \R^d$. 
For arbitrary $v\colon [a,b] \to \R^d$,
 we denote by 
\begin{align*}
\var(v; [a,b]) 
:=
\sup_{a = t_0 \leq t_1 \leq  ... \leq t_N = b, N \in \N}
\sum_{j=1,...,N} |v(t_j) - v(t_{j - 1})| \in [0,\infty]
\end{align*}
the variation of $v$ on $[a,b]$. 
Here, the supremum 
is taken over all partitions
$a = t_0 \leq ... \leq t_N = b$, $N \in \N$, of the interval $[a,b]$.
Note that this definition implies  
$\var(v;[a,b]) = \|\dot v \|_{L^1((a,b);\R^d)}$
for all $v \in W^{1,1}((a,b);\R^d)$. 
The space of all functions
$v\colon [a,b] \to \R^d$ with finite variation
is denoted by $BV([a,b];\R^d)$. Recall 
that $BV([a,b];\R^d)$ is a Banach space
when equipped with the norm 
$\|v\|_{BV([a,b];\R^d)} := |v(a)| + \var(v; [a,b]) $
and a subspace of $G([a,b]; \R^d)$;
see \cite[Theorem 2.2.2]{Monteiro2019}. 
Further, it holds $v_+, v_- \in BV([a,b];\R^d)$
for all $v \in BV([a,b];\R^d)$. We define 
$BV_{rl}([a,b];\R^d) := BV([a,b];\R^d) \cap G_{rl}([a,b];\R^d)$,
$BV_r([a,b];\R^d) :=  BV([a,b];\R^d) \cap G_{r}([a,b];\R^d)$,
$BV_l([a,b];\R^d) :=  BV([a,b];\R^d) \cap G_{l}([a,b];\R^d)$,
and  
$CBV([a,b];\R^d) := C([a,b];\R^d) \cap BV([a,b];\R^d)$.
Note that the latter three of these sets
are  closed subspaces of
$BV([a,b];\R^d)$ and, thus, Banach spaces.
In what follows, 
we endow the space 
$CBV([a,b];\R^d)$ with the 
norm
$
	\|v \|_{CBV([a,b];\R^d)}
	:=
	\|v \|_{C([a,b];\R^d)}
	+
	\var(v; [a,b]) 
        =
        \|v\|_{\infty}
        +
        \var(v; [a,b]),
$
which is equivalent to $\|\cdot \|_{ BV([a,b];\R^d)}$. 
If $K\colon [a,b] \rightrightarrows \R^d$ is a set-valued function, then we
define
$G([a,b]; K) := \{v \in G([a,b]; \R^d) \mid v(t) \in K(t)
\text{ for all }t \in [a,b]\}$.
Analogous shorthand notations are also 
used for $C([a,b]; \R^d)$, $BV([a,b];\R^d)$, etc.
Sets $K \subset \R^d$ are interpreted as constant 
functions in this context. 
If $d=1$ holds, then we simply 
write $G([a,b])$, $L^p(a,b)$, etc.\
instead of 
$G([a,b]; \R^d)$, $L^p((a,b);\R^d)$,  etc. 

\subsection{Preliminaries on projections}
\label{subsec:2.2}
In what follows, we recall basic results on 
Euclidean projections and distance functions that are needed for the derivation of \cref{th:main_summary}.

\begin{definition}[projection and distance function]
Let $Z \subset \R^d$, $d \in \N$, be  a nonempty, convex, and closed set.

\begin{enumerate}[label=\roman*)]
\item  We denote by $\pi_Z$
the Euclidean projection onto $Z$, i.e., 
\begin{equation}
\label{eq:proj_def_problem}
\pi_Z\colon \R^d \to \R^d,
\qquad 
\pi_Z(x) :=
\argmin \left \{
\left.
\frac12  |v - x|^2
\;
\right |
\;
  v \in Z
\right \}.
\end{equation}

\item  
We denote by $\dist_Z$ the Euclidean distance function to $Z$, i.e., 
\[
\dist_Z\colon \R^d \to [0, \infty),
\qquad
\dist_Z(x) := | \pi_Z(x) - x|  = \min_{v \in Z}  |v - x|.
\]
\end{enumerate}
\end{definition}

\begin{proposition}[properties of 
$\pi_Z$ and $\dist_Z$]
\label{lem:basic_proj_properties}
Let $Z \subset \R^d$, $d \in \N$, be nonempty, convex, and closed. 
\begin{enumerate}[label=\roman*)]
\item\label{lem:basic_proj_properties:i} For every $x \in \R^d$, the projection $\pi_Z(x)$
is uniquely characterized by the variational inequality 
\begin{equation}
\label{eq:proj_VI}
\pi_Z(x) \in Z,\qquad \langle \pi_Z(x) - x, v - \pi_Z(x)\rangle \geq 0\quad \forall v \in Z.
\end{equation}
\item\label{lem:basic_proj_properties:ii} 
The projection $\pi_Z\colon \R^d \to \R^d$ is one-Lipschitz, i.e., 
\begin{equation}
\label{eq:1projLip}
| \pi_Z(x_1) - \pi_Z(x_2) | \leq |x_1 - x_2|\qquad \forall x_1, x_2 \in \R^d.
\end{equation}
\item\label{lem:basic_proj_properties:iii}
The function $F\colon \R^d \to \R$, $x \mapsto \dist_Z(x)^2$,
is continuously Fr\'{e}chet differentiable. Its
gradient 
$\nabla F(x)$ w.r.t.\ the Euclidean scalar product 
 at $x \in \R^d$ is given by 
\begin{equation}
\label{eq:F_prime_formula}
\nabla F(x) = 2(x - \pi_Z(x)). 
\end{equation}
\item\label{lem:basic_proj_properties:iv}
If $Z$ is additionally a cone, then it holds
\begin{equation}
\label{eq:Z-Zcirc-identity+hom}
\begin{gathered}
x = \pi_Z(x) + \pi_{Z^\circ}(x)
\quad \text{and}\quad
\langle \pi_Z(x), \pi_{Z^\circ}(x) \rangle = 0
\quad \forall x \in \R^d,
\\
\quad 
\pi_Z(s x) = s \pi_Z(x)
\quad \forall x \in \R^d\quad \forall s \in [0, \infty),
\end{gathered}
\end{equation}
and the function $F\colon \R^d \to \R$, $x \mapsto \dist_Z(x)^2$, in \ref{lem:basic_proj_properties:iii} satisfies
\begin{equation}
\label{eq:F_prime_cone}
\nabla F(x) = 2 \pi_{Z^\circ}(x) \quad \forall x \in \R^d.
\end{equation}
\end{enumerate}
\end{proposition}

\begin{proof}
All of the presented facts are well known. 
We include proofs of \ref{lem:basic_proj_properties:iii} 
and \ref{lem:basic_proj_properties:iv} for the convenience of the reader.
We start with \ref{lem:basic_proj_properties:iii}: 
Suppose that $x,h \in \R^d$ are given. Due to \eqref{eq:proj_VI}, we have 
\begin{equation}
\label{eq:randomeq73636}
\begin{aligned}
\left \langle \pi_Z(x + \tau h) - x - \tau h, \frac{\pi_Z(x) - \pi_Z(x + \tau h)}{\tau}\right \rangle  \geq 0
\quad \text{and}\quad
\left \langle \pi_Z(x ) - x  , \frac{\pi_Z(x + \tau h) - \pi_Z(x )}{\tau}\right \rangle \geq 0
\end{aligned}
\end{equation}
for all $\tau > 0$. From \eqref{eq:1projLip}, it follows further that
the difference quotients
\[
\delta_\tau := \frac{\pi_Z(x + \tau h) - \pi_Z(x)}{\tau}
\]
are bounded by $|h|$ for all $\tau > 0$.
By compactness, this implies $\delta_k := \delta_{\tau_k} \to \delta$
for some sequence $0 < \tau_k \to 0$ and some $\delta \in \R^d$.
By passing to the limit $k \to \infty$ along $\{\tau_k\}$ in 
\eqref{eq:randomeq73636}, we obtain that 
$
0 \leq \langle \pi_Z(x ) - x  , \delta \rangle \leq 0
$, 
i.e., $\langle \pi_Z(x ) - x  , \delta \rangle = 0$. Due to the definition of $\dist_Z$
and the identity $|a|^2 - |b|^2 = \langle a + b, a - b \rangle$ for $a,b \in \R^d$,
we furthermore have 
\begin{equation*}
\begin{aligned}
\dist_Z(x + \tau h)^2 - \dist_Z(x)^2
&=
|x + \tau h - \pi_Z(x + \tau h) |^2 - |x - \pi_Z(x) |^2
\\
&=
\left \langle
x + \tau h - \pi_Z(x + \tau h) + x - \pi_Z(x) 
,
x + \tau h - \pi_Z(x + \tau h) - x + \pi_Z(x) 
\right \rangle
\\
&=
\left \langle
x + \tau h - \pi_Z(x + \tau h) + x - \pi_Z(x) 
,
\tau h - \pi_Z(x + \tau h) + \pi_Z(x) 
\right \rangle\qquad \forall \tau > 0.
\end{aligned}
\end{equation*}
By dividing by $\tau > 0$ and passing to the limit $k \to \infty$ along 
$\{\tau_k\}$ in the above, we obtain that 
\begin{equation}
\label{eq:randomeq37374}
\lim_{k \to \infty}
\frac{\dist_Z(x + \tau_k h)^2 - \dist_Z(x)^2}{\tau_k}
=
2 \left \langle
x - \pi_Z(x) 
,
h - \delta
\right \rangle
=
2 \left \langle
x - \pi_Z(x) 
,
h 
\right \rangle.
\end{equation}
Since the right-hand side of \eqref{eq:randomeq37374} does not depend on the choice of
$\{\tau_k\}$, 
it follows 
that $\R^d \ni x \mapsto \dist_Z(x)^2 \in \R$ is G\^{a}teaux differentiable.
As the G\^{a}teaux derivative depends continuously on $x$ by the continuity of $\pi_Z$,
the continuous Fr\'{e}chet differentiability now  
follows from the mean value theorem. This completes the proof of \ref{lem:basic_proj_properties:iii}.
To establish \ref{lem:basic_proj_properties:iv},
suppose that $x \in \R^d$ is given. 
From the cone property of $Z$ and \eqref{eq:proj_VI},
it follows that 
\begin{equation}
\label{eq:randomeq_z_Zcirc}
\langle \pi_Z(x) - x, \pi_Z(x)\rangle  = 0 \leq \langle \pi_Z(x) - x, v\rangle \quad \forall v \in Z.
\end{equation}
The above implies $x - \pi_{Z}(x) \in Z^\circ$.
For all $w \in Z^\circ$, we further have 
$\langle \pi_Z(x), w \rangle  \leq 0$ and, thus,
\[
0 \geq
\langle \pi_Z(x), w \rangle 
+
\langle \pi_Z(x) - x, \pi_Z(x)\rangle 
=
\left \langle
\pi_Z(x), w + \pi_Z(x) - x
\right \rangle \qquad \forall w \in Z^\circ.
\]
Rewriting this inequality gives
\[
x - \pi_{Z}(x) \in Z^\circ,
\qquad
\left \langle
(x - \pi_Z(x)) - x, w - (x - \pi_Z(x))
\right \rangle
\geq 0\qquad \forall w \in Z^\circ.
\]
Thus, $x - \pi_Z(x) = \pi_{Z^\circ}(x)$ 
holds as claimed. 
To see that $\pi_Z(s x) = s \pi_Z(x)$
holds for all $x \in \R^d$ and $s \in [0, \infty)$, 
it suffices to multiply \eqref{eq:proj_VI}
by $s^2$ and to exploit the cone property and 
\ref{lem:basic_proj_properties:i}. 
This establishes \eqref{eq:Z-Zcirc-identity+hom}.
As \eqref{eq:F_prime_cone}
is a straightforward consequence of
\eqref{eq:F_prime_formula} and \eqref{eq:Z-Zcirc-identity+hom},
\ref{lem:basic_proj_properties:iv} follows and the proof 
of the proposition is complete.
\end{proof}

\subsection{Preliminaries on convex polyhedra}
\label{subsec:2.3}

Convex polyhedra in $\R^d$, $d \in \N$, can be constructed as intersections of half-spaces of the form 
$
H = \{x \in \R^d \mid \langle \nu, x \rangle \le \alpha \},
$
where $(\nu,\alpha) \in \R^d \times \R$ is a tuple with $|\nu| = 1$, i.e., $\nu \in S^{d-1}$.

\begin{definition}[convex polyhedron]
\label{def:convex_polyhedron}
A set $Z\subset \R^d$, $d \in \N$, is called a convex polyhedron if 
$Z = \R^d$ holds or if 
there exist an index set $I=\{1,...,n\}$, $n \in \N$,
and tuples $(\nu_i, \alpha_i)  \in S^{d-1}\times \R$, $i \in I$, 
such that $Z$ satisfies
\[
Z = \bigcap_{i\in I} H_i, \qquad 
H_i := \left \{ x \in \R^d \mid \langle \nu_i, x\rangle \leq \alpha_i \right \}.
\]
In the latter case, we call the $n$-tuple $\{(\nu_i, \alpha_i)\}_{i \in I} \in (S^{d-1} \times \R)^n$
a description of $Z$.
A polyhedron $Z$ is called full-dimensional if it has a nonempty interior.
\end{definition}

\begin{remark}[conventions for the case $Z = \R^d$]
In what follows, 
some care has to be taken regarding the edge case $Z = \R^d$
in \cref{def:convex_polyhedron}.
Henceforth, 
we use the convention that the polyhedron 
$Z = \R^d$ possesses precisely one description, namely 
the $0$-tuple (a.k.a.\ the empty tuple). 
We further define the index set $I$ belonging to this description 
to be the empty set and we use the convention that 
$\{(\nu_i, \alpha_i)\}_{i \in I}$
denotes the $0$-tuple whenever $I =\emptyset$. 
Note that this means in particular that $Z = \R^d$ possesses one and only one description. 
\end{remark}

Note that every convex polyhedron $Z \neq \{0\}$ in $\R^d$, $d \in \N$, containing the origin
is a full-dimensional polyhedron in the sense of \cref{def:convex_polyhedron} 
when considered as a subset of its linear hull. The same is true for 
the case $Z = \{0\}$ if $\{0\}$ is understood as a full-dimensional polyhedron in $\R^0 := \{0\}$.

\begin{definition}[set of active/inactive indices, regular description]
Given a convex polyhedron $Z \subset \R^d$
with a description $\{(\nu_i, \alpha_i)\}_{i \in I}$
and $x \in Z$, we denote by $\AA(x) := \{i \in I \mid \langle \nu_i, x\rangle =  \alpha_i\}$
the set of active indices of  $\{(\nu_i, \alpha_i)\}_{i \in I}$ at $x$
and by $\II(x) := \{i \in I \mid \langle \nu_i, x\rangle <  \alpha_i\}$ the set of inactive indices 
of  $\{(\nu_i, \alpha_i)\}_{i \in I}$ at $x$.
A description is called regular if, for every $x \in Z$,
the vectors in the set $\{ \nu_i \mid i \in \AA(x)\}$ are linearly independent.
\end{definition}

The octahedron in $\R^3$ 
(or, more generally, the closed unit ball w.r.t.\ the $\ell^1$-norm in $\R^d$, $d \ge 3$)
is an example of a polyhedron that does not possess a regular description
since four facets meet at corner points.

Descriptions are typically not unique.
As they are defined as tuples, we can reorder a given description, for example, 
to generate a new one with different maps $\AA, \II\colon Z \rightrightarrows \mathbb{N}$.
Further, a description of a polyhedron $Z$ may contain tuples $(\nu_j,\alpha_j)$
that are redundant in the sense that their removal does not change $Z$, i.e.,
\[
Z = 
\bigcap_{i\in I}  
\{ x \in \R^d \mid  \langle \nu_i, x\rangle \leq \alpha_i \}
=
\bigcap_{i\in I, i \neq j}  
\{ x \in \R^d \mid  \langle \nu_i, x\rangle \leq \alpha_i \}.
\]
The following well-known result shows that, for a full-dimensional convex polyhedron, 
it is possible to overcome these uniqueness and redundancy issues 
(at least in a certain sense) by considering \emph{standard descriptions}.

\begin{proposition}[standard descriptions]
\label{lem:standard_description}
Let $Z \subset \R^d$ be a full-dimensional convex polyhedron.
Then there exists a
description $\HH = \{(\nu_i, \alpha_i)\}_{i \in I}$ of $Z$ that contains no redundant
half-spaces.  
This description is unique up to permutations of the tuples $(\nu_i,\alpha_i)$
and satisfies 
\begin{equation}
\label{eq:faces}
\forall i \in I~\exists x_i \in Z\colon ~ \quad
 \AA(x_i) = \{i\}\quad (\text{that is,}~~
 \langle \nu_i, x_i\rangle = \alpha_i \text{ and }
  \langle \nu_j, x_i\rangle < \alpha_j ~\forall j \neq i).
\end{equation}
Henceforth, we call a description $\HH$ with the above properties a \emph{standard description} of the polyhedron $Z$. 
\end{proposition}
\begin{proof} 
This is a standard result from linear programming and polyhedral theory;
see
\cite[Theorem 1.6]{Bruns2009}
or
\cite[Theorem 12.1.5]{Berger2009}.
Note that the $0$-tuple is the unique standard description of $Z = \R^d$ by our conventions.  
\end{proof}

\begin{corollary}[standard descriptions of cones]
\label{lem:standard_description_cone}
Let $Z \subset \R^d$ be a full-dimensional convex polyhedral cone
and let $\HH$ be a standard description of $Z$. Then $\HH$ has the form $\HH = \{(\nu_i, 0)\}_{i \in I}$.
\end{corollary}
\begin{proof}
Let $\HH = \{(\nu_i, \alpha_i)\}_{i \in I}$ be a standard description of $Z$ and let 
$i \in I$ be arbitrary. Since $0 \in Z$,
we have $\alpha_i \ge \langle \nu_i, 0 \rangle = 0$.
For $x_i$ from \eqref{eq:faces}, it further holds
$\langle \nu_i, 2 x_i \rangle = 2\alpha_i \le \alpha_i$ as $2 x_i\in Z$.
Thus $\alpha_i = 0$.
\end{proof}

Next, we define the \emph{linearization cone}, which provides a local approximation for a convex polyhedron.

\begin{definition}[linearization cone]
\label{def:lin_cone}
Let $Z \subset \R^d$
be a nonempty convex polyhedron with description $\{(\nu_i, \alpha_i)\}_{i \in I}$.
Then, for every $x \in Z$, we define 
$
\ZZ(x) := \{z \in \R^d \mid \langle \nu_i, z\rangle \leq 0~\forall i \in \AA(x) \}.
$
\end{definition}

Note that we have $Z - x \subset \ZZ(x)$ for all $x \in Z$ in \cref{{def:lin_cone}} since
$\langle \nu_j, y - x \rangle = \langle \nu_j, y \rangle - \alpha_j \le 0$
holds for all $y \in Z$ and all $j\in \AA(x)$. It will be shown immediately below in \cref{lem:ZZproperties}\ref{lem:ZZproperties:iii}
that the cone $\ZZ(x)$ is independent of the chosen description of $Z$.

\begin{lemma}[full-dimensionality and standard descriptions of the linearization cone]
\label{lem:ZZstandard}
Let $Z \subset \R^d$ be a full-dimensional convex polyhedron.
Let $\HH = \{(\nu_i, \alpha_i)\}_{i \in I}$ be a standard description of $Z$
and let $x \in Z$ be given.
Then $\ZZ(x)$ is full-dimensional and
$\{(\nu_i, 0)\}_{i \in \AA(x)}$ is a standard description of $\ZZ(x)$.
\end{lemma}

\begin{proof}
Since $Z$ is full-dimensional, so is $\ZZ(x) \supset Z - x$.
It remains to prove the statement about the standard description. 
To this end, we assume w.l.o.g.\ that $0 \in \interior(Z)$.
(This can always be achieved by a translation.) 
Note that this inclusion implies $0 = \langle \nu_i,  0 \rangle < \alpha_i$
for all $i \in I$.
Consider now some $i \in \AA(x)$. We have to show that the half-space
$H_i = \{z \in \R^d \mid \langle \nu_i, z \rangle \leq 0 \}$
is not redundant for $\ZZ(x)$.
We choose $x_i \in Z$ according to \eqref{eq:faces} 
and set $z := (1 + \veps) x_i - x$ with $\veps > 0$.
Then $z \notin H_i$ since $\langle \nu_i, z \rangle = \veps \alpha_i > 0$.
On the other hand, for $j \in \AA(x) \setminus \{i\}$ we have
$\langle \nu_j, x_i - x \rangle < 0$, so
$\langle \nu_j, z\rangle = 
\veps \langle \nu_j, x_i \rangle + \langle \nu_j, x_i - x \rangle \le 0$ 
if $\veps$ is small enough. This shows that $H_i$ is indeed 
nonredundant for $\ZZ(x)$ and completes the proof. 
\end{proof}

\begin{lemma}[properties of the linearization cone]
\label{lem:ZZproperties}
Let $Z \subset \R^d$
be a nonempty convex polyhedron with a description $\{(\nu_i, \alpha_i)\}_{i \in I}$.
Then the following statements are
true.
\begin{enumerate}[label=\roman*)]
\item\label{lem:ZZproperties:i} 
For every $x \in Z$, there exists $\varepsilon > 0$ such that 
\[
w + s z \in Z\qquad \forall w \in B_\varepsilon (x) \cap Z\quad 
\forall s \in [0, \varepsilon]\quad  \forall z \in B_1 (0) \cap \ZZ(x). 
\]
\item\label{lem:ZZproperties:ii}
For every $x \in \R^d$,
there exists $\varepsilon > 0$ such that 
\[
\pi_Z(x + z) = \pi_Z(x) + \pi_{\ZZ(\pi_Z(x)) }(x + z - \pi_Z(x))
\qquad \forall z \in B_\varepsilon(0).
\]

\item\label{lem:ZZproperties:iii} 
For every $x \in Z$, it holds $\ZZ(x) = K_\rad(x;Z) = K_\tan(x;Z)$.
In particular, if $Z$ is a nonempty convex polyhedral cone, then 
 $\ZZ(0) = K_{rad}(0;Z) = Z$.
\item\label{lem:ZZproperties:iv} If
$\{(\nu_i, \alpha_i)\}_{i \in I}$ 
is regular, then, for every point $x \in Z$, there 
exists a unique collection of vectors
$\{e_i\}_ {i \in \AA(x)}\subset \ZZ(x)$ (empty in the case $\AA(x)  = \emptyset$) satisfying 
\begin{equation}
\label{eq:e_i-properties}
	e_i \in \spann(\{\nu_k \mid k \in \AA(x)\}),
	\quad
	|e_i| = 1,\quad\text{and}\quad\langle \nu_i, e_i \rangle < 0 = \langle \nu_j, e_i\rangle
	\quad \forall i, j \in \AA(x), i \neq j.
\end{equation}
The vectors $e_i$, $i \in \AA(x)$,  form a basis of $\spann(\{\nu_k \mid k \in \AA(x)\})$,
and it holds 
\begin{equation}
\label{eq:e_i-equivalence}
\quad\qquad z \in \ZZ(x)
~~\Leftrightarrow~~
z = z_1 + z_2 \text{ with }
z_1 \in \spann(\{\nu_i\mid i \in \AA(x)\})^\perp,
z_2 = \sum_{i \in \AA(x)} \beta_i e_i,~\beta_i \geq 0 ~\forall i \in \AA(x).
\end{equation}
\end{enumerate}
\end{lemma}
\begin{proof}
We begin with \ref{lem:ZZproperties:i}:
Let $x\in Z$.
As the functions $F_i\colon \R^d \times \R^d \to \R$,  
$(w, z) \mapsto \langle \nu_i, w + z\rangle$, $i \in I$, are continuous 
and satisfy $F_i(x, 0) < \alpha_i$ for all $i \in \II(x)$, we can find 
$\varepsilon > 0$ such that $F_i(w, z) < \alpha_i$ holds for all
$w \in B_\varepsilon(x)$, $z \in B_\varepsilon (0)$, and $i \in \II(x)$.
This implies that, for all $i \in \II(x)$, we have   
$\langle \nu_i, w + s z\rangle \leq \alpha_i$ for all
$w \in B_\varepsilon(x) \cap Z$, $s \in [0, \varepsilon]$, and $z \in B_1(0) \cap \ZZ(x)$.
If, on the other hand, $i \in \AA(x)$, then 
we have $\langle \nu_i, w + s z\rangle = \langle \nu_i, w \rangle + s \langle \nu_i, z\rangle \leq \alpha_i + 0$
for all $w \in B_\varepsilon(x) \cap Z$, $s \in [0, \varepsilon]$, and $z \in B_1(0) \cap \ZZ(x)$
by the definitions of $Z$ and $\ZZ(x)$. This establishes \ref{lem:ZZproperties:i}. 

To prove \ref{lem:ZZproperties:ii},
suppose that $x \in \R^d$ is given.
From the Lipschitz continuity of 
$\pi_Z$,
we obtain that there exists $\varepsilon > 0$ such that 
$\AA(\pi_Z(x + z)) \subset \AA(\pi_Z(x))$ holds for all $z \in B_\varepsilon(0)$. 
In combination with \eqref{eq:proj_def_problem}, this yields
\[
\pi_Z(x + z)
= \argmin \left \{\left.
 \frac12 |v - x - z|^2
~\right |~ \langle \nu_i, v \rangle \leq \alpha_i ~\forall i \in \AA(\pi_Z(x))
\right \}\qquad\forall z \in B_\varepsilon(0),
\]
which, by means of the variable transformation $w = v - \pi_Z(x)$, 
can be rewritten as 
\begin{equation*}
\begin{aligned}
\pi_Z(x + z) 
&=  \pi_Z(x) +
\argmin \left \{
\left.
  \frac12 |w + \pi_Z(x) - x - z|^2
~\right |~ \langle \nu_i, w \rangle \leq 0 ~\forall i \in \AA(\pi_Z(x))
\right \}
\\
& =
\pi_Z(x) +
\pi_{\ZZ(\pi_Z(x)) }(x+ z - \pi_Z(x))\qquad \forall z \in B_\varepsilon(0).
\end{aligned}
\end{equation*}
This establishes \ref{lem:ZZproperties:ii}.

Next, we prove \ref{lem:ZZproperties:iii}:
Let $x \in Z$ be given.
Since $\ZZ(x)$ is closed, it suffices to prove that $\ZZ(x) = K_\rad(x;Z)$.
If $z = s(y - x) \in K_\rad(x;Z)$ with $y \in Z$ and $s > 0$, then
$\langle \nu_i, z \rangle = s (\langle \nu_i, y \rangle - \alpha_i) \le 0$
for all $i \in \AA(x)$.
Conversely, if $0 \neq z \in \ZZ(x)$, then from \ref{lem:ZZproperties:i}
we obtain that $x + sz = x + s |z|(z/|z|) \in Z$ if $s > 0$ is small enough.

It remains to prove \ref{lem:ZZproperties:iv}: 
Let $x \in Z$ be given. We assume w.l.o.g.\ that $ \AA(x) \neq \emptyset$. 
(For the degenerate case $\AA(x) = \emptyset$, the assertion is true 
due to the conventions $\spann(\emptyset) = \{0\}$, $\sum_{i \in \emptyset} \beta_i e_i = 0$, etc.)
Fix $i \in \AA(x)$. We set
$V := \spann(\{\nu_j\mid j \in \AA(x)\})$ and
$W := \spann(\{\nu_j \mid j \in \AA(x) \setminus \{i\}\})$.
Since the vectors $\nu_j$, $j \in \AA(x)$, are linearly independent
by our assumptions, we have $\dimm(W) = \dimm(V) - 1$ and
$0 \neq e:= \nu_i - \pi_W(\nu_i) \in W^\bot \cap V$.
Moreover, $\dimm(W^\bot) = d + 1 - \dimm(V)$
and $\dimm(W^\bot + V) = d$ since $W^\bot \supset V^\bot$.
It follows that $\dimm(W^\bot \cap V) = 1$ and 
$\langle \nu_i, e \rangle =  \langle e, e \rangle \neq 0$. 
Thus, $W^\bot \cap V$ contains precisely one vector $e_i$---a scalar multiple 
of $e$---satisfying \eqref{eq:e_i-properties}.
By repeating this construction for all other  indices in $\AA(x)$, the 
uniqueness and existence of vectors $e_i$, $i \in \AA(x)$,
with the properties in \eqref{eq:e_i-properties} follow. 
That the vectors $e_i$, $i \in \AA(x)$, are linearly independent and, consequently, form a basis 
of $V$
follows straightforwardly 
from \eqref{eq:e_i-properties}. It remains to prove \eqref{eq:e_i-equivalence}. 
The implication ``$\Leftarrow$'' in this equivalence 
is again a direct consequence of \eqref{eq:e_i-properties}.
To prove ``$\Rightarrow$'', let $z \in \ZZ(x)$ be given. 
Then there exist unique $z_1 \in  V^\perp$
and $z_2 \in V$ satisfying 
$z = z_1 + z_2$. Since the vectors $e_i$, $i \in \AA(x)$, 
form a basis of $V$, there exist unique $\beta_i \in \R$,
$i \in \AA(x)$, satisfying $z_2 = \sum_{i \in \AA(x)} \beta_i e_i$.
From $z \in \ZZ(x)$ and \eqref{eq:e_i-properties}, it follows that 
$0 \geq \langle \nu_j, z\rangle = \langle \nu_j, z_2\rangle = \beta_j \langle \nu_j, e_j\rangle$
holds for all $j \in \AA(x)$, where
$ \langle \nu_j, e_j\rangle < 0$. Thus, $\beta_j \geq 0$ for all $ j \in \AA(x)$
and \eqref{eq:e_i-equivalence} is proved. 
This establishes \ref{lem:ZZproperties:iv} and completes the proof of the lemma.
\end{proof}

\begin{lemma}[formulas for polar and normal cones]
\label{lem:normal_cone_polyhedral}
\hfill
\begin{enumerate}[label=\roman*)]
\item
Let $Z$ be a convex polyhedral cone
with a description $\{(\nu_i, 0)\}_{i \in I}$. Then it holds 
\[
Z^\circ = \left \{ \left.
\sum_{i \in I} \beta_i \nu_i ~
\right |~ \beta_i \geq 0~ \forall i \in I
\right \}.
\]
\item
Let $Z$ be a nonempty convex polyhedron with a description $\{(\nu_i, \alpha_i)\}_{i \in I}$. 
Then, for every $x \in Z$, it holds
\[
\NN_Z(x) = \ZZ(x)^\circ = \left \{ \left.
\sum_{i \in \AA(x)} \beta_i \nu_i ~
\right |~ \beta_i \geq 0~ \forall i \in \AA(x)
\right \}.
\]
\end{enumerate}
\end{lemma}
\begin{proof}
This is \cite[Lemma 3.1]{Faigle2002}.
\end{proof}

Note that the right-hand sides of the formulas in \cref{lem:normal_cone_polyhedral}
are understood as $\{0\}$ in the case that the summation index set is empty (by the convention 
for the empty sum). We are now in the position to introduce:

\begin{definition}[non-obtuse polyhedron]
\label{def:properties_descript} 
A full-dimensional convex polyhedron $Z\subset \R^d$ is called \emph{non-obtuse} 
if $\langle \nu_i, \nu_j\rangle \leq 0$  for all $i,j \in I$ with  $i \neq j$
holds for one (and thus all) of its standard descriptions $\{(\nu_i, \alpha_i)\}_{i \in I}$.
\end{definition}

The following two results collect properties of non-obtuse 
polyhedra that are essential 
for the analysis of the jumps of the directional 
derivatives of the vectorial stop operator in \cref{subsec:3.3,subsec:3.4}.

\begin{proposition}[regularity of non-obtuse polyhedra]
\label{lem:descriptions:iii}
Let $Z \subset \R^d$, $d \in \N$, be a full-dimensional convex non-obtuse
polyhedron and let $\HH = \{(\nu_i, \alpha_i)\}_{i \in I}$ be a standard description of $Z$. 
Then $\HH$ is regular.
\end{proposition}
\begin{proof}
Let $x \in Z$ be arbitrary
and assume that $\beta_i \in \R$, $i \in \AA(x)$,
are numbers
such that  
$\sum_{i \in \AA(x)} \beta_i \nu_i = 0$ holds.
Choose $\tilde x \in \interior(Z)$.
Then 
it holds
$\langle \nu_i, \tilde x\rangle < \alpha_i$
for all $i \in I$ 
and 
$h := \tilde x -x \in \R^d$ satisfies 
$\langle \nu_i, h \rangle < 0$ for all $i \in \AA(x)$. 
Since $Z$ is non-obtuse,
we further obtain 
\begin{equation*}
\begin{aligned}
0
=
\left |
\sum_{i \in \AA(x)} \beta_i \nu_i
\right |^2
&=
\left |
\sum_{i \in \AA(x), \beta_i > 0} \beta_i \nu_i
\right |^2
+
2
\sum_{i \in \AA(x), \beta_i > 0} 
\sum_{j \in \AA(x), \beta_j < 0}
\beta_i \beta_j 
\left \langle
 \nu_i,
\nu_j
\right \rangle
+
\left |
\sum_{i \in \AA(x), \beta_i < 0} \beta_i \nu_i
\right |^2
\\
&\geq
 \left |
\sum_{i \in \AA(x), \beta_i > 0} \beta_i \nu_i
\right |^2
+
\left |
\sum_{i \in \AA(x), \beta_i < 0} \beta_i \nu_i
\right |^2.
\end{aligned}
\end{equation*}
The above implies 
\[
0
=
\left \langle 
h, \sum_{i \in \AA(x), \beta_i > 0} \beta_i \nu_i
\right \rangle 
=
 \sum_{i \in \AA(x), \beta_i > 0} 
 \beta_i \left \langle h, \nu_i\right \rangle 
\quad
\text{ and }
\quad
0=
\left \langle 
h, \sum_{i \in \AA(x), \beta_i < 0} \beta_i \nu_i
\right \rangle 
=
 \sum_{i \in \AA(x), \beta_i < 0} \beta_i
 \left \langle 
h,  \nu_i\right \rangle.
\]
As $\langle h,\nu_i \rangle < 0$ for all $i \in \AA(x)$, this
can only be true if 
there is no $i \in \AA(x)$ with $\beta_i \neq 0$.
This shows that the vectors in the set 
$\{ \nu_i \mid i \in \AA(x)\}$ are linearly independent
and that $\HH$ is indeed regular.
\end{proof}

\begin{lemma}[properties of non-obtuse polyhedra]
\label{lem:nonobtuse_properties}
Let $Z \subset \R^d$, $d \in \N$,
be a  full-dimensional  non-obtuse convex polyhedron
and let $\HH = \{(\nu_i, \alpha_i)\}_{i \in I}$ be a 
standard description of $Z$.
Then the following is true.
\begin{enumerate}[label=\roman*)]
\item\label{lem:nonobtuse_properties:i} 
It holds  
\begin{equation}
\label{eq:norm_proj}
x, y \in Z,  v \in \R^d, x + v \in Z
\qquad
\Rightarrow
\qquad
x + \pi_{\ZZ(y)^\circ}(v) \in Z.
\end{equation}

\item\label{lem:nonobtuse_properties:ii} 
For all $j \in I$, it holds 
\[
x \in \R^d,~\langle \nu_j, x \rangle = \alpha_j
\qquad\Rightarrow
\qquad\langle \nu_j, \pi_Z(x) \rangle = \alpha_j.
\]
\item\label{lem:nonobtuse_properties:iii} 
If $Z$ is additionally a cone, then it holds 
\begin{equation*}
x,y \in Z\qquad \Rightarrow \qquad 
\pi_{y + Z^\circ}(x)\in Z \text{ and } 
x - \pi_{y + Z^\circ}(x) = \pi_{Z}(x-y)\in Z. 
\end{equation*}
\item \label{lem:nonobtuse_properties:iv}
If $Z$ is additionally a cone,
$J \subset I$ an index set, and
$V := \spann(\{\nu_i \mid i \in J\})^\perp$,
then it holds 
\[
x \in Z\qquad \Rightarrow \qquad \pi_V(x) \in Z.
\]
\end{enumerate}
\end{lemma}

\begin{proof}
We begin with \ref{lem:nonobtuse_properties:i}:
Let $x$, $y$, and $v$ be as on the left-hand side of \eqref{eq:norm_proj}. 
For $\ZZ(y) = \R^d$, \eqref{eq:norm_proj} is trivially true. 
We may thus assume w.l.o.g.\ that $\ZZ(y) \neq \R^d$. 
In this case, we have $\AA(y) \neq \emptyset$. 
Applying \cref{lem:basic_proj_properties}\ref{lem:basic_proj_properties:iv} and 
\cref{lem:normal_cone_polyhedral}, we see 
that there exist $\beta_i \ge 0$, $i \in \AA(y) $, such that
\begin{equation}
\label{eq:KKT_cone}
\pi_{\ZZ(y)}(v) + \pi_{\ZZ(y)^\circ}(v) = v,
\qquad
\langle \pi_{\ZZ(y)^\circ}(v), \pi_{\ZZ(y)}(v) \rangle = 0,
\qquad
\pi_{\ZZ(y)^\circ}(v) = \sum_{i\in \AA(y) }\beta_i \nu_i.
\end{equation}
It follows that
$\sum_{i\in \AA(y) } \beta_i \langle \nu_i , \pi_{\ZZ(y)}(v) \rangle = 0.$
Since all summands are nonpositive, 
\begin{equation}\label{eq:random56321}
\beta_j \langle \nu_j , \pi_{\ZZ(y)}(v) \rangle = 0 \qquad \forall j \in \AA(y) .
\end{equation}
Consider now some $j \in I$. 
If $j \not \in \AA(y)$ or $j \in \AA(y)$ and $\beta_j = 0$, then 
the non-obtuseness of $Z$, $x \in Z$, and \eqref{eq:KKT_cone} yield
\[
\left \langle 
\nu_j, x + \pi_{\ZZ(y)^\circ}(v)
\right \rangle
=
\left \langle 
\nu_j, x 
\right \rangle 
+  
\sum_{i\in \AA(y), i \neq j} \beta_i 
\left \langle 
\nu_j,\nu_i
\right \rangle 
\leq
\left \langle 
\nu_j, x 
\right \rangle 
\leq
\alpha_j. 
\]
If, on the other hand, 
$j \in \AA(y)$ and $\beta_j > 0$, then it follows  from \eqref{eq:random56321}
and $x + v \in Z$ that 
\[
\left \langle 
\nu_j, x + \pi_{\ZZ(y)^\circ}(v)
\right \rangle
=
\left \langle 
\nu_j, x + \pi_{\ZZ(y)^\circ}(v) + \pi_{\ZZ(y)}(v)
\right \rangle
=
\left \langle 
\nu_j, x + v
\right \rangle
\leq
\alpha_j. 
\]
Thus,  $x + \pi_{\ZZ(y)^\circ}(v) \in Z$ and the proof is complete. 

Next, we prove \ref{lem:nonobtuse_properties:ii}:
Let $x \in \R^d$ with $\langle \nu_j, x \rangle = \alpha_j$ for some $j \in I$.
Since $x - \pi_Z(x)$ belongs to the normal cone to $Z$ at $\pi_Z(x)$,
it follows from \cref{lem:normal_cone_polyhedral} 
that there exist $\beta_i\ge 0$ such that
\[
x - \pi_Z(x) = \sum_{i\in \AA(\pi_Z(x))} \beta_i \nu_i.
\]
Since $\langle \nu_j, x \rangle = \alpha_j$,
\begin{equation}\label{eq:random46267}
\alpha_j - \langle \nu_j, \pi_Z(x) \rangle = \sum_{i\in \AA(\pi_Z(x))} \beta_i \langle \nu_j, \nu_i \rangle. 
\end{equation}
Suppose now that $\langle \nu_j, \pi_Z(x) \rangle < \alpha_j$.
Then the left side of \eqref{eq:random46267} is $> 0$, but the right side is $\le 0$
as $j \notin \AA(\pi_Z(x))$ and $Z$ is non-obtuse, a contradiction.
Thus, $\langle \nu_j, \pi_Z(x) \rangle = \alpha_j$
and the proof of \ref{lem:nonobtuse_properties:ii} is complete.

To prove \ref{lem:nonobtuse_properties:iii},
suppose that $Z$ is a cone 
and that $x, y \in Z$ are given.
We claim that 
$\pi_{y + Z^\circ}(x) = y + \pi_{Z^\circ}(x - y)$ holds.
To see this, we note that we trivially have 
$y + \pi_{Z^\circ}(x - y) \in y + Z^\circ$
and that
\eqref{eq:proj_VI} for $\pi_{Z^\circ}$ yields 
\begin{equation*}
\left \langle
y + \pi_{Z^\circ}(x - y) - x, 
w - y - \pi_{Z^\circ}(x - y)
\right \rangle
=
\left \langle
\pi_{Z^\circ}(x - y) - (x-y), 
(w-y) - \pi_{Z^\circ}(x - y)
\right \rangle
\geq 0
\quad
\forall w \in y + Z^\circ.
\end{equation*}
Thus, $\pi_{y + Z^\circ}(x) = y + \pi_{Z^\circ}(x - y)$ as claimed. 
From \eqref{eq:Z-Zcirc-identity+hom}, 
we may now deduce that 
$
x = y + \pi_{Z^\circ}(x-y) + \pi_{Z}(x-y)
=  \pi_{y + Z^\circ}(x)  + \pi_{Z}(x-y).
$
The assertion $x - \pi_{y + Z^\circ}(x) = \pi_{Z}(x-y)\in Z$
follows immediately from this identity. 
It remains to prove that 
$\pi_{y + Z^\circ}(x) = y + \pi_{Z^\circ}(x-y)\in Z$.
Since $\ZZ(0) =Z$ by \cref{lem:ZZproperties}\ref{lem:ZZproperties:iii}, 
setting $y = 0$ in \eqref{eq:norm_proj}
and replacing there $x$ by $y$ as well as $v$ by $x - y$,
we see that
$y + \pi_{Z^\circ}(x-y) \in Z$.
This establishes \ref{lem:nonobtuse_properties:iii}.

It remains to prove \ref{lem:nonobtuse_properties:iv}:
By \cref{lem:standard_description_cone}, we have 
$\HH = \{(\nu_i, 0)\}_{i \in I}$.
If $J = \emptyset$ holds, 
then we have $V = \R^d$ by 
our convention $\spann(\emptyset) = \{0\}$
and there is nothing to prove. 
So let us assume that $\emptyset \neq J \subset I = \{1,...,n\}$,
$n \in \N$, holds; w.l.o.g.\ $J = \{1,...,m\}$ for some $m \leq n$.
Suppose further that $x \in Z$ is given. Then there exists
$\beta := (\beta_1,...,\beta_m)^\top \in \R^m$
such that 
$\pi_V(x) + \sum_{j=1}^m \beta_j \nu_j = x$. 
Define $M := (\nu_1,...,\nu_m) \in \R^{d \times m}$. 
Then it follows from $x \in Z$
and 
$\langle \nu_i,  \pi_V(x) \rangle = 0$ for all
$i=1,...,m$ that
$
M^\top M \beta
=
M^\top \left (
\pi_V(x)
+
M\beta
\right )
=
M^\top x 
\in (-\infty, 0]^m. 
$
By taking the scalar product
with the vector 
$\tilde \beta := (\max(0,\beta_1),...,\max(0,\beta_m))^\top \in \R^m$
in the above,
by exploiting the non-obtuseness of $Z$,
and by distinguishing cases, we obtain that 
\begin{equation*}
0 \geq 
\sum_{i=1}^m \sum_{j=1}^m
\max(0,\beta_i)
\langle  \nu_i, \nu_j \rangle 
\beta_j
\geq
\sum_{i=1}^m \sum_{j=1}^m
\max(0,\beta_i)
\langle  \nu_i, \nu_j \rangle 
\max(0,\beta_j)
=
\langle 
M\tilde \beta, M \tilde \beta
\rangle 
=
  | M\tilde \beta  |^2.
\end{equation*}
Since $\HH$ is regular due to \cref{lem:descriptions:iii},
the above implies $\tilde \beta = 0$ and, 
as a consequence, $\beta_i \leq 0$ for all $i=1,...,m$.
Since $Z$ is non-obtuse and $x \in Z$,
the nonpositivity of the numbers $\beta_i$, $i=1,...,m$,
implies 
\[
\left \langle
\nu_i, \pi_V(x)
\right \rangle 
=
\left \langle
\nu_i, x - \sum_{j=1}^m \beta_j \nu_j 
\right \rangle 
=
\left \langle
\nu_i, x 
\right \rangle 
-
\sum_{j=1}^m \beta_j 
\left \langle
\nu_i,
\nu_j
\right \rangle 
\leq
\left \langle
\nu_i, x 
\right \rangle 
\leq 0\qquad \forall i \in I \setminus J. 
\]
As $\left \langle
\nu_i, \pi_V(x)
\right \rangle = 0$ holds for all $i \in J$
by the definition of $V$,
this shows $\pi_V(x) \in Z$
and completes the proof. 
\end{proof}

Next, we study the properties of 
polyhedra that are not non-obtuse. 
Our main goal is to prove that such 
polyhedra always possess 
an obtuse ridge.
This result is essential for the 
construction of the 
counterexamples in \cref{sec:4}.
We proceed in two steps, 
beginning with the special case of a cone. 

\begin{lemma}[existence 
of an obtuse ridge for a
polyhedral cone that is not non-obtuse]
\label{lem_counterexample_cone}
Let $Z \subset \R^d$, $d \in \N$, 
be a full-dimensional convex polyhedral cone
that is not non-obtuse.
Let $\HH = \{(\nu_i, 0)\}_{i \in I}$
be a standard description of $Z$. 
Then there exists $w \in Z$ 
such that $\AA(w) = \{i_1, i_2\}$ holds for some 
$i_1, i_2\in I$ satisfying $\langle \nu_{i_1}, \nu_{i_2}\rangle > 0$. 
\end{lemma}
\begin{proof}
We use induction w.r.t.\ the 
cardinality $n \in \N_0$ of the index set $I$
of the standard description of $Z$.
For $n\in \{0,1\}$, there are no polyhedral cones 
that are not non-obtuse. For $n=2$, 
the assertion is trivial. 

Suppose now that $2 < n \in \N$ is given and 
that the assertion is proven for 
all cones with the 
properties in the statement of the lemma 
whose  standard  descriptions have at most $n-1$ elements.
Assume further that 
a cone $Z \subset \R^d$
fulfilling the assumptions of the lemma 
is given with a  standard description 
$\{(\nu_i, 0)\}_{i \in I}$ satisfying
$|I| = n$. 
As $Z$ is not non-obtuse, 
there exist $i_1, i_2 \in I$
with $i_1 < i_2$
such that 
$\langle \nu_{i_1}, \nu_{i_2} \rangle \in (0,1)$.
Define 
\[
J := 
\left \{j \in I\setminus \{i_1\} \mid \exists x \in 
Z \cap \spann(\{\nu_1,...,\nu_n\}) \setminus \{0\}
\text{ s.t.\ } 
\langle \nu_{i_1}, x\rangle = \langle \nu_{j}, x\rangle 
 = 0 \right \}.
\]

We claim that $J \neq \emptyset$. 
To see this, choose
$\tilde x$ such that 
the conditions 
in \eqref{eq:faces} hold for $i_1$.
We may assume w.l.o.g.\ 
that $\tilde x \in \spann(\{\nu_1,...,\nu_n\})$
(else we replace $\tilde x$ 
by the projection $\pi_{\spann(\{\nu_1,...,\nu_n\})}(\tilde x)$).
Choose further 
$y \in 
\spann(\{\nu_{i_1}, \nu_{i_2}\})^\perp
\cap \spann(\{\nu_1,...,\nu_n\}) \setminus \{0\}$.
(Such a vector $y$ exists since $\dim(\spann(\{\nu_1,...,\nu_n\})) \geq 3$
by our assumption $n\geq 3$
and the properties in \cref{lem:standard_description}.)
Define $x_s := (1-s) \tilde x + s y$,
$s \in [0,1]$. 
Then it holds 
\[
\langle \nu_{i_1}, x_s\rangle = 0~~
\forall s \in [0,1],
\qquad
\langle \nu_{i_2}, x_1\rangle = 0,
\qquad
\langle \nu_{i_2}, x_s\rangle < 0
~~\forall s \in [0,1).
\]
If $y = x_1 \in Z$ holds, then the above 
yields $i_2 \in J$ and we obtain $J \neq \emptyset$
as desired. 
If $y = x_1 \not \in Z$ holds, 
then we have $\bar s := \sup\{s \in [0,1] \mid x_s \in Z\} \in (0,1)$. 
This implies that there exists 
$l \in I \setminus \{i_1\}$ such that $\langle \nu_l, x_s \rangle< 0$
holds for all $s \in [0, \bar s)$
and $\langle \nu_l, x_{\bar s} \rangle = 0$.
As $\langle \nu_{i_2}, x_{\bar s} \rangle< 0$
holds and due to the properties of 
$\tilde x$, $y$, and $\bar s$, we have 
$x_{\bar s} \in Z \cap \spann(\{\nu_1,...,\nu_n\}) \setminus \{0\}$. 
Thus, $l \in J$ and $J \neq \emptyset$ as asserted.

Next, we prove that 
\begin{equation}
\label{eq:randomeq3636378}
Z \cap 
\spann(\{\nu_1,...,\nu_n\})
\cap
\spann(\{\nu_{i_1}\})^\perp
=
\{
x \in \spann(\{\nu_1,...,\nu_n\})
\mid
\langle \nu_{j}, x\rangle 
 \leq 0 =
\langle \nu_{i_1}, x\rangle \,
 \forall j \in J
\}
.
\end{equation}
To see that this assertion is true,
we first note that the definitions of the 
involved sets immediately imply that ``$\subset$''
holds in \eqref{eq:randomeq3636378}. 
To see that we also have ``$\supset$'',
we argue by contradiction. 
Suppose that there exists 
$x \in \R^d$ that is contained 
in the set on the right-hand side 
of \eqref{eq:randomeq3636378} but not 
in the set on the left. 
Then it holds 
$x \in \spann(\{\nu_1,...,\nu_n\}) \setminus \{0\}$,
$\langle \nu_{i_1}, x\rangle = 0$,
and there exists a nonempty 
index set $I' \subset I \setminus (J \cup \{i_1\})$
such that $\langle \nu_i, x\rangle > 0$
for all $i \in I'$
and $\langle \nu_i, x\rangle \leq 0$
for all $i \in I \setminus I'$. 
We again consider a point 
$\tilde x \in \spann(\{\nu_1,...,\nu_n\})$
that satisfies \eqref{eq:faces} 
for the index $i_1$
and define 
$x_s := (1-s) \tilde x + s x$,
$s \in [0,1]$.
Then it holds 
$x_s 
\in \spann(\{\nu_1,...,\nu_n\})$
for all $s \in [0,1]$ and 
\[
\langle \nu_i, x_0 \rangle < 0 
~\forall i \in I \setminus \{i_1\},
\quad
\langle \nu_{i_1}, x_s \rangle = 0 
~\forall s \in [0,1],
\quad
\langle \nu_{i}, x_s \rangle < 0 
~\forall s \in [0,1)
\,\forall i \in J,
\quad
\langle \nu_i, x_1 \rangle > 0 
~\forall i \in I'.
\]
Define 
$\bar s := \sup\{s \in [0,1] \mid x_s \in Z\}$. 
Then we again have $\bar s \in (0,1)$
and for $x_{\bar s}$
there exists 
$l \in I \setminus (J \cup \{i_1\})$ 
such that $\langle \nu_l, x_s \rangle < 0$
holds for all $s \in [0, \bar s)$
and $\langle \nu_l, x_{\bar s} \rangle = 0$.
As $J \neq \emptyset$ and 
$\langle \nu_{i}, x_s \rangle < 0$
for all $s \in [0,1)$
and all $i \in J$ and due to the definition of $\bar s$,
we have $x_{\bar s} \in Z \setminus \{ 0\}$. 
By the definition of $J$,
this yields $l \in J$, which 
creates a contradiction.
Thus, \eqref{eq:randomeq3636378} holds as claimed.

Now we can turn our attention to the 
proof of the assertion of the lemma. 
Let us first consider the special case 
$J = \{i_2\}$.
In this case, it follows from the definition of $J$
that there exists 
$x \in Z \cap \spann(\{\nu_1,...,\nu_n\}) \setminus \{0\}$
such that 
$\langle \nu_{i_1}, x\rangle = \langle \nu_{i_2}, x\rangle  = 0$
holds and it follows from $J = \{i_2\}$
that this $x$ satisfies 
$\langle \nu_j, x\rangle < 0$ for all 
$j \in I \setminus \{i_1, i_2\}$. In this situation, the 
assertion of the lemma holds with $w := x$ and we are done. 

Henceforth, we may thus assume that 
$J\setminus \{i_2\} \neq \emptyset$.
Let now $\{\mu_1,...,\mu_m\}$ be a basis of
$\spann(\{\nu_1,...,\nu_n\})^\perp$.
(If $\spann(\{\nu_1,...,\nu_n\}) = \R^d$ holds, then we set $m := 0$, $\{\mu_1,...,\mu_m\} := \emptyset$.)
We invoke Motzkin's theorem of the alternative 
\cite[Theorem 4.2, page 28]{Mangasarian1994}
to obtain that precisely one of the following statements is 
true:
\begin{itemize}
\item[(I)] There exists $h \in \R^d$ such that 
\[
\langle\nu_i, h \rangle < 0
~\forall i \in J\setminus \{i_2\},
\qquad
\langle \nu_{i_1}, h \rangle
=
\langle \nu_{i_2}, h \rangle
=
0,
\qquad
\langle \mu_i, h \rangle
= 0 \quad \forall i=1,...,m.
\]

\item[(II)]
There exist 
$\beta_i \geq 0$,
$i \in J\setminus \{i_2\}$,
not all zero, and 
numbers $\gamma_1, \gamma_2, \eta_i \in \R$,
$i=1,...,m$, such that 
\begin{equation}
\label{eq:randomeq73hu2oi3}
\sum_{i \in J\setminus \{i_2\}} \beta_i \nu_i
=
\gamma_1 \nu_{i_1}
+
\gamma_2 \nu_{i_2}
+
\sum_{i=1}^m \eta_i \mu_i.
\end{equation}
\end{itemize}

Let us first consider option (I).
In this case, 
$\langle \mu_i, h \rangle = 0$
for all $i=1,...,m$ yields 
$h \in \spann(\{\nu_1,...,\nu_n\})$,
we have $\langle \nu_{i_1}, h \rangle =
\langle \nu_{i_2}, h \rangle = 0$,
$h$ is not zero since $J\setminus \{i_2\}$ is not empty, 
and it holds $h \in Z$ by \eqref{eq:randomeq3636378}.
Due to the definition of the set $J$,
this yields that we also have 
$\langle \nu_i, h \rangle <0$
for all $i \in I \setminus (J \cup \{i_1, i_2\})$.
Thus, 
after collecting everything, 
we obtain  
$\langle \nu_{i_1}, h \rangle =
\langle \nu_{i_2}, h \rangle = 0$
and $\langle\nu_i, h \rangle < 0$
for all $i \in I \setminus \{i_1, i_2\}$. 
In this case, $w := h$ satisfies 
all of the conditions in the lemma and 
we are done. 

It remains to study option (II). 
In this case, 
by again choosing 
$\tilde x \in \spann(\{\nu_1,...,\nu_n\})$
for $i_1$ as in \eqref{eq:faces} and
by exploiting 
that the coefficients $\beta_i$
are nonnegative and not all zero,
we obtain that 
\[
0 > \sum_{i \in J\setminus \{i_2\}}
\left \langle 
 \beta_i \nu_i,
\tilde x
\right \rangle
=
\left \langle 
\gamma_1 \nu_{i_1}
+
\gamma_2 \nu_{i_2},
\tilde x
\right \rangle
=
\gamma_2 
\left \langle 
\nu_{i_2},
\tilde x
\right \rangle.
\]
Since $\left \langle 
\nu_{i_2},
\tilde x
\right \rangle < 0$,
this yields $\gamma_2 > 0$.
Analogously, we also obtain $\gamma_1 > 0$.
By taking the scalar product 
with $\nu_{i_1}$ in \eqref{eq:randomeq73hu2oi3} 
and by exploiting 
that $\langle \nu_{i_1}, \nu_{i_2}\rangle > 0$,
we now obtain 
\[
0
< 
\gamma_1
<
\langle 
\gamma_1 \nu_{i_1}
+
\gamma_2 \nu_{i_2},
\nu_{i_1}
\rangle 
=
\sum_{i \in J\setminus \{i_2\}} \beta_i 
\langle 
\nu_i,
\nu_{i_1}
\rangle.
\]
Due to $\beta_i \geq 0$,
the above can only be true 
if there exists $k \in  J \setminus \{i_2\}$
with $\langle \nu_k, \nu_{i_1} \rangle > 0$.
According to the definition of $J$,
we can find 
$x \in 
Z \cap \spann(\{\nu_1,...,\nu_n\}) \setminus \{0\}
$
for this $k$ such that 
$\langle \nu_{i_1}, x\rangle = \langle \nu_{k}, x\rangle = 0$.
Note that there exists $j \in I$
such that $\langle \nu_j, x \rangle < 0$ for this $x$.
Indeed, if this was not the case, 
then $\langle \nu_i, x \rangle = 0$ for all $i \in I$
and $x \in \spann(\{\nu_1,...,\nu_n\})$
would imply $x=0$ and yield a contradiction. 
Consider now the cone 
$\ZZ(x) = \{z \in \R^d \mid\langle \nu_i, z\rangle \leq 0~\forall i \in \AA(x) \}$. 
By \cref{lem:ZZstandard}, $\ZZ(x)$ is full-dimensional and
$\{(\nu_i, 0)\}_{i \in \AA(x)}$ is a standard description of $\ZZ(x)$.
Further, we have
$2 \leq |\AA(x)| \leq |I| - 1$ by construction
and $\ZZ(x)$ is not non-obtuse 
since $i_1, k \in \AA(x)$.
By our induction hypothesis, this yields 
that there exists 
$\tilde w \in \ZZ(x)$ 
such that 
$
\{i \in \AA(x) \mid \langle \nu_i, \tilde w\rangle = 0 \}
=
\{j_1, j_2\}$
with $\nu_{j_1}, \nu_{j_2}$
satisfying 
$\langle \nu_{j_1}, \nu_{j_2}\rangle > 0$.
Define $w_s := x + s \tilde w$, $s > 0$. 
Then
\[
\langle 
\nu_i, w_s
\rangle
=
\begin{cases}
\langle 
\nu_i, s \tilde w
\rangle = 0 &\text{ if } i \in \{j_1, j_2\},
\\
\langle 
\nu_i, s \tilde w
\rangle < 0 &\text{ if } i \in \AA(x) \setminus \{j_1, j_2\},
\\
\langle 
\nu_i, x
\rangle
+
s
\langle 
\nu_i,\tilde w
\rangle &\text{ if } i \in \II(x),
\end{cases}
\]
holds for all $s>0$.
By choosing $s>0$ small enough, 
we achieve that 
$\langle 
\nu_i, w_s
\rangle < 0$ for all $i \in \II(x)$.
The resulting $w_s \in Z$ then has all of the desired properties. 
This completes the proof. 
\end{proof}

Next, we extend \cref{lem_counterexample_cone}
to arbitrary polyhedra that are not non-obtuse. 

\begin{proposition}[existence 
of an obtuse ridge in the general case]
\label{lem_counterexample}
Let $Z \subset \R^d$, $d \in \N$, 
be a full-dimensional convex polyhedron that is not non-obtuse.
Let $\HH =\{(\nu_i, \alpha_i)\}_{i \in I}$  be 
a standard  description of $Z$.
Then there exists $w \in Z$ such that $\AA(w) = \{i_1, i_2\}$ holds for some 
$i_1, i_2\in I$ satisfying $\langle \nu_{i_1}, \nu_{i_2}\rangle > 0$. 
\end{proposition}

\begin{proof}
As $Z$ is not non-obtuse, there exist $k,l \in I$ with $\langle \nu_k, \nu_l \rangle > 0$. 
From  \cref{lem:standard_description}, we obtain that, for these $k,l$,
we can find $x_k, x_l \in Z$
satisfying 
$\langle \nu_k, x_k\rangle = \alpha_k$,
$\langle \nu_i, x_k\rangle < \alpha_i ~\forall i \neq k$,
$\langle \nu_l, x_l\rangle = \alpha_l$,
$\langle \nu_i, x_l\rangle < \alpha_i ~\forall i \neq l$.
Define
\[
\tilde x := x_k - \frac{\langle \nu_l, x_k\rangle - \alpha_l}{\langle \nu_l, \nu_k \rangle}\nu_k \in \R^d,
\qquad 
h := 
\tilde x
 - x_l,
\quad \text{and}\quad
\gamma\colon [0,1] \to \R^d,
~
\gamma(s) := x_l + s h.
\]
Then we have $\gamma(0) = x_l$ and $\pi_Z(\gamma(1)) = \pi_Z(\tilde x) = x_k$, 
it holds 
\begin{equation}
\label{eq:randomeq3738i}
\left \langle
\nu_l , \gamma(s)
\right\rangle
=
(1 - s)
\left \langle
\nu_l , x_l 
\right\rangle
 +
 s
 \left \langle
\nu_l , \tilde x
\right\rangle
=
(1-s) \alpha_l
+
s
 \left \langle
\nu_l ,
 x_k - \frac{\langle \nu_l, x_k\rangle - \alpha_l}{\langle \nu_l, \nu_k \rangle}\nu_k
\right\rangle
= \alpha_l\quad \forall s \in [0,1],
\end{equation}
and there exists $\varepsilon \in (0,1)$ such that
\begin{equation}
\label{eq:randomeq36363}
\pi_Z(\gamma(s)) = \gamma(s) \qquad \forall s \in [0,\varepsilon].
\end{equation}
Define 
\[
\bar s := \sup\{s \in [0,1] \mid \langle \nu_l, \pi_Z(\gamma(r)) \rangle = \alpha_l ~\forall r \in [0, s]\}.
\]
Then \eqref{eq:randomeq3738i}, 
\eqref{eq:randomeq36363},
$\langle \nu_l, \pi_Z(\gamma(1)) \rangle = \langle \nu_l, x_k \rangle < \alpha_l$, 
and the continuity of $\pi_Z$ imply that 
$0 < \varepsilon \leq  \bar s < 1$.
Consider now the point
$y := \pi_Z(\gamma(\bar s))$. Then 
the definition of $\bar s$ implies that there exists 
a sequence $\{s_m\} \subset (\bar s, 1]$ such that $s_m \to \bar s$
holds 
and  $\langle \nu_l, \pi_Z(\gamma(s_m)) \rangle < \alpha_l  = \langle \nu_l, y\rangle$
for all $m \in \N$. In combination with 
\cref{lem:ZZproperties}\ref{lem:ZZproperties:ii},
it follows that, for all sufficiently large $m$, we have 
\[
\langle \nu_l, \pi_Z(\gamma(s_m)) \rangle = 
\langle \nu_l, y + \pi_{\ZZ(y) }(\gamma(s_m) - y) \rangle
=
\alpha_l
+
\langle \nu_l, \pi_{\ZZ(y) }(\gamma(s_m) - y) \rangle
< \alpha_l 
\]
and, thus, 
$
\langle \nu_l, \pi_{\ZZ(y) }(\gamma(s_m) - y) \rangle
< 0
$.
Due to \eqref{eq:randomeq3738i} and the definition of $y$, however, we have 
$l \in  \AA(y)$ and
$
\langle \nu_l, \gamma(s_m) - y  \rangle
 =
 \alpha_l - \alpha_l = 0. 
$
In view of \cref{lem:nonobtuse_properties}\ref{lem:nonobtuse_properties:ii}, 
this shows that $\ZZ(y)$ is not non-obtuse.
(Note that $\ZZ(y)$ is full-dimensional and
that 
$\{(\nu_i, 0)\}_{i \in \AA(y)}$ is a standard description of $\ZZ(y)$
by \cref{lem:ZZstandard}.)
Thus, \cref{lem_counterexample_cone} 
can be applied to the polyhedral cone  $\ZZ(y)$
and it follows that there exists 
$\tilde w \in \ZZ(y)$
such that 
$
\{i \in \AA(y) \mid \langle \nu_i, \tilde w\rangle = 0 \}
=
\{ i_1,  i_2\}$
holds with $\nu_{ i_1}, \nu_{ i_2}$
satisfying 
$\langle \nu_{ i_1}, \nu_{ i_2}\rangle > 0$.
Define $w_s := y + s \tilde w$, $s > 0$.
Then
it follows analogously to the last step 
in the proof of \cref{lem_counterexample_cone} 
that, for all small enough $s>0$, we have 
\[
\langle 
\nu_i, w_s
\rangle
=
\langle
\nu_i, y
\rangle
+s
\langle 
\nu_i,  \tilde w
\rangle
\begin{cases}
 = 
\alpha_i + 0 &\text{ if } i \in \{ i_1,  i_2\},
\\
 < \alpha_i 
&\text{ if } i \in \AA(y) \setminus \{ i_1, i_2\},
\\
< \alpha_i &\text{ if } i \in \II(y).
\end{cases}
\]
The above shows that we can choose $w:= w_s$
for a small enough $s>0$ to obtain a point 
$w \in Z$ with all of the desired properties. 
This completes the proof. 
\end{proof}

\begin{corollary}[equivalence of local and global non-obtuseness]
\label{cor:local_nonobtuse}
Let $Z \subset \R^d$ be a full-dimensional convex polyhedron
and let $\{(\nu_i, \alpha_i)\}_{i \in I}$ be a standard description of $Z$.
Then the following statements are equivalent.
\begin{enumerate}[label=\roman*)]
\item 
$Z$ is non-obtuse, that is, $\langle \nu_i, \nu_j \rangle \le 0$ 
for all $i,j\in I$ with $i \neq j$.
\item 
$Z$ has only non-obtuse ridges, that is, 
for all $x \in Z$ with $\AA(x) = \{i_1, i_2\}$,
it holds 
$\langle \nu_{i_1}, \nu_{i_2} \rangle \le 0$.
\end{enumerate}
\end{corollary}
 
\subsection{Preliminaries on the Kurzweil-Stieltjes integral}
\label{subsec:2.4}
A further important ingredient of our analysis is Kurzweil-Stieltjes integration theory.  
Let $a,b \in \R$, $a < b$, be given and fixed. 
We denote the Kurzweil-Stieltjes integral 
of an \emph{integrand} $f\colon [a,b] \to \R$ 
w.r.t.\ an \emph{integrator} $g\colon [a,b] \to \R$ 
by
\[
\int_a^b  f \,\dd g 
\qquad \text{or}
\qquad
\int_a^b  f(s)\,\dd g (s).
\]
For an in-depth introduction to the theory of Kurzweil-Stieltjes
integration for scalar functions, we refer the reader to \cite[Chapter 6]{Monteiro2019}.
Let us now additionally assume that $d \in \N$ is given and fixed.
If $f\colon [a,b] \to \R^d$ and $g\colon [a,b] \to \R^d$ are vector functions, 
then we define the Kurzweil-Stieltjes integral of $f$ w.r.t.\ $g$ via
\begin{equation}
\label{eq:KS-interal-1}
\int_a^b\langle f , \dd g\rangle := 
\sum_{j=1}^d \int_a^b f_j \,\dd g_j,
\end{equation}
provided that the scalar Kurzweil-Stieltjes integrals on the right-hand side of 
this definition all exist. 
Due to its componentwise definition,
the vectorial version of the Kurzweil-Stieltjes integral 
in \eqref{eq:KS-interal-1} inherits essentially all properties 
from its scalar counterpart studied in \cite{Monteiro2019}.
We briefly recall the most important facts and results in the following.
For more details, see \cite[Chapter 6]{Monteiro2019}.

\begin{lemma}[existence of the integral]
\label{lemma:KS_basics}
The Kurzweil-Stieltjes integral in \eqref{eq:KS-interal-1} exists 
if $f, g \in G([a,b]; \R^d)$ holds and 
at least one of the functions $f$ and $g$ is an element of $BV([a,b]; \R^d)$.
In this case, it yields a real number. 
If $f \in C([a,b]; \R^d)$ and $g \in BV([a,b]; \R^d)$, then the
Kurzweil-Stieltjes integral  
coincides with the Riemann-Stieltjes integral. In particular, we have 
\begin{equation*}
\int_a^b \langle f , \dd g\rangle = 
\int_a^b \langle f(s), \dot g(s)\rangle \dd s
\qquad \forall f \in C([a,b]; \R^d)
\quad 
\forall g \in W^{1,1}((a,b); \R^d).
\end{equation*}
\end{lemma}
\begin{proof}
The assertions follow from 
\cite[Theorems 5.6.1, 6.2.12, 6.3.11, Corollary 5.4.4]{Monteiro2019} and \eqref{eq:KS-interal-1}.
\end{proof}

\begin{lemma}[linearity and additivity of the integral]
\label{lemma:KS_lin_add}
The Kurzweil-Stieltjes integral in \eqref{eq:KS-interal-1} is linear 
in both its integrand and its integrator. 
Further, for all $c \in (a,b)$ and $f,g \in G([a,b]; \R^d)$
with $f \in BV([a,b]; \R^d)$ or $g \in BV([a,b]; \R^d)$, we have
\[
\int_a^b \langle f,\dd g\rangle = \int_a^c \langle f,\dd g\rangle + \int_c^b \langle f,\dd g\rangle.
\]
\end{lemma}
 \begin{proof}
 This follows from \cite[Theorems 6.2.7, 6.2.9, 6.2.10]{Monteiro2019} and \eqref{eq:KS-interal-1}.
 \end{proof}

\begin{lemma}[constant integrators and integrands]
\label{lem:KS-constant_functions}
If $f, g \in G([a,b]; \R^d)$ holds and 
$c \in \R^d$ is interpreted as a constant function, then we have
\begin{equation*} 
\int_a^b \langle c, \dd g\rangle = \langle c, g(b) - g(a) \rangle 
 \qquad \text{and} \qquad
\int_a^b \langle f, \dd c\rangle = 0. 
\end{equation*}
\end{lemma}
\begin{proof}
The assertions of this lemma follow from \cite[Remark 6.3.1]{Monteiro2019} and \eqref{eq:KS-interal-1}.
\end{proof}

\begin{lemma}[singular integrators and integrands]
\label{lem:KS-point-measures}
Let $f, g\in G([a,b];\R^d)$ and $c \in \R^d$  be given.
Then it holds  
\begin{equation*}
\int_a^b \langle f(s), \dd (c\mathds{1}_{\{t\}}(s)) \rangle
=
\begin{cases}
- \langle f(a), c\rangle & \text{ if } t=a, 
\\
0 &\text{ if } a < t < b,
\\
\langle f(b), c\rangle &\text{ if } t=b,
\end{cases}
\end{equation*}
and
\begin{equation*}
\int_a^b \langle c\mathds{1}_{\{t\}}(s), \dd g(s) \rangle
=
\begin{cases}
\langle c, g(a+)-g(a)\rangle & \text{ if } t=a,
\\
\langle c, g(t+)-g(t-)\rangle &\text{ if } a < t < b,
\\
\langle c, g(b)-g(b-)\rangle&\text{ if } t=b.
\end{cases}
\end{equation*}
\end{lemma}
\begin{proof}
The assertions of the lemma follow from 
\cite[Lemmas 6.3.2, 6.3.3]{Monteiro2019}
and \eqref{eq:KS-interal-1}.
\end{proof}

\begin{lemma}[integrator-integrand estimate]
\label{lem:KS-int-int-estimate}
For all $f \in G([a,b]; \R^d)$ and $g \in BV([a,b]; \R^d)$, it holds 
\[
    \left |
        \int_a^b \langle f , \dd g\rangle
    \right |
    \leq 
    d\, 
    \|f\|_\infty \var(g; [a,b]).
\]
\end{lemma}
\begin{proof}
Using \eqref{eq:KS-interal-1}, \cite[Theorem 6.3.6]{Monteiro2019}, and 
our definitions of $\|\cdot\|_\infty$ and $\var(\cdot; [a,b])$,
we may compute that 
\begin{equation*}
    \left |
        \int_a^b \langle f , \dd g\rangle
    \right |
     \leq
    \sum_{j=1}^d 
        \left |
        \int_a^b f_j \,\dd g_j
    \right |
     \leq
      \sum_{j=1}^d \|f_j\|_\infty\var(g_j;[a,b])
    \leq
    d\, \|f\|_\infty   \var(g; [a,b]). 
\end{equation*}
This establishes the desired inequality.\end{proof}

\begin{lemma}[limit transitions with uniform convergence]
\label{lem:KS-uniform_convergence}
Let $\{g_k\} \subset BV([a,b]; \R^d)$ 
and 
$\{f_k\} \subset G([a,b]; \R^d)$ 
be sequences satisfying 
$\sup_{k \in \N}\var(g_k; [a,b]) < \infty$,
$\lim_{k \to \infty}
    \| g_k - g\|_{\infty}
    =
    0$,
    and 
$
    \lim_{k \to \infty}
    \| f_k - f\|_{\infty}
    =
    0
$
for some $g \in  BV([a,b]; \R^d)$ and $f \in G([a,b]; \R^d)$.
Then it holds 
\[
\lim_{k \to \infty}
\int_a^b \langle f_k , \dd g_k\rangle
=
\int_a^b \langle f, \dd g\rangle.
\]
\end{lemma}
\begin{proof}
    The assertion of this lemma follows 
    immediately from \cite[Theorem 6.8.8]{Monteiro2019}
    and definition \eqref{eq:KS-interal-1}.
\end{proof}

 \begin{theorem}[bounded convergence theorem]
\label{th:KS-bounded_convergence}
Let $g \in BV([a,b]; \R^d)$ be given and 
let $\{f_k\} \subset G([a,b]; \R^d)$ 
be a sequence that satisfies
$\lim_{k \to \infty} f_k(s) = f(s)$ for all $s \in [a,b]
$ and
$
    \sup_{k \in \N}
    \| f_k\|_{\infty}
    < \infty
$
for some $f \in G([a,b];\R^d)$.
Then it holds 
\[
\lim_{k \to \infty}
\int_a^b \langle f_k , \dd g\rangle
=
\int_a^b \langle f, \dd g\rangle.
\]
\end{theorem}
\begin{proof}
This result follows from 
 \cite[Theorem 6.8.13]{Monteiro2019}
 and again the componentwise definition in \eqref{eq:KS-interal-1}.
\end{proof}

\begin{lemma}[integration of limits]
\label{lem:KS-int_limits}
Let $f\in G([a,b]; \R^d)$ and
$g\in CBV([a,b];\R^d)$  be given.
Then it holds  
\begin{equation*}
\int_a^b \langle f, \dd g \rangle
= 
\int_a^b \langle f_+, \dd g \rangle
=
\int_a^b \langle f_-, \dd g \rangle.
\end{equation*}
\end{lemma}
\begin{proof}
The assertion of the lemma follows 
straightforwardly 
from \cref{th:KS-bounded_convergence},
\cref{lemma:KS_lin_add,lem:KS-point-measures}, 
and the fact that regulated functions possess 
at most countably many points of  discontinuity.
\end{proof}

\begin{lemma}[integration over subintervals]
\label{lem:KS-sub_int}
Let $f\in G([a,b]; \R^d)$,
$g\in BV([a,b];\R^d)$,
and $a \leq s <  t \leq b$ be given.
Suppose that $s$ and $t$ are points of continuity of 
$g$. 
Let $J \in \{ (s,t), (s,t], [s,t), [s,t]\}$.
Then it holds  
\begin{equation}
\label{eq:randomeq:727hd83}
\int_a^b \langle \mathds{1}_J f, \dd g \rangle
= 
\int_s^t \langle f, \dd g \rangle
= 
\int_s^t \langle f, \dd g_+ \rangle
=
\int_s^t \langle f, \dd g_- \rangle.
\end{equation}
\end{lemma}
\begin{proof}
The first equality in \eqref{eq:randomeq:727hd83} 
follows from 
\cite[Theorem 6.9.7]{Monteiro2019}
and \eqref{eq:KS-interal-1}.
The second and third one 
are easy consequences of 
\cref{lemma:KS_lin_add,lem:KS-point-measures,lem:KS-uniform_convergence}
and the fact that functions of bounded variation possess 
at most countably many points of discontinuity. 
\end{proof}

 \begin{lemma}[integration by parts formula]
\label{lem:KS-IBP}
Let  $f, g \in G([a,b]; \R^d)$ be given 
and suppose that 
at least one of the functions $f$ and $g$ is an element of $BV([a,b]; \R^d)$.
Then it holds 
\begin{equation*}
\begin{aligned}
   \int_a^b \langle f , \dd g\rangle 
   &=
   \langle f(b), g(b) \rangle - \langle f(a), g(a) \rangle 
   - \int_a^b \langle g , \dd f\rangle
   \\
   &\quad 
   +
   \sum_{a< s \leq b}
   \left \langle f(s) - f(s-), g(s) - g(s-)\right \rangle
   -
   \sum_{a\leq s < b}
   \left \langle f(s+) - f(s), g(s+) - g(s)\right \rangle.
   \end{aligned}
\end{equation*}
\end{lemma}
\begin{proof}
This result is a straightforward consequence of
 \cite[Theorem 6.4.2]{Monteiro2019} and 
 \eqref{eq:KS-interal-1}.
\end{proof}

The following lemma is a useful tool 
in \cref{sec:3}.

\begin{lemma}[criterion for nonnegativity]
\label{lemma:crit_noneg}
Let $f \in G([a,b])$ and $g \in CBV([a,b])$
be given. Suppose that $g$ is nondecreasing on $[a,b]$
and that the following is true:
\begin{equation}
\label{eq:randomeq232738i-s}
    \forall s \in [a,b] \text{ with } f(s) < 0 
    \; \exists \varepsilon_s > 0 
    \; \colon 
    g = \mathrm{const} \text{ on } (s - \varepsilon_s, s + \varepsilon_s) \cap [a,b].
\end{equation}
Then it holds 
\[
\int_a^b f \,\dd g \geq 0.
\]
\end{lemma}
\begin{proof}
Let $f$ and $g$ be as in the assertion of the lemma.
Choose for every $s \in [a,b]$
satisfying $f(s) < 0$ a number $\varepsilon_s \in (0,1)$
as in \eqref{eq:randomeq232738i-s} and define 
$
U := \bigcup_{s \in [a,b], f(s) < 0} (s - \varepsilon_s, s + \varepsilon_s) \subset \R
$.
Then $U$ is an open and bounded set and possesses at most countably many
(necessarily open and disjoint)
connected components.
Further, all of these connected components are nonempty 
open intervals. 
We denote the
components of $U$ with $U_m = (\alpha_m, \beta_m) \subset \R$,
$-\infty < \alpha_m < \beta_m < \infty$,
and the associated finite/countable index set with $M$.
Due to \eqref{eq:randomeq232738i-s} and $g \in CBV([a,b])$,
$g$ is constant on $[\alpha_m, \beta_m] \cap [a,b]$
for all $m \in M$.
Further, we have by construction that 
$\min(0,f) \in G([a,b])$ is zero on $[a,b] \setminus U$.
This implies that
\begin{equation}
\label{eq:lem_nonneg_1}
\min(0, f) = \min(0, f) \mathds{1}_U = \sum_{m \in M} 
\min(0, f)\mathds{1}_{U_m}
\end{equation}
holds in $[a,b]$ and,
due to \cref{lem:KS-constant_functions,lem:KS-sub_int}, that
\begin{equation}
\label{eq:lem_nonneg_2}
\int_{a}^b
\mathds{1}_{U_m}
\min(0,f)\dd g 
=
\int_{\max(a,\alpha_m)}^{\min(b,\beta_m)}
\min(0,f)\dd g 
= 0
\qquad \forall m \in M.
\end{equation}
In combination 
with \cref{lemma:KS_lin_add} 
and 
\cref{th:KS-bounded_convergence},
the equations 
\eqref{eq:lem_nonneg_1} and \eqref{eq:lem_nonneg_2}
yield
\[
\int_a^b \min(0,f) \dd g
=
 \int_a^b \sum_{m \in M} 
\min(0, f)\mathds{1}_{U_m} \dd g
=
\sum_{m \in M} 
 \int_a^b 
\min(0, f)\mathds{1}_{U_m} \dd g
=
0.
\]
As the definition of the Kurzweil-Stieltjes integral
implies that the integral of a 
nonnegative integrand w.r.t.\ a 
nondecreasing integrator is nonnegative, 
see \cite[Chapter 6]{Monteiro2019}, 
we may now conclude that 
\[
\int_a^b f\,\dd g
=
\int_a^b \max(0,f) \dd g
+
\int_a^b \min(0,f) \dd g
=
\int_a^b \max(0,f) \dd g
\geq 0.
\]
This completes the proof. 
\end{proof}

Next, we recall Helly's selection theorem,
which plays an important role in \cref{sec:3}.

\begin{theorem}[Helly selection theorem]
\label{th:HellySelect}
Suppose that $\{g_k\}$ is a bounded sequence in 
$BV([a,b];\R^d)$. Then there exist a 
subsequence $\{g_{k_l}\}$ and a function 
$g \in BV([a,b];\R^d)$ such that 
\[
\lim_{l \to \infty}
g_{k_l}(s)= g(s)\quad \forall s \in [a,b]
\qquad
\text{and}
\qquad
\var(g;[a,b]) \leq \sup_{k \in \N}  \var(g_k;[a,b]).
\]
\end{theorem}

\begin{proof}
The assertion of this theorem follows 
straightforwardly 
from \cite[Theorem 2.7.4]{Monteiro2019}.
\end{proof}

We conclude this subsection by proving two 
convenient auxiliary results for $CBV([a,b]; \R^d)$-functions.

\begin{lemma}[mollification of continuous functions of bounded variation]
\label{lem:strict_mollification}
Let $v \in CBV([a,b]; \R^d)$ be given and 
let $\{\eta_\varepsilon\}_{\varepsilon > 0} \subset C^\infty(\R)$
be standard mollifiers as defined in \cite[Section C.5]{Evans2010}.
Define 
\begin{equation}
\label{eq:mollifier_def}
v_\varepsilon := (\eta_\varepsilon \star \hat v)|_{[a,b]},
~\varepsilon > 0,
\qquad
\text{where}
\qquad
\hat v(s)
:=
\begin{cases}
v(a) & \text{ if } s < a,
\\
v(s) & \text{ if } s \in [a,b],
\\
v(b)& \text{ if } s > b,
\end{cases}
\end{equation}
and where $\star$ denotes a (componentwise) convolution, i.e., 
\[
\eta_\varepsilon \star \hat v\colon \R \to \R^d,
\qquad
(\eta_\varepsilon \star \hat v)(t)
:=
\int_\R \hat v(s) \eta_\varepsilon(t-s) \dd s.
\]
Then $v_\varepsilon \in C^\infty([a,b]; \R^d)$ holds for all $\varepsilon > 0$ and 
we have 
\begin{equation}
\label{eq:mollifier_convergence}
\| v - v_\varepsilon\|_{C([a,b]; \R^d)} \to 0
\qquad
\text{ and}
\qquad
\var(v; [a,b]) \geq \var(v_\varepsilon; [a,b]) \to \var(v; [a,b])
\qquad
\text{ for }\varepsilon \to 0. 
\end{equation}
\end{lemma}
\begin{proof}
That $v_\varepsilon \in C^\infty([a,b]; \R^d)$ holds 
for all $\varepsilon > 0$ and that 
$\| v - v_\varepsilon\|_{C([a,b]; \R^d)} \to 0$ for $\varepsilon \to 0$
follows from \cite[Theorem C.5-7(i),(iii)]{Evans2010}.
To see that  $\var(v_\varepsilon;[a,b]) \leq \var(v; [a,b]) $
holds for all $\varepsilon > 0$, 
suppose that a partition $\Delta$ of the form $a = t_0 \leq t_1 \leq ... \leq t_N = b$, $N \in \N$,
of $[a,b]$ is given. The definition of $v_\varepsilon$ implies
\begin{equation}
\label{eq:randomvardef}
\begin{aligned}
\var_\Delta(v_\varepsilon;[a,b]) &:=
\sum_{j=1}^N 
|v_\varepsilon(t_j) - v_\varepsilon(t_{j - 1})|
=
\sum_{j=1}^N 
\left | \int_\R \eta_\varepsilon(s)   \hat v(t_j - s) -   \eta_\varepsilon(s)   \hat v(t_{j-1} - s) \dd s\right |
\\
&\leq
 \int_\R \eta_\varepsilon(s) \sum_{j=1}^N   \left | \hat v(t_j - s) -  \hat v(t_{j-1} - s) \right | \dd s
\leq
\int_\R \eta_\varepsilon(s) \dd s \var(v; [a,b]) 
=
 \var(v; [a,b]),
\end{aligned}
\end{equation}
where we used that $\hat v$ is constant on $(-\infty, a]$ and $[b, \infty)$,
respectively. Taking the supremum over all $\Delta$ in the above
yields $\var(v_\varepsilon; [a,b]) \leq \var(v; [a,b]) $ for all $\varepsilon > 0$
as desired. 
It remains to prove that $\var(v_\varepsilon; [a,b])$ converges to $\var(v; [a,b])$
for $\varepsilon \to 0$. 
To this end, we note that 
$\var(\cdot; [a,b])\colon C([a,b];\R^d) \to [0, \infty]$
is, by definition, the pointwise supremum of the family of functions 
$\{ \var_\Delta(\cdot; [a,b]) \colon C([a,b];\R^d) \to [0, \infty) \mid \Delta \text{ is a partition of }[a,b]\}$.
Here, $ \var_\Delta(\cdot; [a,b])$ is defined as in \eqref{eq:randomvardef}. 
As the functions $\var_\Delta(\cdot; [a,b]) \colon C([a,b];\R^d) \to [0, \infty)$
are continuous, this implies that $\var(\cdot; [a,b])\colon C([a,b];\R^d) \to [0, \infty]$
is a lower semicontinuous function. Consequently, 
\[
\var(v; [a,b]) 
\geq 
\limsup_{\varepsilon \to 0} 
\var(v_\varepsilon; [a,b])
\geq
\liminf_{\varepsilon \to 0} 
\var(v_\varepsilon; [a,b])
\geq
\var(v; [a,b]).
\]
This completes the proof. 
\end{proof}

\begin{lemma}[{fundamental theorem in $CBV([a,b];\R^d)$}]
\label{lem:CBV-Stampacchia}
Let $v \in CBV([a,b]; \R^d)$ be given and suppose that 
$F\colon \R^d \to \R$ is continuously differentiable. 
Then it holds 
\begin{equation}
\label{eq:FundThCBV}
F(v(b)) - F(v(a))
=
\int_a^b \langle \nabla F (v(s)), \dd v (s) \rangle. 
\end{equation}
\end{lemma}
\begin{proof}
Let $\{v_\varepsilon\}$ be as in \cref{lem:strict_mollification}. 
Then, for each $\varepsilon > 0$, we have 
\begin{equation}
\label{eq:randomeq273h3}
F(v_\varepsilon(b)) - F(v_\varepsilon(a))
=
\int_a^b \langle \nabla F (v_\varepsilon (s)),  \dot v_\varepsilon (s) \rangle \dd s
=
\int_a^b \langle \nabla F (v_\varepsilon(s)), \dd v_\varepsilon (s) \rangle,
\end{equation}
where the second identity follows from \cref{lemma:KS_basics}.
Since $\| v - v_\varepsilon\|_{C([a,b]; \R^d)} \to 0$ holds and since $F$ is $C^1$, 
we have 
$\| \nabla F(v) - \nabla F(v_\varepsilon)\|_{C([a,b]; \R^d)} \to 0$.
Further, we know that $\var(v_\varepsilon; [a,b]) \leq \var(v; [a,b])$
for all $\varepsilon > 0$. In combination
with \cref{lem:KS-uniform_convergence},
this allows us to pass to the limit $\varepsilon \to 0$ in \eqref{eq:randomeq273h3} 
to arrive at \eqref{eq:FundThCBV}.
\end{proof}

\subsection{Preliminaries on the vector play and stop}
\label{subsec:2.5}

As a final preparation for our analysis, we recall 
some results on the vector play and stop. 
Throughout this subsection, we make the following assumptions:

\begin{assumption}[standing assumptions for Section 2.5]\label{ass:sec:2.5}~
\begin{enumerate}[label=\roman*)]
\item $d \in \N$ and $T>0$ are given and fixed;
\item $Z \subset \R^d$ is a full-dimensional convex polyhedron;
\item $\HH = \{(\nu_i, \alpha_i)\}_{i \in I}$ is a standard description of $Z$.
\end{enumerate}
\end{assumption}

\begin{proposition}[unique solvability of the EVI]
\label{prop:unique_solvability_stop_play}
The evolution variational inequality 
\begin{equation*}
\label{eq:EVI}
\tag{V}
\begin{aligned}
& y \in CBV([0,T]; \R^d), &&y(t) \in Z\quad \forall t \in [0, T],\qquad y(0) = y_0,
\\
&\int_0^T \left \langle v - y, \dd (y - u) \right \rangle \geq 0 &&\forall v \in C([0, T]; Z),
\end{aligned} 
\end{equation*}
has a unique solution $y \in CBV([0,T]; \R^d)$ for all 
$(u, y_0) \in CBV([0,T]; \R^d) \times Z$. If $u \in W^{1,1}((0,T); \R^d)$ holds, 
then the solution $y$ of \eqref{eq:EVI} is also the unique solution of the problem 
\begin{equation}
\label{eq:EVI-W11}
\begin{aligned}
& y \in W^{1,1}((0,T); \R^d), &&y(t) \in Z\quad \forall t \in [0, T],\qquad y(0) = y_0,
\\
&\langle \dot y(t) - \dot u(t), v -  y(t) \rangle \geq 0 &&\forall v \in Z\quad \text{for a.a.\ } t \in (0,T).
\end{aligned}
\end{equation}
\end{proposition}

\begin{proof}
The unique solvability of \eqref{eq:EVI} 
follows from \cite[Theorem 4.1]{Krejci1999},
and that
the solution of \eqref{eq:EVI} 
is uniquely characterized by \eqref{eq:EVI-W11} 
in the case $u \in W^{1,1}((0,T); \R^d)$
follows from \cite[Proposition 4.1]{Krejci1999}. 
(Note that the additional assumption $0 \in Z$ in \cite{Krejci1999}
is unimportant since we can 
always satisfy it by means of a
translation.)
\end{proof}

\begin{lemma}[regulated test functions]
\label{lem:EVI_G}
In the situation of \cref{prop:unique_solvability_stop_play},
the solution $y$ of \eqref{eq:EVI}  also satisfies
\begin{equation}
\label{eq:EVI_G}
\int_{s_1}^{s_2}
 \left \langle v - y, \dd (y - u) \right \rangle \geq 0\qquad \forall v \in G([s_1, s_2]; Z)
 \qquad \forall 0 \leq s_1 < s_2 \leq T. 
\end{equation}
\end{lemma}
\begin{proof}
This lemma can be proved completely analogously to 
\cite[Lemma 4.2]{BrokateChristof2023};
see also \cite[Lemma 4.1]{Krejci1999}. 
\end{proof}

\begin{definition}[play and stop operator]
The solution operator $\SS\colon CBV([0,T];\R^d) \times Z \to CBV([0,T];\R^d)$,
$(u, y_0) \mapsto y$, of \eqref{eq:EVI} is called the stop operator.
The  function 
$\PP\colon CBV([0,T];\R^d) \times Z \to CBV([0,T];\R^d)$
defined by 
$\PP(u, y_0) := u - \SS(u, y_0)$ is called the play operator. 
\end{definition}

Note that, due to the above definitions, continuity and differentiability results 
for $\SS$ carry over to $\PP$ and vice versa.
This allows us to focus on $\SS$ in the 
remainder of this paper. Next, we study Lipschitz continuity properties 
of $\SS$. For $W^{1,1}((0,T); \R^d)$-inputs $u$, a Lipschitz estimate 
for $\SS$ with values in $W^{1,1}((0,T); \R^d)$
has been derived in \cite[Theorems 7.1, 7.2]{Krejci1999}.
We extend this result to $CBV([0,T];\R^d)$ with the following lemma:

\begin{lemma}[extension of Lipschitz estimates]
\label{th:CBV-Lipschitz-Extension}
Let $X$ be a normed space, let $D\subset X$ be a set, 
and let $\FF\colon CBV([0,T];\R^d) \times D \to CBV([0,T];\R^d)$ 
be given. Suppose that there exist constants $L_1,L_2>0$ satisfying 
\begin{equation}
\label{eq:F-conti_C}
\begin{aligned}
&\|\FF(u_1, x_1) - \FF(u_2,  x_2)\|_{C([0,T];\R^d)}
\\
&\quad\leq 
L_1
\left (
\|u_1 - u_2\|_{C([0,T];\R^d)}
+
\|x_1 - x_2 \|_X
\right )
\qquad
\forall
(u_1, x_1),(u_2, x_2) \in CBV([0,T];\R^d) \times D
\end{aligned}
\end{equation}
and 
\begin{equation}
\label{eq:F-conti_W11}
\begin{aligned}
&
\var \left ( \FF(u_1, x_1) - \FF(u_2,  x_2); [0,T]\right )
\\
&\qquad\leq 
L_2
\left (
\|\dot u_1 - \dot u_2\|_{L^1((0,T);\R^d)}
+
\|x_1 - x_2 \|_X
\right )
\qquad
\forall
(u_1, x_1),(u_2, x_2) \in C^\infty([0,T];\R^d) \times D.
\end{aligned}
\end{equation}
Then it holds 
\begin{equation*}
\begin{aligned}
&\|\FF(u_1, x_1) - \FF(u_2,  x_2)\|_{CBV([0,T];\R^d)}
\\
&\quad \leq 
(L_1 + L_2)
\left (
\|u_1 - u_2\|_{CBV([0,T];\R^d)}
+
\|x_1 - x_2 \|_X
\right )
\qquad
\forall
(u_1, x_1),(u_2, x_2) \in CBV([0,T];\R^d) \times D.
\end{aligned}
\end{equation*}
\end{lemma}
\begin{proof}
Let $(u_1, x_1),(u_2, x_2) \in CBV([0,T];\R^d) \times D$ be given,
and
let $u_{1,\varepsilon}, u_{2,\varepsilon}, (u_1 - u_2)_\varepsilon 
\in C^\infty([0,T];\R^d)$, $\varepsilon > 0$,
be mollifications of $u_1$, $u_2$, and $u_1 - u_2$, respectively, 
as constructed in \cref{lem:strict_mollification}. 
Then it follows from \eqref{eq:mollifier_def}, \eqref{eq:mollifier_convergence}, \eqref{eq:F-conti_C}, and the smoothness of $u_{1,\varepsilon} - u_{2,\varepsilon}$ that 
\begin{equation}
\label{eq:randomeq3838389-42}
\FF(u_{1,\varepsilon}, x_1) - \FF(u_{2,\varepsilon}, x_2)
\to 
\FF(u_1, x_1) - \FF(u_2, x_2)
\text{ for } \varepsilon \to 0 \text{ in } C([0,T];\R^d)
\end{equation}
and 
\begin{equation}
\label{eq:randomeq3838389-43}
\left \| \dot u_{1,\varepsilon} - \dot u_{2,\varepsilon} \right \|_{L^1((0,T);\R^d)}
=
\var(u_{1,\varepsilon} -  u_{2,\varepsilon}; [0,T])
=
\var( (u_1 - u_2)_\varepsilon; [0,T])
\leq 
\var( u_1 - u_2 ; [0,T])\quad \forall \varepsilon > 0.
\end{equation}
Since $C([0,T];\R^d) \to [0, \infty]$, $v \mapsto \var(v; [0,T])$,  
is lower semicontinuous 
w.r.t.\ strong convergence in $C([0,T]; \R^d)$
(see the proof of \cref{lem:strict_mollification}), 
the results in \eqref{eq:F-conti_W11}, \eqref{eq:randomeq3838389-42}, and \eqref{eq:randomeq3838389-43} imply
\begin{equation*}
\begin{aligned}
\var(\FF(u_1, x_1) - \FF(u_2, x_2) ; [0,T])
&\leq
\liminf_{\varepsilon \to 0}
\var(\FF(u_{1,\varepsilon}, x_1) - \FF(u_{2,\varepsilon}, x_2) ; [0,T])
\\
&\leq
L_2 \limsup_{\varepsilon \to 0}
\left (
\left \| \dot u_{1,\varepsilon} - \dot u_{2,\varepsilon} \right \|_{L^1((0,T);\R^d)} 
+
\|x_1 - x_2\|_X
\right )
\\
&=
L_2
\limsup_{\varepsilon \to 0}
\left (
\var( (u_1 - u_2)_\varepsilon; [0,T])
+
\|x_1 - x_2\|_X
\right )
\\
&\leq
L_2
\left (
\var(u_1 - u_2; [0,T])
+
\|x_1 - x_2\|_X
\right ).
\end{aligned}
\end{equation*}
By combining the above with \eqref{eq:F-conti_C},
the assertion follows. 
\end{proof}

For the stop operator $\SS$, we now obtain:

\begin{theorem}[{CBV-Lipschitz continuity of $\SS$}]
\label{th:CBV-Lipschitz}
There exists a constant $L>0$ depending only on $Z$ such that
\begin{equation}
\label{eq:BV_Lipschitz}
\begin{aligned}
&\|\SS(u, y_0) - \SS(\tilde u, \tilde y_0)\|_{CBV([0,T];\R^d)}
\\
&\qquad
\leq 
L
\left (
\|u - \tilde u\|_{CBV([0,T];\R^d)}
+
|y_0 - \tilde y_0|
\right )
\qquad \forall (u, y_0),(\tilde u, \tilde y_0) \in CBV([0,T]; \R^d) \times Z.
\end{aligned}
\end{equation}
\end{theorem}

\begin{proof}
As in the proof of \cref{lem:ZZstandard}, 
we may assume w.l.o.g.\ that  $0 \in \interior(Z)$ holds and that $\HH = \{(\nu_i, \alpha_i)\}_{i \in I}$
satisfies $\alpha_i >0$ for all $i \in I$. For this situation, we obtain 
from \cite[Theorems 7.1, 7.2]{Krejci1999} that there exist constants $L_1,L_2>0$ 
depending only on $Z$ such that 
\begin{equation}
\label{eq:C_S_continuity}
\begin{aligned}
&\|\SS(u, y_0) - \SS(\tilde u, \tilde y_0)\|_{C([0,T];\R^d)}
\\
&\qquad\leq 
L_1
\left (
\|u - \tilde u\|_{C([0,T];\R^d)}
+
|y_0 - \tilde y_0|
\right )
\qquad \forall (u, y_0),(\tilde u, \tilde y_0) \in CBV([0,T]; \R^d) \times Z
\end{aligned}
\end{equation}
and 
\begin{equation*}
\begin{aligned}
&\left \|\dot \SS(u, y_0) - \dot \SS(\tilde u, \tilde y_0)\right \|_{L^1((0,T);\R^d)}
\\
&\qquad\leq 
L_2
\left ( 
\left \|
\dot u - \dot{\tilde u} \right \|_{L^1((0,T);\R^d)}
+
|y_0 - \tilde y_0|
\right )
\qquad \forall (u, y_0),(\tilde u, \tilde y_0) \in W^{1,1}((0,T); \R^d) \times Z.
\end{aligned}
\end{equation*}
By invoking \cref{th:CBV-Lipschitz-Extension}, the assertion now follows 
immediately. 
\end{proof}

Next, we establish a result on the 
local behavior of the play operator 
that generalizes the 
local monotonicity property 
documented for the scalar case in 
\cite[Lemma 4.3]{BrokateChristof2023}.
Recall that, for every point $x \in Z$, we have 
$\AA(x) := \{i \in I \mid \langle \nu_i, x\rangle =  \alpha_i\}$
and 
$\ZZ(x) := \{z \in \R^d \mid \langle \nu_i, z\rangle \leq 0~\forall i \in \AA(x) \}$. 

\begin{lemma}[local monotone decomposition]
\label{lemma:monotone_play}
Assume that $\HH = \{(\nu_i, \alpha_i)\}_{i \in I}$ is regular.
Let $t \in [0,T]$ be fixed
and let $(u, y_0) \in CBV([0,T]; \R^d) \times Z$ be given.
Set $y := \SS(u, y_0)$ and $w := \PP(u, y_0)$.
Then there exist a number 
$\varepsilon > 0$ and, for this $\varepsilon$, unique
 $\lambda_i \in CBV([t - \varepsilon, t + \varepsilon] \cap [0,T])$,
$i \in \AA(y(t))$, such that
\begin{equation}
\label{eq:mon_play}
\AA(y(s)) \subset \AA(y(t))
\quad
\text{and}
\quad
w(s) = 
w(\max(0,t - \varepsilon))
+
\sum_{i \in \AA(y(t))} \lambda_i(s)\nu_i \quad \forall s \in [t - \varepsilon, t + \varepsilon] \cap [0,T].
\end{equation}
Further, the functions $\lambda_i$, $i \in \AA(y(t))$,
are each nonnegative and nondecreasing
on $[t - \varepsilon, t + \varepsilon] \cap [0,T]$. 
\end{lemma}

\begin{proof}
Since $\HH$ is regular
and $y(t) \in Z$, 
we obtain from \cref{lem:ZZproperties}\ref{lem:ZZproperties:iv}
that there exist unique $e_i \in \ZZ(y(t))$, $i \in \AA(y(t))$, such that \eqref{eq:e_i-properties}
and \eqref{eq:e_i-equivalence} 
hold for $x:= y(t)$. From \cref{lem:ZZproperties}\ref{lem:ZZproperties:i} and the continuity of $y$,
it follows further that there exist $\varepsilon , \tau > 0$ such that 
$\AA(y(r)) \subset \AA(y(t))$ and 
$y(r) + \tau z  \in Z$ 
 for all $r \in [t - \varepsilon, t + \varepsilon] \cap [0,T]$
and all $z \in B_1(0) \cap \ZZ(y(t))$.
This allows us 
to consider test functions 
of the form 
$v\colon [t - \varepsilon, s] \cap [0,T] \to \R^d$, $r \mapsto y(r) + \tau z$,
with arbitrary $z \in  B_1(0) \cap \ZZ(y(t))$ 
and $s \in (t - \varepsilon, t + \varepsilon] \cap (0,T]$
in \cref{lem:EVI_G}
and to obtain 
\begin{equation}
\label{eq:randomeq288dni}
\begin{aligned}
0 &\leq 
\int_{\max(0, t - \varepsilon)}^{s}
\langle (y(r) + \tau z) - y(r)  , \dd( y(r) - u(r))\rangle
=
- \tau 
\int_{\max(0, t - \varepsilon)}^{s}
\langle z , \dd w(r) \rangle
\\
&=
\tau 
 \left \langle z ,  w(\max(0, t - \varepsilon)) - w(s) \right \rangle 
 \qquad \forall z \in B_1(0) \cap \ZZ(y(t))\qquad \forall s \in (t - \varepsilon, t + \varepsilon] \cap (0,T]
\end{aligned}
\end{equation}
and, thus, 
\begin{equation}
\label{eq:randomeq2273whdid3}
 \left \langle z ,    w(\max(0, t - \varepsilon)) - w(s) \right \rangle \geq 0
 \qquad \forall z \in \ZZ(y(t))\qquad \forall s \in (t - \varepsilon, t + \varepsilon] \cap (0,T].
\end{equation}
Here, we have used \cref{lem:KS-constant_functions}. 
Note that \eqref{eq:randomeq2273whdid3} 
is also true for $s = \max(0, t-\varepsilon)$. 
Due to \cref{lem:normal_cone_polyhedral}, the definitions of $\ZZ(y(t))$ and $\ZZ(y(t))^\circ$,
and the regularity of $\HH$, 
this allows us to conclude that, for each $ s \in [t - \varepsilon, t + \varepsilon] \cap [0,T]$,
there exist unique values $\lambda_i(s) \geq 0$, $i \in \AA(y(t))$, 
such that 
\begin{equation}
\label{eq:randomeq282z8e3h93hoi}
w(s) 
=
w(\max(0, t - \varepsilon)) 
+
\sum_{i \in \AA(y(t))} \lambda_i(s) \nu_i\qquad \forall s \in [t - \varepsilon, t + \varepsilon] \cap [0,T].
\end{equation}
If we take the scalar product with $e_j$, $j \in \AA(y(t))$,
in \eqref{eq:randomeq282z8e3h93hoi}, then it follows that 
\[
\lambda_j(s) 
=
\frac{\langle e_j, w(s) \rangle
-
\langle e_j, w(\max(0, t - \varepsilon))  \rangle}{\langle e_j,  \nu_j \rangle }
\qquad 
\forall s \in [t - \varepsilon, t + \varepsilon] \cap [0,T].
\]
Thus,
$\lambda_j \in CBV([t - \varepsilon, t + \varepsilon] \cap [0,T])$
for all $j\in \AA(y(t))$ as claimed. 
It remains to prove that the functions $\lambda_i$ are nondecreasing. 
To this end, we note 
that---completely analogously to \eqref{eq:randomeq2273whdid3}---we obtain that
\begin{equation*}
 \left \langle z ,    w(s_1) - w(s_2) \right \rangle \geq 0
 \qquad \forall z \in \ZZ(y(t))\qquad \forall \max(0, t-\varepsilon) \leq s_1 < s_2 \leq \min(T, t + \varepsilon).
\end{equation*}
Using \eqref{eq:randomeq282z8e3h93hoi} in the above yields 
\begin{equation*}
 \left \langle z ,   
 \sum_{i \in \AA(y(t))} (\lambda_i(s_1) - \lambda_i(s_2))\nu_i \right \rangle \geq 0
 \qquad \forall z \in \ZZ(y(t))\qquad \forall \max(0, t-\varepsilon) \leq s_1 < s_2 \leq \min(T, t + \varepsilon)
\end{equation*}
and, after choosing $z = e_j \in \ZZ(y(t))$,
\begin{equation*}
 \left \langle e_j ,   
 \nu_j \right \rangle (\lambda_j(s_1) - \lambda_j(s_2))\geq 0
\qquad \forall \max(0, t-\varepsilon) \leq s_1 < s_2 \leq \min(T, t + \varepsilon).
\end{equation*}
Since $\left \langle e_j ,   
 \nu_j \right \rangle < 0$ by \eqref{eq:e_i-properties}, 
it follows that $\lambda_j(s_1) \leq \lambda_j(s_2)$
for all  $\max(0, t-\varepsilon) \leq s_1 < s_2 \leq \min(T, t + \varepsilon)$
and all $j \in \AA(y(t))$. This shows that the functions $\lambda_j$
are indeed nondecreasing and completes the proof.
\end{proof}

Note that, if one changes $\varepsilon$ 
in \eqref{eq:mon_play},
then the functions $\lambda_i$ are
typically shifted by constant values due to the term $w(\max(0,t-\veps))$.
The multiplier maps  $\lambda_i$ in \cref{lemma:monotone_play} are 
thus indeed only uniquely determined 
for fixed $\varepsilon $. The next lemma builds
on \cref{lemma:monotone_play} and is needed for the discussion 
of the notion of \emph{criticality} in \cref{subsec:3.2}.

\begin{lemma}[multiplier integrals and tangent vectors]
\label{lem:mult_ints_tang}
Assume that $\HH = \{(\nu_i, \alpha_i)\}_{i \in I}$
 is regular. 
Let $t \in [0,T]$ 
and $(u, y_0) \in CBV([0,T]; \R^d) \times Z$ be given.
Let $y$, $w$, $\varepsilon$, and 
 $\lambda_i$,
$i \in \AA(y(t))$, be as in \cref{lemma:monotone_play}.
Then, for all 
$\max(0, t-\varepsilon) \leq s_1 < s_2 \leq \min(T, t + \varepsilon)$
and 
all $z \in G([s_1, s_2]; \R^d)$
satisfying $z(r) \in \ZZ(y(r))$ for all $r \in [s_1, s_2]$,
it holds
\[
\int_{s_1}^{s_2}\left \langle z(r), \nu_i  \right \rangle \dd  \lambda_i(r) \leq 0
\qquad
\forall i \in \AA(y(t)).
\]
\end{lemma}
\begin{proof}
Let $e_i \in \ZZ(y(t))$, $i \in \AA(y(t))$,
again denote the vectors that satisfy  \eqref{eq:e_i-properties}
and \eqref{eq:e_i-equivalence} for $x:= y(t)$.
Due to \cref{lem:KS-sub_int}, it suffices to prove the 
assertion of the lemma for the case $s_1 = \max(0, t-\varepsilon)$,
$s_2 = \min(T, t + \varepsilon)$. 
Suppose that a function
 $z \in G( [t - \varepsilon, t + \varepsilon] \cap [0,T]; \R^d)$
satisfying $z(r) \in \ZZ(y(r))$ for all $r \in  [t - \varepsilon, t + \varepsilon] \cap [0,T]$
is given and fix $j \in \AA(y(t))$. Assume further that 
$s \in [t - \varepsilon, t + \varepsilon] \cap [0,T]$ is a point 
such that $\left \langle z(s), \nu_j  \right \rangle > 0$ holds. 
Then $z(s) \in \ZZ(y(s))$ implies $j \not \in \AA(y(s))$
and we obtain from the inclusion 
$\AA(y(s)) \subset \AA(y(t))$ in \eqref{eq:mon_play}
and \eqref{eq:e_i-properties} that $\pm e_j  \in \ZZ(y(s))$.
Due to \cref{lem:ZZproperties}\ref{lem:ZZproperties:i}
and the continuity of $y$, it follows that there exist $\tau, \gamma > 0$
such that $y(r) \pm \tau e_j  \in Z $
holds for all $r \in [s - \gamma, s + \gamma] \cap [0,T]$. 
Analogously to \eqref{eq:randomeq288dni},
we may now deduce from \cref{lem:EVI_G},  \eqref{eq:mon_play},
\eqref{eq:e_i-properties}, and \cref{lem:KS-constant_functions}
that 
\begin{equation*}
\begin{aligned}
0 &\leq \int_{\max(0, t - \varepsilon, s - \gamma)}^{\min(T, t + \varepsilon, s + \gamma)}
\langle (y(r) \pm \tau e_j) - y(r)  , \dd( y(r) - u(r))\rangle
\\
& =
- \tau 
\int_{\max(0, t - \varepsilon, s - \gamma)}^{\min(T, t + \varepsilon, s + \gamma)}
\langle \pm e_j , \dd w(r) \rangle
\\
& =
- \tau 
\int_{\max(0, t - \varepsilon, s - \gamma)}^{\min(T, t + \varepsilon, s + \gamma)}
\langle \pm e_j ,\nu_j \rangle  \dd   \lambda_j(r)
 \\
 & 
 = \pm \langle  e_j ,\nu_j \rangle \tau \left (
 \lambda_j(\max(0, t - \varepsilon, s - \gamma))
 -
\lambda_j(\min(T, t + \varepsilon, s + \gamma))
 \right ).
\end{aligned}
\end{equation*}
As $\langle  e_j ,\nu_j \rangle \neq 0$ holds by \eqref{eq:e_i-properties} and since 
$\lambda_j$ is nondecreasing by \cref{lemma:monotone_play},
it follows that $\lambda_j = \mathrm{const}$ holds 
on $ (s - \gamma, s + \gamma) \cap [\max(0, t - \varepsilon), \min(T, t + \varepsilon)]$.
As $s \in [t - \varepsilon, t + \varepsilon] \cap [0,T]$ was an arbitrary 
point satisfying $\left \langle z(s), \nu_j  \right \rangle > 0$,
we may now invoke \cref{lemma:crit_noneg}
(with $g:= \lambda_j$,
$f := - \left \langle z, \nu_j  \right \rangle$,
$a := \max(0, t - \varepsilon)$,
and
$b:=\min(T, t + \varepsilon)$)
to obtain that 
\[
\int_{\max(0, t - \varepsilon)}^{\min(T, t + \varepsilon)}\left \langle z(r), \nu_j  \right \rangle \dd  \lambda_j(r) \leq 0.
\]
As $j \in \AA(y(t))$ was arbitrary, this completes the proof. 
\end{proof}

As an immediate consequence of \cref{lem:mult_ints_tang},
we obtain the following refined version of \cref{lem:EVI_G}.

\begin{lemma}[regulated tangential test functions]
\label{lem:nonneg_tang_play}
Assume that $\HH = \{(\nu_i, \alpha_i)\}_{i \in I}$
is regular 
and let $(u, y_0) \in CBV([0,T]; \R^d) \times Z$ be given.
Define $y := \SS(u, y_0)$.
Then it holds 
\begin{equation}
\label{eq:EVI_G_ZZ}
0 \leq s_1 < s_2 \leq T,
~~
z \in G([s_1, s_2]; \R^d),
~~
z(r) \in \ZZ(y(r))
~\forall r \in [s_1, s_2]
\qquad \Rightarrow
\qquad
\int_{s_1}^{s_2}
 \left \langle z, \dd (y - u) \right \rangle \geq 0.
\end{equation}
\end{lemma}
\begin{proof}
Due to \cref{lem:KS-sub_int},
it suffices to consider the case $s_1 = 0$, $s_2 = T$.
So let us assume that a function $z \in G([0, T]; \R^d)$
satisfying $z(r) \in  \ZZ(y(r))$
for all $r \in [0,T]$ is given. 
Choose for all $t \in [0,T]$ a number $\varepsilon_t > 0$
with the properties in \cref{lemma:monotone_play}.
Since $[0,T]$ is compact, we 
can find $t_m$, $m=1,...,M$, $M \in \N$, such that 
the collection of intervals  $(t_m - \varepsilon_{t_m}, t_m + \varepsilon_{t_m})$,
$m=1,...,M$, covers $[0,T]$.
Let $\{\psi_m\}_{m=1}^M$ be a 
smooth partition of unity on $[0,T]$ subordinate 
to this  cover in the sense of \cite{Evans2010},
i.e., a collection of functions  satisfying 
\begin{equation}
\label{eq:POU_properties}
\begin{gathered}
\psi_m \in C_c^\infty(\R),\quad 0 \leq \psi_m(r) \leq 1~\forall r \in \R,
\quad 
\supp(\psi_m) \subset (t_m - \varepsilon_{t_m}, t_m + \varepsilon_{t_m}) \quad\forall m=1,...,M,
\\
\sum_{m=1}^M \psi_m(r) = 1~\forall r \in [0,T].
\end{gathered}
\end{equation}
Define $z_m := \psi_m z \in G([0,T]; \R^d)$, $m=1,...,M$.
Then we clearly have 
\begin{equation}
\label{eq:randomeq_72he9}
\int_{0}^{T}
 \left \langle z, \dd (y - u) \right \rangle
 =
 -  \int_{0}^{T}
\left \langle z, \dd w \right \rangle
 =
 - \sum_{m=1}^M
  \int_{0}^{T}
 \left \langle z_m, \dd w \right \rangle,
\end{equation}
where we again write $w := u - y =  \PP(u, y_0)$ for the play.
Note that, for each of the integrals in the sum on the right-hand side of
\eqref{eq:randomeq_72he9}, we have 
\begin{equation}
\label{eq:randomeq_72he9233}
   \int_{0}^{T}
 \left \langle z_m, \dd w \right \rangle
 =
  \int_{\max(0, t_m - \varepsilon_{t_m})}^{\min(T, t_m + \varepsilon_{t_m})}
 \left \langle z_m, \dd w \right \rangle
 =
 \sum_{i \in \AA(y(t_m))}
  \int_{\max(0, t_m - \varepsilon_{t_m})}^{\min(T, t_m + \varepsilon_{t_m})}
 \left \langle z_m, \nu_i \right \rangle \dd  \lambda_i^m
\end{equation}
by \cref{lem:KS-sub_int}, where
 $\lambda_i^m \in CBV([t_m - \varepsilon_{t_m}, t_m + \varepsilon_{t_m}] \cap [0,T])$,
 $i \in \AA(y(t_m))$,
 denote the functions from \cref{lemma:monotone_play}
 associated with $t_m$. 
 As the cone property of the linearization cone and the 
 sign of $\psi_m$ imply that 
 $z_m(r) \in \ZZ(y(r))$ holds for all 
 $r \in [t_m - \varepsilon_{t_m}, t_m + \varepsilon_{t_m}] \cap [0,T]$,
 we obtain from \cref{lem:mult_ints_tang} that all of the 
 integrals in the sum on the right-hand side of \eqref{eq:randomeq_72he9233} 
 are nonpositive. In combination with \eqref{eq:randomeq_72he9},
 this establishes the assertion of the lemma and completes the proof. 
\end{proof}
Note that, while trivial in the $W^{1,1}((0,T);\R^d)$-case due to 
\eqref{eq:EVI-W11}, the implication \eqref{eq:EVI_G_ZZ} is not 
obvious for $CBV([0,T];\R^d)$-solutions $y$ since 
for this type of regularity one cannot argue in a pointwise manner 
and since 
for functions $z \in G([0,T]; \R^d)$
satisfying $z(r) \in \ZZ(y(r))$
for all $r \in [0,T]$ there does not necessarily exist 
a number $\tau > 0$ satisfying $y(r) + \tau z(r) \in Z$
for all $r \in [0,T]$. In the proof of \cref{lem:nonneg_tang_play}, 
we have overcome this problem by means of
the decomposition in \cref{lemma:monotone_play} 
and \cref{lemma:crit_noneg}.

\section{Proof of \ref{main_sum:item:I}: directional differentiability in the non-obtuse case}
\label{sec:3}

With the necessary preliminaries in place, we can turn 
our attention to the 
main results of this paper, 
namely, points \ref{main_sum:item:I} and \ref{main_sum:item:II}
of \cref{th:main_summary}. 
The present section is devoted to 
the proof of
\ref{main_sum:item:I}, i.e., 
the pointwise directional differentiability of 
the vectorial stop operator
$\SS\colon CBV([0,T];\R^d) \times Z \to CBV([0,T];\R^d)$
in the case of a
full-dimensional 
non-obtuse convex polyhedral 
admissible set $Z$. The structure of this section is as follows:
\Cref{subsec:3.1} introduces basic concepts and results
that are needed for our analysis. This subsection 
also contains our standing assumptions
and a rough overview of our method of proof. 
\Cref{subsec:3.2} is concerned with the 
notion of \emph{criticality} which expresses that 
directional derivatives of $\SS$ have to satisfy the 
inequality on the right-hand side of \eqref{eq:EVI_G_ZZ}
with equality. In 
\cref{subsec:3.3,subsec:3.4}, we identify how the 
directional derivatives jump at their points of discontinuity. 
At this point of our analysis, 
the non-obtuseness of the admissible set $Z$
enters in a crucial manner;
see the proofs of 
\cref{lem_left_jump_proto,lem:3.18-2,lemma:the_little_box}.
\Cref{subsec:3.5} introduces the 
concept of \emph{temporal polyhedricity}---an approximation 
property that is  needed to derive the system \eqref{eq:char_sys_1} 
that uniquely characterizes the directional derivatives of $\SS$. 
In \cref{subsec:3.6}, we 
finally 
put all of the pieces of our proof together to  
arrive at  \ref{main_sum:item:I} as desired;
see \cref{th:dir_diff_final_contra}. 
\Cref{subsec:3.7} concludes our discussion of 
the non-obtuse case by demonstrating 
that point \ref{main_sum:item:I} 
of \cref{th:main_summary}
indeed 
contains the 
 directional differentiability 
and characterization results 
proved for the case $d=1$ and $Z = [-r,r]$, $r>0$,
in \cite{Brokate2015,Brokate2021} as a special case.

\subsection{Standing assumptions and basic idea of the proof}
\label{subsec:3.1}

In all of  \cref{sec:3}, 
we consider the following situation,
unless explicitly stated otherwise. 

\begin{assumption}[standing assumptions for \cref{sec:3}]\label{ass:sec:3}~
\begin{enumerate}[label=\roman*)]
\item $d \in \N$ and $T>0$ are given and fixed;
\item $Z \subset \R^d$ is a 
full-dimensional 
non-obtuse convex polyhedron;
\item $\HH = \{(\nu_i, \alpha_i)\}_{i \in I}$  is a
standard description of $Z$;
\item $y_0 \in Z$ is a given initial value and $u \in CBV([0,T]; \R^d)$ is a given input function;
\item $y:= \SS(u,y_0)$ is the stop associated with $(u,y_0)$ and  $w:= \PP(u,y_0)$ is the play associated with $(u,y_0)$;
\item $h_0 \in \ZZ(y_0)$ is a given perturbation of $y_0$
and $h \in CBV([0,T]; \R^d)$ is a given perturbation of $u$;
\item $\tau_0 > 0$ is a number satisfying  $y_0 + \tau_0 h_0 \in Z$.
\end{enumerate}
\end{assumption}

Recall that, by
\cref{def:properties_descript,lem:standard_description,lem:descriptions:iii},
the standard description $\HH$ in \cref{ass:sec:3} is unique up to permutations, 
regular, satisfies $\langle \nu_i, \nu_j\rangle \leq 0$  for all $i,j \in I$ with  $i \neq j$, 
and possesses the property 
\eqref{eq:faces}. We introduce:

\begin{definition}[difference quotients]
We denote by $\{\delta_\tau\} \subset CBV([0,T]; \R^d)$ the family of difference quotients 
\begin{equation}
\label{eq:DQ-def}
\delta_\tau := \frac{\SS(u + \tau h, y_0 + \tau h_0) - \SS(u, y_0)}{\tau},\qquad 0 < \tau < \tau_0.
\end{equation}
\end{definition}

As a direct consequence of \cref{th:CBV-Lipschitz},
\cref{lem:EVI_G},
and the EVI \eqref{eq:EVI}, we obtain:

\begin{lemma}[EVI for the difference quotients]
\label{lem:EVI_diff_quot}
The family $\{\delta_\tau\}$ is bounded in 
$CBV([0,T]; \R^d)$.
Further, for every $\tau \in (0, \tau_0)$, $\delta_\tau$
satisfies 
\begin{equation}
\label{eq:DQ-EVI}
\begin{aligned}
&\delta_\tau(t) \in \frac{1}{\tau}\left (Z - y(t) \right ) \forall t \in [0,T],
\qquad
\delta_\tau(0) = h_0,
\\
&\int_{s_1}^{s_2}
 \left \langle z - \delta_\tau, \dd (\delta_\tau - h) \right \rangle 
 - \frac{1}{\tau}
 \int_{s_1}^{s_2}
 \left \langle z - \delta_\tau, \dd w \right \rangle 
 \geq 0
\\
&\hspace{0.4cm}\forall z \in G([s_1, s_2]; \R^d)\colon 
 z(t) \in \frac{1}{\tau}\left (Z - y(t) \right )~\forall t \in [s_1, s_2]\qquad\forall\, 0 \leq s_1 < s_2\leq T,
\end{aligned}
\end{equation}
and 
\begin{equation}
\label{eq:proto_criticality}
\begin{aligned}
\int_{s_1}^{s_2}
 \left \langle \delta_\tau, \dd w \right \rangle \leq 0
 \qquad \forall\,0 \leq s_1 < s_2\leq T.
\end{aligned}
\end{equation}
\end{lemma}
\begin{proof}
Since $Z$ is full-dimensional,
\cref{th:CBV-Lipschitz} is applicable
and the boundedness of $\{\delta_\tau\}$ in $CBV([0,T]; \R^d)$
follows straightforwardly from \eqref{eq:DQ-def}. 
The same is true for the first line in \eqref{eq:DQ-EVI}. 
To obtain the variational inequality in \eqref{eq:DQ-EVI},
we note that  \eqref{eq:DQ-def} implies 
$\SS(u + \tau h, y_0 + \tau h_0)  = y + \tau \delta_\tau$ for all $0 < \tau < \tau_0$.
By plugging this identity into the EVI \eqref{eq:EVI_G} satisfied 
by $\SS(u + \tau h, y_0 + \tau h_0)$, by 
choosing test functions of the form 
$v = y + \tau z$ with $z$ as in \eqref{eq:DQ-EVI},
and by using that $w = \PP(u, y_0) = u - \SS(u, y_0) = u - y$, we obtain that 
\begin{equation*}
\begin{aligned}
0 &\leq \int_{s_1}^{s_2}
 \left \langle v - \SS(u + \tau h, y_0 + \tau h_0) , \dd ( \SS(u + \tau h, y_0 + \tau h_0) - u - \tau h) \right \rangle 
 \\
 &= 
  \tau^2
 \int_{s_1}^{s_2}
 \left \langle z - \delta_\tau , \dd ( \delta_\tau -  h) \right \rangle  
 - \tau
 \int_{s_1}^{s_2}
 \left \langle z - \delta_\tau , \dd w  \right \rangle 
\end{aligned}
\end{equation*}
holds for all $z$, $s_1$, and $s_2$ as in \eqref{eq:DQ-EVI}. 
Dividing by $\tau^2$ now yields the desired variational inequality. 
It remains to prove \eqref{eq:proto_criticality}. 
To this end, we choose $v = \SS(u + \tau h, y_0 + \tau h_0)$ in \eqref{eq:EVI_G}
for $y$. In combination with the definition of $w$, this gives
\[
0 \leq
\int_{s_1}^{s_2}
 \left \langle \SS(u + \tau h, y_0 + \tau h_0) - y, \dd (y - u) \right \rangle 
 =
-  \tau 
\int_{s_1}^{s_2}
 \left \langle \delta_\tau , \dd w \right \rangle  
 \qquad \forall 0 \leq s_1 < s_2 \leq T. 
\]
The above establishes \eqref{eq:proto_criticality} and completes the proof. 
\end{proof}

A central object of the analysis of this section
is the following set:

\begin{definition}[limits of difference quotients]
\label{def:DD}
We define 
\begin{equation}
\label{eq:DD_def}
\DD :=
\left \{
\left.
\delta \in BV([0,T]; \R^d)
~\right |~
\exists \{\tau_k\} \subset (0, \tau_0)
\text{ s.t.\ } \tau_k \to 0
\text{ and }
\delta_{\tau_k}(t) \to \delta(t)~\forall t \in [0,T] 
\text{ for }k \to \infty
\right \}.
\end{equation}
\end{definition}

Note that, from the boundedness of 
$\{\delta_\tau\}$ in $CBV([0,T]; \R^d)$
and \cref{th:HellySelect}, we obtain:

\begin{lemma}[nontriviality of $\DD$]
\label{lem:DD_nonempty}
The set $\DD$ is nonempty. 
\end{lemma}

The main goal 
of the next five subsections
is to prove that
the set $\DD$ is a singleton. 
If this is established, then 
we obtain from a trivial contradiction argument that 
$\SS$ is directionally differentiable 
at $(u, y_0)$ in direction $(h, h_0)$ 
in a pointwise manner 
and point \ref{main_sum:item:I} of \cref{th:main_summary}
follows;
see \cref{th:dir_diff_final_contra}. 
To show that $|\DD| = 1$ holds,
we will successively establish 
more and more 
properties of the elements of $\DD$ 
until only one element remains. 
More precisely,
the plan for \cref{subsec:3.2,subsec:3.3,subsec:3.4,subsec:3.5,subsec:3.6} is as follows:
\begin{itemize}
\item 
In \cref{subsec:3.2},
we use  \eqref{eq:DQ-EVI} and \eqref{eq:proto_criticality}
to show 
that all $\delta \in \DD$ satisfy the inequality 
on the right-hand side of \eqref{eq:EVI_G_ZZ}
with equality. 
By means of \cref{lemma:monotone_play,lem:mult_ints_tang} 
and results from \cref{subsec:2.4}, we 
then transform 
the resulting notion of \emph{criticality} to obtain information about the pointwise 
behavior of the left-/right-limit functions 
of the elements of $\DD$;
see \cref{th_pointwise_criticality}.

\item In \cref{subsec:3.3,subsec:3.4},
we identify precisely how the 
elements of $\DD$ jump at their points of 
discontinuity. We remark that this step of the proof 
is the most intricate one as the 
identification of the jump conditions 
has to be achieved solely based on the 
study of the difference quotients $\delta_\tau$
which are elements of the space
$CBV([0,T];\R^d)$ and, thus, do not possess
any jumps themselves. We overcome this difficulty by exploiting several auxiliary EVIs and 
vectorial Stampacchia truncation 
arguments that
rely crucially on the non-obtuseness of $Z$ and 
the properties in 
\cref{lem:nonobtuse_properties};
see \cref{th:left_jumps,th:jump_right}.

\item 
In \cref{subsec:3.5},
we combine the information about the pointwise 
behavior of the elements of 
$\DD$ obtained in  \cref{subsec:3.2,subsec:3.3,subsec:3.4} 
to define 
a \emph{critical cone}
that contains all right-limit functions 
$\delta_+$ of elements $\delta$ of the set $\DD$.
The main feature of this critical 
cone is that 
it possesses a 
\emph{temporal polyhedricity} property. 
Roughly speaking,
this property makes it possible 
to approximate critical directions with 
critical radial directions which are admissible in 
\eqref{eq:DQ-EVI} and 
for which the passage to the 
limit $\tau \to 0$ can be performed in 
the EVI satisfied by $\delta_\tau$. 
For the main result on this topic, 
see \cref{theorem:tempoly}. 
\item 
In \cref{subsec:3.6}, we 
carry out the limit transition $\tau \to 0$
in \eqref{eq:DQ-EVI} 
to show that all elements $\delta$ of $\DD$
are solutions of the system 
\eqref{eq:char_sys_1}. Using the results 
from \cref{subsec:2.4},
we then show that this system can have at most one 
solution, thus proving that $\DD$ is indeed a singleton. 
By putting all of 
the results 
of \cref{subsec:3.2,subsec:3.3,subsec:3.4,subsec:3.5,subsec:3.6}
together, statement \ref{main_sum:item:I} 
of \cref{th:main_summary} then follows as desired.
\end{itemize}

We remark that the  above argumentation 
roughly follows well-known approaches to the sensitivity analysis of elliptic systems;
see \cite{ChristofPhd2018,ChristofWachsmuth2020}.
Due to the loss of regularity that 
occurs when passing to the limit 
$\tau \to 0$ with the difference 
quotients $\delta_\tau$
(namely, the transition from 
$ CBV([0,T];\R^d)$ to $BV([0,T];\R^d)$),
however,
our analysis is completely beyond the scope of standard results. In particular the techniques that we develop
for the analysis of the pointwise properties 
of the elements of the set $\DD$ seem to be entirely new. 

\subsection{Criticality}
\label{subsec:3.2}

The first thing that we establish about the elements of 
the set $\DD$ is that they are \emph{critical}.

\begin{definition}[criticality]
\label{def:critical_direction}
A function $z \in G([0,T];\R^d)$ is called a 
\emph{critical direction} or simply \emph{critical} if 
\begin{equation}
\label{eq:critical_def}
z(t) \in \ZZ(y(t))\quad\forall t \in [0,T]
\qquad
\text{and}
\qquad
\int_{0}^{T}
\left \langle z, \dd w \right \rangle = 0.
 \end{equation}
\end{definition}
\begin{lemma}[criticality of limits of difference quotients]
\label{lem:delta_critical}
If $\delta \in \DD$ holds, then $\delta$ is a
critical direction.
\end{lemma}
\begin{proof}
Let $\delta \in \DD$ be given and let $\{\tau_k\} \subset (0, \tau_0)$ be a sequence for $\delta$
as in the definition 
of $\DD$. Then it follows from 
\eqref{eq:DQ-EVI} that $\delta_{\tau_k}(t) \in K_\rad(y(t);Z)$ for all $k$
and all $t \in [0,T]$. Due to \cref{lem:ZZproperties}\ref{lem:ZZproperties:iii}, 
this yields  $\delta_{\tau_k}(t) \in \ZZ(y(t))$ and, 
since $\ZZ(y(t))$ is closed, $\delta_{\tau_k}(t)  \to \delta(t) \in \ZZ(y(t)) $
for all $t \in [0,T]$.
From 
\eqref{eq:proto_criticality}, 
\eqref{eq:DQ-EVI} with $z \equiv 0$,
\cref{lem:CBV-Stampacchia},
and \cref{lem:KS-int-int-estimate}, we further obtain that 
\begin{equation*}
\begin{aligned}
0 &\geq 
\int_{0}^{T}\left \langle \delta_{\tau_k}, \dd w \right \rangle
\geq 
\tau_k
\int_{0}^{T}
 \left \langle  \delta_{\tau_k}, \dd (\delta_{\tau_k} - h) \right \rangle 
\geq
 \tau_k\left ( 
 \frac{1}{2}|\delta_{\tau_k}(T)|^2 - \frac{1}{2}|\delta_{\tau_k}(0)|^2
 - d\,\|\delta_{\tau_k}\|_{C([0,T]; \R^d)} \var(h; [0,T])
 \right ).
\end{aligned}
\end{equation*}
Due to the boundedness of 
$\{\delta_\tau\}$  in 
$CBV([0,T]; \R^d)$, the right-hand side of the last inequality 
goes to zero for $k \to \infty$. In combination with 
\cref{th:KS-bounded_convergence}
and again the boundedness of $\{\delta_\tau\}\subset CBV([0,T]; \R^d)$,
this yields
\[
0 \geq 
\lim_{k\to \infty}
\int_{0}^{T}\left \langle \delta_{\tau_k}, \dd w \right \rangle
=
\int_{0}^{T}\left \langle \delta, \dd w \right \rangle
\geq 0. 
\]
This completes the proof.
\end{proof}
We remark that the concept of criticality is well known 
in the field of sensitivity analysis of elliptic 
variational inequalities of the first and the second kind.
Analogously to 
the conditions in \eqref{eq:critical_def},
in the elliptic setting, it characterizes 
tangential directions that are contained in the kernel 
of the 
multiplier that is associated with the solution of 
the variational inequality under consideration.
For details on this topic
and 
additional relations to the field of 
second-order necessary and sufficient optimality conditions 
for optimization problems, we refer the reader to 
\cite{Harder2017,ChristofPhd2018,Christof2018SSC,BonnansShapiro2000,Haraux1977,Mignot1976} and the references therein.
A feature that distinguishes \eqref{eq:critical_def} 
from its elliptic counterpart is that the 
evolutionary nature of \eqref{eq:EVI} makes it possible 
to localize the property of being an element of 
the kernel of the multiplier in time. 
Indeed, from \cref{lem:nonneg_tang_play}, 
we readily obtain:

\begin{lemma}[criticality on subintervals]
\label{lem:crit_subinterval}
A function  $z \in G([0,T];\R^d)$ is critical if and 
only if 
\begin{equation}
\label{eq:randomeq236364-43}
z(t) \in \ZZ(y(t))\quad\forall t \in [0,T]
\qquad
\text{and}
\qquad
\int_{s_1}^{s_2}
\left \langle z, \dd w \right \rangle = 0
\quad\forall\,0 \leq s_1 < s_2 \leq T.
\end{equation}
\end{lemma}
\begin{proof}
The implication ``$\Leftarrow$'' is trivial. 
The implication ``$\Rightarrow$'' follows 
from \cref{lemma:KS_lin_add,lem:nonneg_tang_play}.
\end{proof}

If we combine the localization property
in \cref{lem:crit_subinterval} with 
the decomposition
in \cref{lemma:monotone_play} and 
\cref{lem:mult_ints_tang}, then we get:

\begin{lemma}[criticality and limits]
\label{lem:nice_lemma}
Let $t \in [0,T]$ be given. Let 
$\varepsilon > 0$ and
$\lambda_i \in CBV([t - \varepsilon, t + \varepsilon] \cap [0,T])$,
$i \in \AA(y(t))$, be as in \cref{lemma:monotone_play}. 
Then, for all critical directions 
$z \in G([0,T];\R^d)$, the following is true.
\begin{enumerate}[label=\roman*)]
\item\label{nice_lemma:i} If $t < T$ holds, then,
for every $i \in \AA(y(t))$, there exists $0 < \tilde \varepsilon \leq \min(\varepsilon, T - t)$
such that  
\begin{equation}
\label{eq:item:i:3.9}
	\int_{t}^{t + \tilde \varepsilon} \langle z(t+), \nu_i \rangle \dd\lambda_i(s) = 0. 
\end{equation}
\item\label{nice_lemma:ii} If $t > 0$ holds, then,
for every $i \in \AA(y(t))$, there exists $0 < \tilde \varepsilon \leq \min(\varepsilon, t)$
such that  
\begin{equation}
\label{eq:item:ii:3.9}
\int_{t - \tilde \varepsilon}^{t} \langle z(t-), \nu_i \rangle \dd\lambda_i(s) = 0.
\end{equation}
\end{enumerate}
\end{lemma}

\begin{proof}
We only prove \ref{nice_lemma:i}. 
The proof of \ref{nice_lemma:ii} is analogous.
Let $t \in [0,T)$, 
$\varepsilon > 0$, 
and 
$\lambda_i \in CBV([t - \varepsilon, t + \varepsilon] \cap [0,T])$,
$i \in \AA(y(t))$, be as in the assertion of the lemma
and let 
$z\in G([0,T];\R^d)$ be a critical direction. 
From \eqref{eq:randomeq236364-43},
\eqref{eq:mon_play},  \eqref{eq:KS-interal-1},
and \cref{lemma:KS_lin_add},
we obtain that 
\begin{equation}
\label{eq:randomeq32673zdh}
0
=
\int_{s_1}^{s_2}
\left \langle z, \dd w \right \rangle
=
\int_{s_1}^{s_2}\left \langle z(s), \dd 
\left (\sum_{i \in \AA(y(t))} \lambda_i(s)\nu_i \right )\right \rangle
=
\sum_{i \in \AA(y(t))}
\int_{s_1}^{s_2}\left \langle z(s), \nu_i  \right \rangle \dd  \lambda_i(s)
\end{equation}
holds for all $\max(0, t-\varepsilon) \leq s_1 < s_2 \leq \min(T, t + \varepsilon)$. 
Since $z(s)$ is an element of $\ZZ(y(s))$ for all $s \in [0,T]$,
\cref{lem:mult_ints_tang}
implies that the integrals in the sum on the right-hand side of  \eqref{eq:randomeq32673zdh} are nonpositive each. 
Thus, 
\begin{equation}
\label{eq:randomeq38383-33}
\int_{s_1}^{s_2}\left \langle z(s), \nu_i  \right \rangle \dd  \lambda_i(s)
=
0
\qquad
\forall\,
\max(0, t-\varepsilon) \leq s_1 < s_2 \leq \min(T, t + \varepsilon)
\qquad
\forall 
i \in \AA(y(t)).
\end{equation}
Let us now suppose 
that a number $\tilde \varepsilon > 0$ with the 
properties in \ref{nice_lemma:i} does not exist for some $i \in \AA(y(t))$. 
Then it follows from 
\cref{lem:KS-constant_functions,lem:KS-sub_int,lem:KS-int-int-estimate},
\eqref{eq:randomeq38383-33},
and the monotonicity and continuity of $\lambda_i$ that 
\begin{equation*}
\begin{aligned}
0
&< 
\left | \langle z(t+), \nu_i \rangle\right |
 \left (
\lambda_i(t + \tilde \varepsilon )
- \lambda_i(t)
\right )
=
\left |
\int_{t}^{t + \tilde \varepsilon} \langle z(t+), \nu_i \rangle \dd\lambda_i(s)
\right |
=
\left |
\int_{t}^{t + \tilde \varepsilon} \langle z(t+) -z(s),  \nu_i \rangle \dd\lambda_i(s)
\right |
\\
&\leq
\sup_{t < s \leq t + \tilde \varepsilon}
 | z(t+) -z(s)| 
 \var(\lambda_i ;[t, t + \tilde \varepsilon])
=
\sup_{t < s \leq t + \tilde \varepsilon}
 | z(t+) -z(s)| 
\left (
\lambda_i(t + \tilde \varepsilon )
- \lambda_i(t)
\right )\quad \forall\, 0 < \tilde \varepsilon \leq \min(\varepsilon, T - t).
\end{aligned}
\end{equation*}
Since we necessarily have 
$\lambda_i(t + \tilde \varepsilon )
- \lambda_i(t) > 0$ in the above, it follows that 
\[
0
<
\left | \langle z(t+), \nu_i \rangle\right |
\leq
\sup_{t < s \leq t + \tilde \varepsilon}
 | z(t+) -z(s)| 
 \qquad \forall\, 0 < \tilde \varepsilon \leq \min(\varepsilon, T - t).
\]
This produces a contradiction 
with the definition of $z(t+)$.
Thus, a number $\tilde \varepsilon$ with the 
properties in \ref{nice_lemma:i} has to exist
for all $i \in \AA(y(t))$ and the proof is complete. 
\end{proof}

By
exploiting 
\cref{lem:KS-constant_functions}
and the monotonicity of the multiplier maps 
$\lambda_i$, we can transform
the identities
in \eqref{eq:item:i:3.9}
and \eqref{eq:item:ii:3.9} into conditions 
on the left and the right limits
of $z$. To be able to formulate the resulting conditions concisely, we introduce:

\begin{definition}[directionally active, strictly active, and biactive indices] 
\label{def:temporal_indices}
Let $t \in [0,T]$ be given. Let 
$\varepsilon > 0$ and
$\lambda_i \in CBV([t - \varepsilon, t + \varepsilon] \cap [0,T])$,
$i \in \AA(y(t))$, be as in \cref{lemma:monotone_play}. We define
\begin{equation*}
\begin{aligned}
\AA^+(t) &:= 
\begin{cases}
\left \{
i \in \AA(y(t))
\mid
\exists \{s_k\} \subset (t, T]
\text{ s.t. } s_k \to t \text{ and } i \in \AA(y(s_k))~\forall k \in \N
\right \} & \text{if } t \in [0,T),
\\
\emptyset &\text{if } t = T,
\end{cases}
\\
\AA^-(t) &:= 
\begin{cases}
\left \{
i \in \AA(y(t))
\mid
\exists \{s_k\} \subset [0, t)
\text{ s.t. } s_k \to t \text{ and } i \in \AA(y(s_k))~\forall k \in \N
\right \} & \text{if } t \in (0,T],
\\
\emptyset &\text{if } t = 0,
\end{cases} 
\\
\AA_\strict^+(t) &:= 
\begin{cases}
\{i \in \AA(y(t)) \mid \lambda_i|_{[t, t+s]} \neq \mathrm{const}~ \forall s \in (0, \min(T-t,\varepsilon))\} & \text{if } t \in [0,T),
\\
\emptyset &\text{if } t = T,
\end{cases} 
\\
\AA_\strict^-(t) &:= 
\begin{cases}
\{i \in \AA(y(t)) \mid \lambda_i|_{[t-s, t]} \neq \mathrm{const}~ \forall s \in (0, \min(t,\varepsilon))\} & \text{if } t \in (0,T],
\\
\emptyset &\text{if } t = 0,
\end{cases} 
\\
\BB^+(t) &:=
 \AA(y(t)) \setminus \AA_\strict^+(t),
\\
\BB^-(t) &:=
 \AA(y(t)) \setminus \AA_\strict^-(t),
\\
V^+(t) &:= \spann(\{ \nu_i \mid i \in \AA_\strict^+(t)\})^\perp, &&
\\
V^-(t) &:= \spann(\{ \nu_i \mid i \in \AA_\strict^-(t)\})^\perp. &&
\end{aligned}
\end{equation*}
\end{definition}

Note that the index sets in \cref{def:temporal_indices} 
do not depend on the choice of $\varepsilon$;
see the comments after \cref{lemma:monotone_play}.
Using the quantities introduced in \cref{def:temporal_indices},
we can reformulate 
the integral notion of criticality in \eqref{eq:critical_def}
in an equivalent, purely pointwise manner.

\begin{theorem}[pointwise reformulation of criticality]
\label{th_pointwise_criticality}
A function $z \in G([0,T];\R^d)$ is critical if and only if 
\begin{equation}
\label{eq:pointwise_criticality}
z(t) \in \ZZ(y(t)),
\quad
z(t+) \in V^+(t),
\quad
\text{and}
\quad
z(t-) \in V^-(t)
\qquad \forall t \in [0,T].
\end{equation}
\end{theorem}
\begin{proof}
We begin with the implication ``$\Rightarrow$''.
Suppose that a critical direction 
$z \in G([0,T];\R^d)$ is given and let 
$t \in [0,T]$ be fixed. 
Then $z(t) \in \ZZ(y(t))$ holds by \eqref{eq:critical_def}.
In the case $t = T$, we further have 
$\AA_\strict^+(t) = \emptyset$ and 
$V^+(t) = \R^d$ by 
\cref{def:temporal_indices},
and the inclusion 
$z(t+) \in V^+(t)$ is trivial. 
If $t \in [0,T)$ holds, then 
we obtain from 
\cref{lem:KS-constant_functions,lem:nice_lemma} 
that, for every $i \in \AA_\strict^+(t)$, 
there exists a number 
$0 < \tilde \varepsilon \leq \min(\varepsilon, T - t)$
satisfying 
\[
0
=
\int_{t}^{t + \tilde \varepsilon} \langle z(t+), \nu_i \rangle \dd\lambda_i(s)
=
\langle z(t+), \nu_i \rangle
 \left (
\lambda_i(t + \tilde \varepsilon )
- \lambda_i(t)
\right ),
\]
where $\varepsilon > 0$ and
$\lambda_i \in CBV([t - \varepsilon, t + \varepsilon] \cap [0,T])$ 
are again as in \cref{lemma:monotone_play}. 
By the definition of $ \AA_\strict^+(t)$ and the monotonicity of $\lambda_i$, 
we have 
$\lambda_i(t + \tilde \varepsilon )
- \lambda_i(t) > 0$. 
Thus, $\langle z(t+), \nu_i \rangle = 0$
for all $i \in \AA_\strict^+(t)$
and it follows that 
$z(t+) \in V^+(t)$ holds as desired.
To check the inclusion
$z(t-) \in V^-(t)$, we can proceed along the exact same lines.
This shows that every critical direction $z$ satisfies 
\eqref{eq:pointwise_criticality}
and completes the first half of the proof. 

It remains to prove the implication ``$\Leftarrow$''.
To this end, let us assume that 
a function $z \in G([0,T];\R^d)$ 
satisfying \eqref{eq:pointwise_criticality}
is given. That $z(t) \in \ZZ(y(t))$ holds for all 
$t\in [0,T]$ in this case 
as required in \eqref{eq:critical_def} is obvious. 
To establish the integral identity in \eqref{eq:critical_def},
let us assume that an arbitrarily small 
$\gamma > 0$ is given. 
As $z$  
possesses left and right limits at all points 
$t \in [0,T]$ and due to 
\cref{lemma:monotone_play} and \cref{def:temporal_indices}, for every $t \in [0,T]$,
we can find a number $\varepsilon_t > 0$ such that 
the step function 
\[
v_t\colon [0,T] \to \R^d,
\qquad
v_t(s)
:=
\begin{cases}
z(t-) & \text{ if } s \in [0,t),
\\
z(t) & \text{ if } s = t,
\\
z(t+) & \text{ if } s \in (t,T],
\end{cases}
\]
satisfies 
$\| z - v_t\|_{\infty, [t-\varepsilon_t, t + \varepsilon_t] \cap [0,T]} \leq \gamma$, such that the 
assertions of \cref{lemma:monotone_play}
hold for $t$ on 
$[t-\varepsilon_t, t + \varepsilon_t] \cap [0,T]$ with functions 
$\lambda_i^t$, $i \in \AA(y(t))$,
and such that 
$\lambda_i^t = \mathrm{const}$ holds on $[t-\varepsilon_t, t]\cap [0,T]$ for all $i \in \BB^-(t)$
and
$\lambda_i^t = \mathrm{const}$ holds on $[t, t+ \varepsilon_t]\cap [0,T]$ for all $i \in \BB^+(t)$. 
As $[0,T]$ is compact,
we can cover it by finitely many 
of the intervals $(t - \varepsilon_t, t + \varepsilon_t)$.
We denote the times $t$ of this cover and 
their $\varepsilon_t$, $v_t$, and $\lambda_i^t$ by 
$0\leq t_1 \leq ... \leq t_M \leq T$
and $\varepsilon_m$, $v_m$, $\lambda_i^m$,
$i \in \AA(y(t_m))$, $m=1,...,M$, 
$M \in \N$,
respectively. Let $\{\psi_m\}_{m=1}^M$ be a 
smooth partition of unity on $[0,T]$ subordinate 
to the cover 
$\{(t_m - \varepsilon_m, t_m + \varepsilon_m)\}_{m=1}^M$, 
i.e., a collection of functions 
satisfying  \eqref{eq:POU_properties}.
We define 
\[
v(s) := \sum_{m=1}^M \psi_m(s) v_{m}(s),\qquad s \in [0,T].
\]
Then we clearly have 
$v \in G([0,T]; \R^d)$ and, by construction and 
\cref{lemma:KS_lin_add,lem:KS-constant_functions,lem:KS-sub_int}, 
\begin{equation*}
\begin{aligned}
\int_0^T \langle v, \dd w \rangle
&=
\sum_{m=1}^M
\int_{\max(0, t_m - \varepsilon_m)}^{\min(T, t_m + \varepsilon_m)} \langle \psi_m(s) v_{m}(s), \dd w(s) \rangle
\\
&=
\sum_{m=1}^M
\int_{\max(0, t_m - \varepsilon_m)}^{\min(T, t_m + \varepsilon_m)} 
\left \langle \psi_m(s) v_{m}(s), \dd
\left ( \sum_{i \in \AA(y(t_m))} \lambda_i^m(s)\nu_i \right )\right \rangle
\\
&=
\sum_{m\colon t_m \neq 0} 
\sum_{i \in \AA(y(t_m))}
\int_{\max(0, t_m - \varepsilon_m)}^{t_m} \langle \psi_m(s) z(t_m-), \nu_i \rangle \dd
\lambda_i^m(s) 
\\
&\qquad\qquad
+
\sum_{m \colon t_m \neq T}
\sum_{i \in \AA(y(t_m))}
\int_{t_m}^{\min(T, t_m + \varepsilon_m)} \langle \psi_m(s) z(t_m+), \nu_i \rangle \dd
\lambda_i^m(s) 
\\
&=
\sum_{m\colon t_m \neq 0} 
\sum_{i \in \AA_\strict^-(t_m)}
\int_{\max(0, t_m - \varepsilon_m)}^{t_m} \langle \psi_m(s) z(t_m-), \nu_i \rangle \dd
\lambda_i^m(s) 
\\
&\qquad\qquad
+
\sum_{m\colon t_m \neq T} 
\sum_{i \in \AA_\strict^+(t_m)}
\int_{t_m}^{\min(T, t_m + \varepsilon_m)} \langle \psi_m(s) z(t_m+), \nu_i \rangle \dd
\lambda_i^m(s) 
\\
&=0.
\end{aligned}
\end{equation*}
Here, we have used the inclusions
$z(t_m +) \in V^+(t_m)$
and 
$z(t_m-) \in V^-(t_m)$.
On the other hand, we also have 
\[
\sup_{s \in [0,T]}
|v(s) - z(s)|
\leq
\sup_{s \in [0,T]}
\sum_{m=1}^M
 \psi_m(s) |v_m(s) - z(s)|
\leq
\sup_{s \in [0,T]}
\sum_{m=1}^M \psi_m(s) \gamma
= \gamma
\]
and, thus, 
\[
\left |
\int_0^T \langle z, \dd w \rangle
\right |
=
\left |
\int_0^T \langle z - v, \dd w \rangle
\right |
\leq
d \gamma \var(w;[0,T])
\]
by \cref{lem:KS-int-int-estimate}.
Since $\gamma>0$ was arbitrary, the above
shows that $z$ satisfies the 
integral identity in 
\eqref{eq:critical_def} and, 
as a consequence, that $z$ is critical.
This completes the proof. 
\end{proof}

The pointwise reformulation 
of the notion of criticality in 
\cref{th_pointwise_criticality}
will be an essential ingredient for the 
analysis of the concept of 
\emph{temporal polyhedricity}
in \cref{subsec:3.5} and,
ultimately, the derivation of 
the system 
\eqref{eq:char_sys_1} that uniquely 
characterizes the directional derivatives of
$\SS$ in \ref{main_sum:item:I}. Before we can approach this topic, 
however, we require more information 
about the pointwise behavior of the 
elements of the set $\DD$. More precisely, 
we have to study in detail 
how the elements of $\DD$
jump at their points of discontinuity.
As we will see in \cref{subsec:3.5},
analyzing the discontinuities 
of the elements of the set 
$\DD$ will allow us to show that 
 not only 
all elements $\delta$ of $\DD$
are critical directions,
but also all right-limit functions
$\delta_+$ for $\delta \in \DD$.
Note that this is not obvious at all 
since it is a priori completely 
unclear whether $\delta(t+) \in \ZZ(y(t))$
holds for all $t \in [0,T]$. 
To prepare the proof of the inclusion $\delta(t+) \in \ZZ(y(t))$ 
and our study of the 
jump directions of the 
elements of $\DD$, we prove:
\begin{lemma}[tangentiality of left and right limits]
\label{lemma:rl-active-indices}
Let $z \in G([0,T];\R^d)$ be critical 
and let $t \in [0,T]$ be given.
\begin{enumerate}[label=\roman*)]
\item\label{lemma:rl-active-indices:i} If $t < T$ holds, then there exists $0 < \varepsilon < T - t$ such that 
\[
	 z(t+) \in \ZZ(y(s))\qquad\forall s \in (t, t+\varepsilon].
\]
\item\label{lemma:rl-active-indices:ii} If $t > 0$ holds, then there exists $0 < \varepsilon < t$ such that 
\[
	 z(t-) \in \ZZ(y(s))\qquad\forall s \in [t-\varepsilon, t).
\]
\end{enumerate}
\end{lemma}
\begin{proof}
We prove \ref{lemma:rl-active-indices:i}.
The proof of \ref{lemma:rl-active-indices:ii} is analogous. 
Let $z$ be critical and let $t \in [0,T)$ be given.
A simple contradiction argument shows that
we can find 
 $0 < \varepsilon < T - t$ 
 such that $\AA(y(s)) \subset \AA^+(t)$ holds for all $s \in (t, t+\varepsilon]$. 
 If $i \in \AA^+(t)$ holds and $\{s_k\}$ is a sequence for this $i$ as in the definition 
 of $\AA^+(t)$, then we further obtain from \eqref{eq:critical_def},
 the definition of $z(t+)$, and $ (t, T] \ni s_k \to t$ that 
$0 \geq \langle \nu_i, z(s_k) \rangle \to \langle \nu_i, z(t+) \rangle$
for $k \to \infty$.
 Thus, $\langle \nu_i, z(t+) \rangle \leq 0$ for all $i \in \AA^+(t)$ 
 and, consequently, 
 $\langle \nu_i, z(t+) \rangle \leq 0$ for all $i \in \AA(y(s))$ and $s \in (t, t+\varepsilon]$.
\end{proof}

\subsection{Identification of  the jumps from the left}
\label{subsec:3.3}

We now turn our attention to the jumps $\delta(t) - \delta(t-)$
of the elements $\delta$ of the set $\DD$. To analyze these jumps from the left, 
we consider the following situation:

\begin{assumption}[standing assumptions for \cref{subsec:3.3}]\label{ass:subsec:3.3}~
\begin{enumerate}[label=\roman*)]
\item $\delta \in \DD$ is arbitrary but fixed;
\item $t \in (0,T]$ is arbitrary but fixed;
\item $\gamma > 0$ is an arbitrary but fixed given small number;
\item\label{ass:subsec:3.3:itemiv}
 $\varepsilon > 0$ is chosen such that all of the following conditions hold:
\begin{enumerate}
\item $\varepsilon < t$ (trivially achievable);
\item $\II(y(t)) \subset \II(y(s))$ for all $s \in [t-\varepsilon, t]$ (achievable by continuity of $y$);
\item $\delta(t-) \in \ZZ(y(s))$ for all $s \in [t-\varepsilon, t)$ (achievable by \cref{lem:delta_critical,lemma:rl-active-indices});
\item the assertions in \cref{lemma:monotone_play} hold 
with functions $\lambda_i \in CBV([t - \varepsilon, t + \varepsilon] \cap [0,T])$,
$i \in \AA(y(t))$;
\item $\lambda_i = \mathrm{const}$ holds on $[t-\varepsilon, t]$ for all $i \in \BB^-(t)$
(achievable by \cref{def:temporal_indices});
\item\label{ass:subsec:3.3:itemiv:f}
$\|  h  -  h(t) \|_{CBV([t-\varepsilon, t];\R^d)} + |\delta(t - \varepsilon) - \delta(t-)|\leq  \gamma/L$,
where $ L >0$ is as in \eqref{eq:BV_Lipschitz} (trivially achievable);
\end{enumerate}
\item $\{\tau_k\}$ is a sequence as 
in the definition of $\DD$ for $\delta$
with associated difference quotients 
$\delta_k := \delta_{\tau_k}$ and $y(t-\varepsilon) + \tau_k \delta(t-) \in Z$
holds 
for all $k$ (achievable due to  $ \delta(t-) \in \ZZ(y(t-\varepsilon))$
and \cref{lem:ZZproperties}\ref{lem:ZZproperties:iii}).
\end{enumerate}
\end{assumption}

Note that, 
in \cref{ass:subsec:3.3}\ref{ass:subsec:3.3:itemiv:f}, the $CBV$-regularity of $h$
enters crucially. 
To study the jump $\delta(t) - \delta(t-)$, we consider
auxiliary variational inequalities of the following type:
\begin{equation*}
\label{eq:EVI_aux_left}
\tag{AUX$_1$}
\begin{aligned}
&
\zeta_{k} \in CBV([t-\varepsilon, t]; \R^d),
\qquad
\zeta_k(s) \in \frac{1}{\tau_k}\left (Z - y(s) \right ) \forall s \in [t-\varepsilon, t],
\qquad
\zeta_k(t-\varepsilon) = \zeta_0,
\\
&\int_{t-\varepsilon}^{t}
 \left \langle z - \zeta_k, \dd  \zeta_k \right \rangle 
 - \frac{1}{\tau_k}
 \int_{t-\varepsilon}^{t}
 \left \langle z - \zeta_k, \dd w \right \rangle 
 \geq 0
\qquad \forall z \in G([t-\varepsilon, t]; \R^d)\colon 
 z(s) \in \frac{1}{\tau_k}\left (Z - y(s) \right )~\forall s \in [t-\varepsilon, t].
\end{aligned}
\end{equation*}

\begin{lemma}[unique solvability of \eqref{eq:EVI_aux_left}]
\label{lem:aux1_solv}
The EVI
\eqref{eq:EVI_aux_left} has a unique 
solution $\zeta_k \in CBV([t-\varepsilon, t]; \R^d)$ for all $\zeta_0$
satisfying $y(t-\varepsilon) + \tau_k \zeta_0 \in Z$
and all $k$.
Further, for all such $\zeta_0$ and $k$, it holds
\begin{equation}
\label{eq:randomeq36737dhu}
\| \delta_k  -  \zeta_k \|_{CBV([t-\varepsilon, t];\R^d)} \leq
L
\left (
\|  h  -  h(t) \|_{CBV([t-\varepsilon, t];\R^d)}
  +
|\delta_k(t-\varepsilon) - \zeta_0|
\right ),
\end{equation}
where $L$ is the Lipschitz constant in \eqref{eq:BV_Lipschitz} associated with $Z$.
\end{lemma}

\begin{proof}
To see that 
\eqref{eq:EVI_aux_left} admits a unique solution $\zeta_k \in CBV([t-\varepsilon, t]; \R^d)$, 
we note that,
by performing the steps 
leading to \eqref{eq:DQ-EVI} in reverse 
and by exploiting \cref{lem:KS-constant_functions},
\eqref{eq:EVI_aux_left} can be recast as 
\begin{equation*}
\begin{aligned}
&
y + \tau_k \zeta_{k} \in CBV([t-\varepsilon, t]; \R^d),
\qquad
y(s) + \tau_k\zeta_k(s) \in Z~\forall s \in [t-\varepsilon, t],
\\
&y(t - \varepsilon )+ \tau_k\zeta_k(t-\varepsilon) = y(t - \varepsilon )+ \tau_k \zeta_0,
\\
&\int_{t-\varepsilon}^{t}
 \left \langle v(s) - y(s) - \tau_k \zeta_k(s), 
 \dd \big ( y(s) + \tau_k \zeta_k(s) - u(s) - \tau_k h(t) \big )\right \rangle 
 \geq 0
\qquad \forall v \in G([t-\varepsilon, t]; Z).
\end{aligned}
\end{equation*}
The above problem has a unique solution $y + \tau_k \zeta_{k} \in CBV([t-\varepsilon, t]; \R^d)$
for all $\zeta_0 $ satisfying 
$y(t - \varepsilon )+ \tau_k \zeta_0 \in Z$
by the existence result for the 
stop
on $[t-\varepsilon, t]$ with forcing term $u + \tau_k h(t)$
and initial value 
$y(t - \varepsilon )+ \tau_k \zeta_0 \in Z$; cf.\ \cref{prop:unique_solvability_stop_play,lem:EVI_G}.
This proves the unique solvability of \eqref{eq:EVI_aux_left}.
By considering the EVI satisfied by $y + \tau_k \delta_k$ on $[t-\varepsilon, t]$
and by invoking \eqref{eq:BV_Lipschitz}, 
we further obtain that 
\begin{equation}
\label{eq:randomeq3636478398-2}
\begin{aligned}
\|(y + \tau_k \delta_k) - (y + \tau_k \zeta_k)\|_{CBV([t-\varepsilon, t];\R^d)}
\leq 
L
&\left (
\|(u + \tau_k h) - (u + \tau_k h(t))\|_{CBV([t-\varepsilon, t];\R^d)}
\right.
\\
&\left. +
|(y(t - \varepsilon )+ \tau_k \delta_k(t-\varepsilon)) - (y(t - \varepsilon )+ \tau_k \zeta_0)|
\right ).
\end{aligned}
\end{equation}
The above implies \eqref{eq:randomeq36737dhu} as desired and 
completes the proof. 
\end{proof}

The main advantage of the auxiliary EVI \eqref{eq:EVI_aux_left} is
 that, in contrast to the variational inequality \eqref{eq:DQ-EVI} satisfied by 
the difference quotient $\delta_k$, it is independent of
the perturbation $h$
of the forcing term. This property makes it possible 
to derive formulas for the value $\zeta_k(t)$, as the following lemma shows. 

\begin{lemma}[terminal condition for \eqref{eq:EVI_aux_left}]
\label{lem_left_jump_proto}
For all large enough $k$,
the solution $\zeta_k$ of the EVI  \eqref{eq:EVI_aux_left}
with initial value  $\zeta_0 := \delta(t-)$
satisfies
$
\zeta_k(t) = \pi_{\ZZ(y(t))}(\delta(t-))
$.
\end{lemma}
\begin{proof}
That  
$y(t-\varepsilon) + \tau_k  \delta(t-) \in Z$
holds 
follows from \cref{ass:subsec:3.3}.
Choosing $\delta(t-)$ as $\zeta_0 $ in \eqref{eq:EVI_aux_left}
is thus allowed 
by  \cref{lem:aux1_solv}.
Note that 
$y(t-\varepsilon) + \tau_k \delta(t-) \in Z$
also implies 
$y(t-\varepsilon) + \tau_k \pi_{\ZZ(y(t))^\circ}(\delta(t-))  \in Z$
by \cref{lem:nonobtuse_properties}\ref{lem:nonobtuse_properties:i}  and \eqref{eq:Z-Zcirc-identity+hom}.
Thus, we may also consider \eqref{eq:EVI_aux_left}
with $\zeta_0$ chosen as $\pi_{\ZZ(y(t))^\circ}(\delta(t-))$.

We begin with the latter case, i.e., 
\eqref{eq:EVI_aux_left} with $\zeta_0 := \pi_{\ZZ(y(t))^\circ}(\delta(t-))$,
and denote the associated solution of \eqref{eq:EVI_aux_left}
by $\xi_k$. Define 
$z  := \pi_{\ZZ(y(t))^\circ}(\xi_k ) \in G([t-\varepsilon, t]; \R^d)$.
Since 
$y(s) + \tau_k\xi_k(s) \in Z$ holds for all 
$s \in [t-\varepsilon, t]$, 
we again obtain from  \cref{lem:nonobtuse_properties}\ref{lem:nonobtuse_properties:i}  and \eqref{eq:Z-Zcirc-identity+hom}
that
$y(s) + \tau_k z(s) \in Z$
for all $s \in [t-\varepsilon, t]$.
Thus, $z$ 
is an admissible test function in \eqref{eq:EVI_aux_left},
and we may deduce that
\begin{equation*}
\begin{aligned} 
0
&\leq  
\int_{t-\varepsilon}^{t}
 \left \langle \pi_{\ZZ(y(t))^\circ}(\xi_k) - \xi_k, \dd  \xi_k \right \rangle 
 - \frac{1}{\tau_k}
 \int_{t-\varepsilon}^{t}
 \left \langle \pi_{\ZZ(y(t))^\circ}(\xi_k) - \xi_k, \dd w \right \rangle 
 \\
 &=
 -
 \int_{t-\varepsilon}^{t}
 \left \langle \pi_{\ZZ(y(t))}(\xi_k), \dd  \xi_k \right \rangle 
 + \frac{1}{\tau_k}
 \int_{t-\varepsilon}^{t}
 \left \langle \pi_{\ZZ(y(t))}(\xi_k), \dd w \right \rangle 
 \\
 &\leq
  -
 \int_{t-\varepsilon}^{t}
 \left \langle \pi_{\ZZ(y(t))}(\xi_k), \dd  \xi_k \right \rangle 
 \\
 &=
  -
  \frac12
\dist_{\ZZ(y(t))^\circ}( \xi_k(t))^2
 +
 \frac12
\dist_{\ZZ(y(t))^\circ}(\xi_k(t-\varepsilon))^2
  \\
 &=
  -
  \frac12
\dist_{\ZZ(y(t))^\circ}(\xi_k(t))^2.
\end{aligned}
\end{equation*}
Here, we have used \cref{lem:basic_proj_properties},
$w = u - y$, 
the inclusion
$\ZZ(y(t)) \subset \ZZ(y(s))$ for all $s \in [t-\varepsilon, t]$
obtained from 
\eqref{eq:mon_play},
\cref{lem:nonneg_tang_play},
\cref{lem:CBV-Stampacchia}, 
and
$\xi_k(t-\varepsilon) = \pi_{\ZZ(y(t))^\circ}(\delta(t-))$.
The above shows that $\xi_k(t) \in \ZZ(y(t))^\circ$ 
which, in combination with \eqref{eq:EVI_aux_left},
yields $ \xi_k(t) \in \ZZ(y(t)) \cap \ZZ(y(t))^\circ$
and, thus, $\xi_k(t)  = 0$.

 Let us now consider the functions 
 $\eta_k(s) := \xi_k(s) + \pi_{\ZZ(y(t))}(\delta(t-))$,
 $s \in [t-\varepsilon, t]$. For these functions, it clearly 
 holds 
 $\eta_k \in CBV([t-\varepsilon, t]; \R^d)$
 and
 $\eta_k(t-\varepsilon) = \pi_{\ZZ(y(t))^\circ}(\delta(t-)) +
 \pi_{\ZZ(y(t))}(\delta(t-)) = \delta(t-)$ by \eqref{eq:Z-Zcirc-identity+hom}
 and the initial value condition for $\xi_k$.
 For all $i \in \AA(y(t))$, we further have 
 \begin{equation*}
 \begin{aligned}
 \langle \nu_i, y(s) + \tau_k  \eta_k(s) \rangle 
 &=
  \langle \nu_i, y(s) + \tau_k  \xi_k(s) + \tau_k \pi_{\ZZ(y(t))}(\delta(t-))  \rangle 
  \leq
  \langle \nu_i, y(s) + \tau_k \xi_k(s) \rangle 
  \leq \alpha_i\qquad \forall s \in [t-\varepsilon, t],
 \end{aligned}
 \end{equation*}
 and,
 from our assumption 
 $\II(y(t)) \subset \II(y(s))$ for all $s \in [t-\varepsilon, t]$,
  the fact that $\{\eta_k\}$ and $\{\xi_k\}$ are bounded in $CBV([t-\varepsilon, t]; \R^d)$
 by \eqref{eq:randomeq36737dhu} and \cref{lem:EVI_diff_quot}, and $\tau_k \to 0$,
 we obtain
 that $\langle \nu_i, y(s) + \tau_k  \eta_k(s) \rangle \leq \alpha_i$
 holds 
for all $i \in \II(y(t))$,
all $s \in [t-\varepsilon, t]$,
and all sufficiently large $k$. 
 In total, this shows that $y(s) + \tau_k  \eta_k(s) \in Z$
 for all $s \in [t-\varepsilon, t]$ and all large enough $k$. 
 Suppose now that a function $z \in G([t-\varepsilon, t]; \R^d)$
 satisfying $z(s) \in (1/ \tau_k)\left (Z - y(s) \right ) $ for all $s \in [t-\varepsilon, t]$
 is given. Then we may compute that 
 \begin{equation*}
 \begin{aligned}
 &\int_{t-\varepsilon}^{t}
 \left \langle z - \eta_k, \dd \eta_k \right \rangle 
 - \frac{1}{\tau_k}
 \int_{t-\varepsilon}^{t}
 \left \langle z - \eta_k, \dd w \right \rangle 
 \\
 &=
 \int_{t-\varepsilon}^{t}
 \left \langle z -  \xi_k  - \pi_{\ZZ(y(t))}(\delta(t-)), 
 \dd ( \xi_k  + \pi_{\ZZ(y(t))}(\delta(t-)) ) \right \rangle 
 - \frac{1}{\tau_k}
 \int_{t-\varepsilon}^{t}
 \left \langle z -  \xi_k  - \pi_{\ZZ(y(t))}(\delta(t-)), \dd w \right \rangle
  \\
 &=
 \int_{t-\varepsilon}^{t}
 \left \langle z -  \xi_k, 
 \dd \xi_k \right \rangle 
 - \frac{1}{\tau_k}
 \int_{t-\varepsilon}^{t}
 \left \langle z -  \xi_k, \dd w \right \rangle  
 -
  \int_{t-\varepsilon}^{t}
 \left \langle  \pi_{\ZZ(y(t))}(\delta(t-)), 
 \dd \xi_k \right \rangle 
  + \frac{1}{\tau_k}
 \int_{t-\varepsilon}^{t}
 \left \langle  \pi_{\ZZ(y(t))}(\delta(t-)), \dd w \right \rangle  
 \\
 &\geq
 - \left \langle  \pi_{\ZZ(y(t))}(\delta(t-)), \xi_k(t) \right \rangle 
 +
 \left \langle  \pi_{\ZZ(y(t))}(\delta(t-)), \xi_k(t-\varepsilon) \right \rangle 
 + \frac{1}{\tau_k}
 \int_{t-\varepsilon}^{t}
 \left \langle  \pi_{\ZZ(y(t))}(\delta(t-)), \dd \left ( \sum_{i \in \AA(y(t))} \lambda_i \nu_i \right ) \right \rangle  
  \\
 &=
 \left \langle  \pi_{\ZZ(y(t))}(\delta(t-)), \pi_{\ZZ(y(t))^\circ}(\delta(t-)) \right \rangle 
 + \frac{1}{\tau_k}\sum_{i \in \AA(y(t))}
 \left ( \lambda_i(t) - \lambda_i(t-\varepsilon)  \right )
 \left \langle  \pi_{\ZZ(y(t))}(\delta(t-)),   \nu_i  \right \rangle
   \\
 &=
 \frac{1}{\tau_k}\sum_{i \in \AA_\strict^-(t)}
 \left ( \lambda_i(t) - \lambda_i(t-\varepsilon)  \right )
 \left \langle  \pi_{\ZZ(y(t))}(\delta(t-)),   \nu_i  \right \rangle
 \\
 &= 0.
 \end{aligned}
 \end{equation*}
 Here, we have used \cref{lem:KS-constant_functions},
 the EVI satisfied by $\xi_k$, 
 $\xi_k(t) = 0$,
 $\xi_k(t-\varepsilon) = \pi_{\ZZ(y(t))^\circ}(\delta(t-))$, \cref{ass:subsec:3.3},
 the fact that $\left \langle  \pi_{\ZZ(y(t))}(x), \pi_{\ZZ(y(t))^\circ}(x) \right \rangle = 0$ holds 
 for all $x\in \R^d$ by 
 \eqref{eq:Z-Zcirc-identity+hom} and
 \cref{lem:basic_proj_properties}, 
 and the fact that $\langle \delta(t-), \nu_i \rangle = 0$ holds 
 for all $i \in \AA_\strict^-(t)$ by \cref{lem:delta_critical,th_pointwise_criticality},
which implies 
 $\langle  \pi_{\ZZ(y(t))}(\delta(t-)), \nu_i \rangle = 0$
 for all $i \in \AA_\strict^-(t)$ by \cref{lem:nonobtuse_properties}\ref{lem:nonobtuse_properties:ii}.
 If we combine all of the above, then it follows that, for all sufficiently 
 large $k$, 
 the solution $\zeta_k$ of 
 \eqref{eq:EVI_aux_left}
 with initial value $\zeta_0 = \delta(t-)$
  is given by $\zeta_k(s) = \eta_k(s) = \xi_k(s) + \pi_{\ZZ(y(t))}(\delta(t-))$.
 In particular, we have
 $\zeta_k(t) = \xi_k(t) + \pi_{\ZZ(y(t))}(\delta(t-)) = \pi_{\ZZ(y(t))}(\delta(t-))$
and the assertion of the lemma follows.
 \end{proof}
 
 Note that the proof of \cref{lem_left_jump_proto} makes extensive use of the 
 non-obtuseness of $Z$.
 With \cref{lem_left_jump_proto} and the stability estimate \eqref{eq:randomeq36737dhu}
 at hand, we can now determine the jump $\delta(t)-\delta(t-)$.
 For convenience, we formulate the resulting theorem 
 such that it is independent of  \cref{ass:subsec:3.3}. 
 
 \begin{theorem}[jumps from the left]
 \label{th:left_jumps}
 For all $t \in [0,T]$ and all $\delta \in \DD$,
 it holds $\delta(t) = \pi_{\ZZ(y(t))}(\delta(t-))$.
 \end{theorem}
 
 \begin{proof}
 In the situation of \cref{ass:subsec:3.3}, 
 we obtain from \cref{lem:aux1_solv,lem_left_jump_proto}
 that, for all sufficiently large $k$, 
  we have 
  \[
 	|\delta_k(t) - \pi_{\ZZ(y(t))}(\delta(t-))|
 	\leq
 	L
\left (
\|  h  -  h(t) \|_{CBV([t-\varepsilon, t];\R^d)}
  +
|\delta_k(t-\varepsilon) - \delta(t-)|
\right ).
 \]
 After passing to the limit $k \to \infty$, this yields 
  \[
 	|\delta(t) - \pi_{\ZZ(y(t))}(\delta(t-))|
 	\leq
 	L
	\left (
	\|  h  -  h(t) \|_{CBV([t-\varepsilon, t];\R^d)}
  	+
	|\delta(t-\varepsilon) - \delta(t-)|
	\right )
	\leq
	\gamma,
 \]
 where $\gamma > 0$ is the arbitrarily small number from \cref{ass:subsec:3.3}.
 Letting $\gamma$ go to zero proves the claim for all $t \in (0,T]$. 
 For $t = 0$, the assertion follows 
 from the convention $\delta(0) = \delta(0-)$
 and \cref{lem:delta_critical}.
 \end{proof}

\subsection{Identification of  the jumps from the right}
\label{subsec:3.4}

Next, we study the jumps $\delta(t+) - \delta(t)$ of the elements $\delta$ of $\DD$.
To this end, we consider the following situation
(similarly to the last subsection).

\begin{assumption}[standing assumptions for \cref{subsec:3.4}]\label{ass:subsec:3.4}~
\begin{enumerate}[label=\roman*)]
\item $\delta \in \DD$ is arbitrary but fixed;
\item $t \in [0,T)$ is arbitrary but fixed;
\item $\gamma > 0$ is an arbitrary but fixed given small number;
\item $\varepsilon > 0$ is chosen such that all of the following conditions hold:
\begin{enumerate}
\item $\varepsilon < T-t$ (trivially achievable);
\item $\II(y(t)) \subset \II(y(s))$ for all $s \in [t, t+\varepsilon]$ (achievable by continuity of $y$);
\item the assertions in \cref{lemma:monotone_play} hold 
with functions $\lambda_i \in CBV([t - \varepsilon, t + \varepsilon] \cap [0,T])$,
$i \in \AA(y(t))$;
\item $\lambda_i = \mathrm{const}$ holds on $[t, t +\varepsilon]$ for all $i \in \BB^+(t)$
(achievable by \cref{def:temporal_indices});
\item 
\label{ass:subsec:3.4:iv_e}
$\|  h  -  h(t) \|_{CBV([t, t + \varepsilon];\R^d)} \leq  \gamma/L $ holds,
where $ L >0$ is as in \eqref{eq:BV_Lipschitz} (trivially achievable);
\end{enumerate}
\item $\{\tau_k\}$ is a sequence as 
in the definition of $\DD$ for $\delta$
with associated difference quotients 
$\delta_k := \delta_{\tau_k}$.
\end{enumerate}
\end{assumption}

Note that 
\cref{ass:subsec:3.4}\ref{ass:subsec:3.4:iv_e}
again relies crucially on the 
$CBV$-regularity of $h$.
To identify the jump
$\delta(t+) - \delta(t)$
in the situation of \cref{ass:subsec:3.4},
we proceed analogously to \cref{subsec:3.3}
and consider suitably defined 
auxiliary EVIs. 
The first one reads as follows:
\begin{equation*}
\label{eq:EVI_aux_right_1}
\tag{AUX$_2$}
\begin{aligned}
&
\zeta_{k} \in CBV([t, t + \varepsilon]; \R^d),
\qquad
\zeta_k(s) \in \frac{1}{\tau_k}\left (Z_\AA - y(s) \right ) \forall s \in [t, t + \varepsilon],
\qquad
\zeta_k(t) = \delta_k(t),
\\
&\int_{t}^{t + \varepsilon}
 \left \langle z - \zeta_k, \dd  \zeta_k \right \rangle 
 - \frac{1}{\tau_k}
 \int_{t}^{t + \varepsilon}
 \left \langle z - \zeta_k, \dd w \right \rangle 
 \geq 0
\qquad \forall z \in G([t, t + \varepsilon]; \R^d)\colon 
 z(s) \in \frac{1}{\tau_k}\left (Z_\AA - y(s) \right )~\forall s \in [t, t + \varepsilon].
\end{aligned}
\end{equation*}
Here and in what follows, $Z_\AA$ is defined by 
\[
Z_\AA := 
\{x \in \R^d \mid \langle \nu_i, x\rangle \leq \alpha_i~\forall i \in \AA(y(t)) \}.
\]

\begin{lemma}[unique solvability of \eqref{eq:EVI_aux_right_1}]
\label{lem_close_right_jump}
For all large enough $k$, 
\eqref{eq:EVI_aux_right_1}
is uniquely solvable and 
its solution $\zeta_k$ satisfies
\begin{equation}
\label{eq:random_Lipschitz_6376}
\| \delta_k  -  \zeta_k \|_{CBV([t, t + \varepsilon];\R^d)} \leq \gamma.
\end{equation}
\end{lemma}

\begin{proof}
From our assumption 
$\II(y(t)) \subset \II(y(s))$ for all $s \in [t, t+\varepsilon]$, the convergence $\tau_k \to 0$,
the Lipschitz estimate \eqref{eq:BV_Lipschitz},
the continuity of $y$,
and \cref{lem:nonneg_tang_play}, 
we obtain 
that, for all large enough $k$, the set $Z$ 
in the EVIs satisfied by 
$y = \SS(u, y_0)$
and $\SS(u + \tau_k h, y_0 + \tau_k h_0)$
on the interval $[t, t + \varepsilon]$
can be replaced by $Z_\AA$ without changing 
the solutions of these problems. 
This implies that the 
difference quotients $\delta_k$
are, for all large enough $k$, 
on $[t, t + \varepsilon]$
also the (necessarily unique) solutions 
of the following EVIs
(cf.\ the derivation of \eqref{eq:DQ-EVI}):
\begin{equation*}
\begin{aligned}
&
\delta_{k} \in CBV([t, t + \varepsilon]; \R^d),
\qquad
\delta_k(s) \in \frac{1}{\tau_k}\left (Z_\AA - y(s) \right ) \forall s \in [t, t + \varepsilon],
\qquad
\delta_k(t) = \delta_k(t),
\\
&\int_{t}^{t + \varepsilon}
 \left \langle z - \delta_k, \dd  (\delta_k - h) \right \rangle 
 - \frac{1}{\tau_k}
 \int_{t}^{t + \varepsilon}
 \left \langle z - \delta_k, \dd w \right \rangle 
 \geq 0
\\
&\hspace{2cm}\forall z \in G([t, t + \varepsilon]; \R^d)\colon 
 z(s) \in \frac{1}{\tau_k}\left (Z_\AA - y(s) \right )~\forall s \in [t, t + \varepsilon].
\end{aligned}
\end{equation*}
Note that, in the case 
$h|_{[t, t + \varepsilon]} = \mathrm{const}$, the above problem  
is identical to \eqref{eq:EVI_aux_right_1}
by \cref{lem:KS-constant_functions}.
To establish the unique solvability of \eqref{eq:EVI_aux_right_1},
we can thus argue along the exact same lines 
as in the proof of \cref{lem:aux1_solv}. 
The same is true for the proof of  \eqref{eq:random_Lipschitz_6376};
cf.\ the arguments in \eqref{eq:randomeq3636478398-2}.
\end{proof}

\begin{lemma}[reformulation of \eqref{eq:EVI_aux_right_1}]
\label{lem:3.18-2}
For all large enough $k$, 
the solution $\zeta_k$ of \eqref{eq:EVI_aux_right_1}
is also the (necessarily unique)
solution of the following EVI:
\begin{equation*}
\label{eq:EVI_aux_right_2}
\tag{AUX$_3$}
\begin{aligned}
&
\zeta_{k} \in CBV([t, t + \varepsilon]; \R^d),
\qquad
\zeta_k(s) \in \ZZ(y(t))~\forall s \in [t, t + \varepsilon],
\qquad
\zeta_k(t) = \delta_k(t),
\\
&\int_{t}^{t + \varepsilon}
 \left \langle z - \zeta_k, \dd  \zeta_k \right \rangle 
 - \frac{1}{\tau_k}
 \int_{t}^{t + \varepsilon}
 \left \langle z - \zeta_k, \dd w \right \rangle 
 \geq 0
\qquad \forall z \in G([t, t + \varepsilon]; \R^d)\colon 
 z(s) \in \ZZ(y(t))~\forall s \in [t, t + \varepsilon].
\end{aligned}
\end{equation*}
\end{lemma}
\begin{proof}
For every $z \in G([t, t + \varepsilon]; \R^d)$
satisfying $z(s) \in \ZZ(y(t))$
for all $s \in [t, t + \varepsilon]$, 
it holds 
\begin{equation}
\label{eq:randomeq2673738}
\langle \nu_i, y(s) + \tau_k z(s) \rangle
\leq \alpha_i + 0 \qquad \forall s \in [t, t + \varepsilon]
\qquad \forall i \in \AA(y(t)).
\end{equation}
That $\zeta_k$ satisfies the 
variational inequality
in \eqref{eq:EVI_aux_right_2} thus follows 
immediately from \eqref{eq:EVI_aux_right_1}. 
The same is true for the regularity statement 
and the initial condition in \eqref{eq:EVI_aux_right_2}.
It remains to prove that 
$\zeta_k(s) \in \ZZ(y(t))$ for all $s \in [t, t+\varepsilon]$. 
To this end, 
we note that, 
for all $s \in [t, t+\varepsilon]$,
we have 
$y(s) + \tau_k \pi_{\ZZ(y(t))}( \zeta_k(s)) \in Z_\AA$
(by the same estimate as in \eqref{eq:randomeq2673738})
and 
$y(s) + \tau_k \pi_{\ZZ(y(t))^\circ}( \zeta_k(s)) \in Z_\AA$
by \cref{lem:nonobtuse_properties}\ref{lem:nonobtuse_properties:i}  and \eqref{eq:Z-Zcirc-identity+hom}. 
This allows us to obtain the following from 
\eqref{eq:EVI_aux_right_1}
for all large enough $k$:
\begin{equation*}
\begin{aligned}
0
&\leq
\int_{t}^{s}
 \left \langle \pi_{\ZZ(y(t))}(\zeta_k) - \zeta_k, \dd  \zeta_k  \right \rangle 
 - \frac{1}{\tau_k}
 \int_{t}^{s}
 \left \langle \pi_{\ZZ(y(t))}(\zeta_k) - \zeta_k, \dd w \right \rangle 
 \\
 &=
\int_{t}^{s}
 \left \langle - \pi_{\ZZ(y(t))^\circ}(\zeta_k), \dd   \zeta_k \right \rangle 
 + \frac{1}{\tau_k^2}
 \int_{t}^{s}
 \left \langle \tau_k\pi_{\ZZ(y(t))^\circ}(\zeta_k), \dd w \right \rangle 
  \\
 &=
- \frac12 \dist_{\ZZ(y(t))}(\zeta_k(s))^2
 + \frac12 \dist_{\ZZ(y(t))}(\zeta_k(t))^2
 - \frac{1}{\tau_k^2}
 \int_{t}^{s}
 \left \langle (y + \tau_k\pi_{\ZZ(y(t))^\circ}(\zeta_k)) - y, \dd (y - u) \right \rangle 
 \\
 &\leq 
 - \frac12 \dist_{\ZZ(y(t))}(\zeta_k(s))^2
 \qquad \forall s \in (t, t+\varepsilon].
\end{aligned}
\end{equation*}
Here, we have used \eqref{eq:Z-Zcirc-identity+hom},
\eqref{eq:F_prime_cone}, \cref{lem:CBV-Stampacchia},
$\zeta_k(t) = \delta_k(t) \in \ZZ(y(t))$,
and the fact that 
$y$ satisfies 
\eqref{eq:EVI_G}
on $[t,t+\varepsilon]$
with $Z$ replaced by $Z_\AA$;
cf.\ the proof
of \cref{lem_close_right_jump}. The above shows 
that $\zeta_k(s) \in \ZZ(y(t))$ holds
for all $s \in (t, t + \varepsilon]$.
Since this inclusion is trivially true 
for $s=t$, this shows 
that $\zeta_k$ indeed solves
\eqref{eq:EVI_aux_right_2}.
That \eqref{eq:EVI_aux_right_2}
can have at most one solution 
follows from the 
standard existence and uniqueness result
for the stop operator;
cf.\
\cref{prop:unique_solvability_stop_play}.
(Note that the admissible set of \eqref{eq:EVI_aux_right_2} does not vary with time.)
\end{proof}

Next, we consider the auxiliary problem
\begin{equation*}
\label{eq:EVI_aux_right_3}
\tag{AUX$_4$}
\begin{aligned}
&
\xi_{k} \in CBV([t, t + \varepsilon]; \R^d),
\qquad
\xi_k(s) \in \tilde \ZZ ~\forall s \in [t, t + \varepsilon],
\qquad
\xi_k(t) = \xi_0,
\\
&\int_{t}^{t + \varepsilon}
 \left \langle z - \xi_k, \dd  \xi_k \right \rangle 
 - \frac{1}{\tau_k}
 \int_{t}^{t + \varepsilon}
 \left \langle z - \xi_k, \dd w \right \rangle 
 \geq 0
\qquad \forall z \in G([t, t + \varepsilon]; \R^d)\colon 
 z(s) \in \tilde \ZZ~\forall s \in [t, t + \varepsilon]
\end{aligned}
\end{equation*}
with 
\[
\tilde \ZZ:= 
\{z \in \R^d \mid \langle \nu_i, z\rangle \leq 0~\forall i \in \AA_\strict^+(t) \}.
\]

\begin{lemma}[unique solvability of \eqref{eq:EVI_aux_right_3}]
\label{lem:aux-sol_again_42}
Problem \eqref{eq:EVI_aux_right_3} 
has a unique solution 
$\xi_k$ for all $\xi_0 \in \tilde \ZZ$
and $k$.
\end{lemma}
\begin{proof}
The assertion follows from the standard 
existence result for the stop operator;
cf.\
\cref{prop:unique_solvability_stop_play}.
(Note that the admissible set is time-independent 
in \eqref{eq:EVI_aux_right_3}.)
\end{proof}

\begin{lemma}[order property of \eqref{eq:EVI_aux_right_3}]
\label{lemma:the_little_box}
Let 
$\xi_0 \in \tilde \ZZ$ be given
and let $\xi_k$ be the corresponding solution
of \eqref{eq:EVI_aux_right_3}.
Then it holds
\[
\xi_k(s) \in 
\tilde \ZZ \cap (\xi_0 + \tilde \ZZ^\circ) 
\quad\text{and}\quad
\langle \nu_i, \xi_k(s) \rangle \leq 
\langle \nu_i, \xi_0 \rangle
\quad \forall s \in [t, t+\varepsilon]
\quad \forall i \in I\setminus \AA_\strict^+(t).
\]
\end{lemma}

\begin{proof}
Consider the function 
$z := \pi_{\xi_0 + \tilde \ZZ^\circ}(\xi_k) \in G([t,t+\varepsilon];\R^d)$.
Then it follows 
from 
\cref{lem:nonobtuse_properties}\ref{lem:nonobtuse_properties:iii}
that we have 
$z(s) \in \tilde \ZZ$
and
$\xi_k(s) - z(s) \in \tilde \ZZ$
for all $s \in [t, t + \varepsilon]$.
In particular, $z$ is an admissible test 
function in \eqref{eq:EVI_aux_right_3}, 
and we may deduce that 
\begin{equation*}
\begin{aligned}
0
&\leq
\int_{t}^{s}
 \left \langle \pi_{\xi_0 + \tilde \ZZ^\circ}(\xi_k) - \xi_k, \dd  \xi_k \right \rangle 
 - \frac{1}{\tau_k}
 \int_{t}^{s}
 \left \langle z - \xi_k, \dd w \right \rangle 
 \\
 &=
 - \frac12 \dist_{\xi_0 + \tilde \ZZ^\circ}(\xi_k(s))^2
 + \frac12 \dist_{\xi_0 + \tilde \ZZ^\circ}(\xi_0)^2
 - \frac{1}{\tau_k}
 \int_{t}^{s}
 \left \langle z- \xi_k, \dd \left (
 \sum_{i \in \AA(y(t))} \lambda_i\nu_i
 \right ) \right \rangle 
  \\
 &=
 - \frac12 \dist_{\xi_0 + \tilde \ZZ^\circ}(\xi_k(s))^2
 + \frac{1}{\tau_k}
 \sum_{i \in \AA_\strict^+(t)}
 \int_{t}^{s}
  \langle \xi_k - z, \nu_i \rangle \dd \lambda_i 
  \\
  &\leq
  - \frac12 \dist_{\xi_0 + \tilde \ZZ^\circ}(\xi_k(s))^2
  \qquad \forall s \in (t, t + \varepsilon].
\end{aligned}
\end{equation*}
Here, we have used 
\eqref{eq:F_prime_formula}, 
\cref{lem:CBV-Stampacchia},
\eqref{eq:mon_play},
$\xi_0 \in \xi_0 + \tilde \ZZ^\circ$,
our assumption that 
$\lambda_i = \mathrm{const}$ holds on 
$[t, t +\varepsilon]$ for all $i \in \BB^+(t)$,
\cref{lem:KS-constant_functions},
the fact that
$\lambda_i$ is nondecreasing
for all $i \in \AA_\strict^+(t)$, 
$\xi_k(s) - z(s) \in \tilde \ZZ$ for all 
$s \in [t, t+\varepsilon]$, 
the definition of $\tilde \ZZ$,
and
\cref{lemma:crit_noneg}. 
In combination with the admissibility 
of $\xi_k$ in \eqref{eq:EVI_aux_right_3}
and the trivial inclusion 
$\xi_0 \in \tilde \ZZ \cap (\xi_0 + \tilde \ZZ^\circ)$,
the above yields 
$\xi_k(s) \in \tilde \ZZ \cap (\xi_0 + \tilde \ZZ^\circ) $ for all $s \in [t, t+\varepsilon]$
as desired. Note that, 
due to the definition of $\tilde \ZZ$
and \cref{lem:normal_cone_polyhedral},
the latter inclusion implies that, 
for every $s \in [t, t + \varepsilon]$,
there exist numbers $\beta_j \geq 0$,
$j \in \AA_\strict^+(t)$,
satisfying 
\[
\xi_k(s) = \xi_0 + \sum_{j \in \AA_\strict^+(t)} \beta_j \nu_j.
\]
As $Z$ is non-obtuse, the above yields
\[
\langle \nu_i, \xi_k(s) \rangle
=
\left \langle \nu_i, \xi_0 + \sum_{j \in \AA_\strict^+(t)} \beta_j \nu_j\right \rangle
=
\langle \nu_i, \xi_0\rangle
+ \sum_{j \in \AA_\strict^+(t)}\beta_j
\langle \nu_i, \nu_j \rangle
\leq
\langle \nu_i, \xi_0\rangle
\qquad \forall i \in I\setminus \AA_\strict^+(t).
\]
This completes the proof of the lemma. 
\end{proof}

Note that the proofs of 
\cref{lem:3.18-2,lemma:the_little_box}
again rely crucially on the non-obtuseness of 
$Z$.
Next, we 
establish a connection
between 
\eqref{eq:EVI_aux_right_1}
and
\eqref{eq:EVI_aux_right_3}.
Recall that 
$V^+(t) := \spann(\{ \nu_i \mid i \in \AA_\strict^+(t)\})^\perp$
and, consequently, 
$V^+(t)^\perp = \spann(\{ \nu_i \mid i \in \AA_\strict^+(t)\})$;
see \cref{def:temporal_indices}.

\begin{lemma}[relationship between 
the auxiliary problems
\eqref{eq:EVI_aux_right_1}
and
\eqref{eq:EVI_aux_right_3}]
\label{lem:superposition_shift}
Let
$\xi_k$
be the solution of \eqref{eq:EVI_aux_right_3}
with initial value 
$\xi_0 := \pi_{V^+(t)^\perp}(\delta_k(t))$.
Then, for all sufficiently large $k$,
it holds 
$\zeta_k  = \xi_k + \pi_{V^+(t)}(\delta_k(t))$,
where $\zeta_k$ is the solution of 
the EVI \eqref{eq:EVI_aux_right_1}.
\end{lemma}

\begin{proof}
Note that $\delta_k(t) \in \ZZ(y(t))$
and the definition of $V^+(t)$ yield
$\langle \nu_i, \xi_0 \rangle = \langle \nu_i, 
\pi_{V^+(t)^\perp}(\delta_k(t)) \rangle 
=
\langle \nu_i, 
\delta_k(t)  \rangle \leq 0$ for all $i \in \AA_\strict^+(t)$.
Thus, $\xi_0 \in \tilde \ZZ$  and $\xi_0$ 
is an admissible initial value in \eqref{eq:EVI_aux_right_3}.
Define $\hat \xi_k := \xi_k + \pi_{V^+(t)}(\delta_k(t))$.
We claim that this function solves \eqref{eq:EVI_aux_right_3}
with initial value $\hat \xi_0 := \delta_k(t)\in \ZZ(y(t)) \subset \tilde \ZZ$. 
To see that this is true, we note 
that $\hat \xi_k$ trivially satisfies 
$\hat \xi_k \in CBV([t, t + \varepsilon]; \R^d)$
and 
\[
\hat \xi_k(t) = 
\xi_k(t) + \pi_{V^+(t)}(\delta_k(t))
=
\pi_{V^+(t)^\perp}(\delta_k(t)) + \pi_{V^+(t)}(\delta_k(t))
=
\delta_k(t)
=
\hat \xi_0.
\]
As $\langle \nu_i, \pi_{V^+(t)}(\delta_k(t))\rangle = 0$
holds for all $i \in \AA_\strict^+(t)$
and due to $\xi_k \in G([t, t+\varepsilon]; \tilde \ZZ)$,
we further have 
\[
\langle \nu_i, \hat \xi_k(s) \rangle 
=
\langle \nu_i, \xi_k(s) + \pi_{V^+(t)}(\delta_k(t)) \rangle 
=
\langle \nu_i, \xi_k(s) \rangle 
\leq
0\qquad \forall i \in \AA_\strict^+(t)\qquad \forall s \in [t, t+\varepsilon]
\]
and, thus, $\hat \xi_k(s) \in \tilde \ZZ$ for 
all $s \in [t, t+\varepsilon]$.
To prove that $\hat \xi_k$ indeed solves \eqref{eq:EVI_aux_right_3}
with  $\hat \xi_0 := \delta_k(t)$, it remains 
to check the variational inequality in \eqref{eq:EVI_aux_right_3}. 
So let $z \in G([t, t + \varepsilon]; \R^d)$
satisfying $z(s) \in \tilde \ZZ$ for all $s \in [t, t + \varepsilon]$
be given. From \cref{lemma:the_little_box} and \cref{lem:normal_cone_polyhedral}, 
we obtain that 
$\xi_k(s) \in (\xi_0 + \tilde \ZZ^\circ) \subset 
\spann(\{ \nu_i \mid i \in \AA_\strict^+(t)\}) = V^+(t)^\perp$
holds for all $s \in [t, t+\varepsilon]$,
and from \eqref{eq:mon_play},
\cref{lem:KS-constant_functions},
and our assumption that $\lambda_i = \mathrm{const}$ holds on 
$[t, t +\varepsilon]$ for all $i \in \BB^+(t)$,
it follows that 
\begin{equation*}
\begin{aligned}
\int_{t}^{t + \varepsilon}
 \left \langle \pi_{V^+(t)}(\delta_k(t)) , \dd w \right \rangle 
 &=
 \sum_{i \in \AA(y(t))}
 \int_{t}^{t + \varepsilon}
 \left \langle \pi_{V^+(t)}(\delta_k(t)) , \nu_i\right \rangle \dd \lambda_i
 \\
 &=
 \sum_{i \in \AA(y(t))}
 \langle \nu_i, \pi_{V^+(t)}(\delta_k(t)) \rangle(\lambda_i(t+\varepsilon) - \lambda_i(t))
 \\
 &=
  \sum_{i \in \AA_\strict^+(t)}
 \langle \nu_i, \pi_{V^+(t)}(\delta_k(t)) \rangle(\lambda_i(t+\varepsilon) - \lambda_i(t))
 =
 0.
\end{aligned}
\end{equation*}
In combination with the EVI for $\xi_k$ and again 
\cref{lem:KS-constant_functions}, this yields
\begin{equation*}
\begin{aligned}
&\int_{t}^{t + \varepsilon}
 \left \langle z - \hat \xi_k, \dd  \hat \xi_k \right \rangle 
 - \frac{1}{\tau_k}
 \int_{t}^{t + \varepsilon}
 \left \langle z - \hat \xi_k, \dd w \right \rangle 
 \\
 &=
 \int_{t}^{t + \varepsilon}
 \left \langle z -  \xi_k - \pi_{V^+(t)}(\delta_k(t)), \dd  
 ( \xi_k + \pi_{V^+(t)}(\delta_k(t))) \right \rangle 
 - \frac{1}{\tau_k}
 \int_{t}^{t + \varepsilon}
 \left \langle z -  \xi_k - \pi_{V^+(t)}(\delta_k(t)), \dd w \right \rangle 
 \\
  &=- \int_{t}^{t + \varepsilon}
 \left \langle \pi_{V^+(t)}(\delta_k(t)), \dd  
  \xi_k  \right \rangle 
  +
  \int_{t}^{t + \varepsilon}
 \left \langle z -  \xi_k , \dd  
  \xi_k  \right \rangle 
 - \frac{1}{\tau_k}
 \int_{t}^{t + \varepsilon}
 \left \langle z -  \xi_k, \dd w \right \rangle 
 \\
 &\geq
 - \int_{t}^{t + \varepsilon}
 \left \langle \pi_{V^+(t)}(\delta_k(t)), \dd  
  \xi_k  \right \rangle 
  \\
  &=  
 \left \langle \pi_{V^+(t)}(\delta_k(t)), \xi_k(t) - \xi_k(t + \varepsilon) \right \rangle =0,
\end{aligned}
\end{equation*}
where, in the last line, we have 
again used that 
$\xi_k(s)  \in V^+(t)^\perp$
holds for all $s \in [t, t+\varepsilon]$.
Thus, $\hat \xi_k := \xi_k + \pi_{V^+(t)}(\delta_k(t))$
indeed solves \eqref{eq:EVI_aux_right_3}
with initial value $\hat \xi_0 := \delta_k(t)$.
Next, we prove that $\hat \xi_k = \zeta_k$
holds for all large enough $k$.
Note that, since $\zeta_k$ solves \eqref{eq:EVI_aux_right_2}
for all large enough $k$, 
since $\ZZ(y(t)) \subset \tilde \ZZ$ holds, 
since $\zeta_k$ and $\hat \xi_k$ have the 
same initial value, 
and since \eqref{eq:EVI_aux_right_2} is uniquely solvable,
to show that $\hat \xi_k$ is identical $\zeta_k$,
it suffices to prove that
$\hat \xi_k(s) \in \ZZ(y(t))$ for all $s \in [t, t+\varepsilon]$.
To see this, we note that $\xi_k \in G([t, t+\varepsilon]; \tilde \ZZ)$
yields
\[
\langle \nu_i, \hat \xi_k(s)\rangle
=
\langle \nu_i, \xi_k(s) + \pi_{V^+(t)}(\delta_k(t))\rangle
=
\langle \nu_i, \xi_k(s)\rangle
\leq 0
\qquad
\forall i \in \AA_\strict^+(t)
\qquad\forall s \in [t, t+\varepsilon]
\]
and that \cref{lemma:the_little_box} and 
the inclusion $\delta_k(t) \in \ZZ(y(t))$
imply
\[
\langle \nu_i, \hat \xi_k(s) \rangle \leq 
\langle \nu_i, \hat \xi_0 \rangle
=
\langle \nu_i, \delta_k(t) \rangle
\leq 
0
\quad \forall s \in [t, t+\varepsilon]
\quad \forall i \in \AA(y(t))\setminus \AA_\strict^+(t).
\]
Thus, we indeed have $\hat \xi_k(s) \in \ZZ(y(t))$ for all $s \in [t, t+\varepsilon]$ and the assertion of the lemma follows. 
\end{proof}

By combining 
\cref{lem_close_right_jump,lem:3.18-2,lem:aux-sol_again_42,lemma:the_little_box,lem:superposition_shift},
we can now prove the main result of this 
subsection. 
Analogously to \cref{th:left_jumps},
we
formulate it 
independently of \cref{ass:subsec:3.4}. 
 \begin{theorem}[jumps from the right]
 \label{th:jump_right}
 For all $t \in [0,T]$ and all $\delta \in \DD$,
 it holds $\delta(t+) = \pi_{V^+(t)}(\delta(t))$.
 \end{theorem}

 \begin{proof}
 For $t = T$, we have $V^+(t) = \R^d$ and $\delta(t+) = \delta(t)$ 
 by definition/convention
 and the 
 assertion is trivial. We may thus assume in the following 
 that $t \in [0,T)$.
 Suppose that $\gamma > 0$ is given and 
 consider the situation in \cref{ass:subsec:3.4}. 
 Let $k$ be large enough and let 
 $\xi_k$ and $\zeta_k$ be as in \cref{lem:superposition_shift}. 
 Then we know from \cref{lem_close_right_jump}
 that $\| \delta_k  -  \zeta_k \|_{CBV([t, t + \varepsilon];\R^d)} \leq \gamma$ holds,
 from \cref{lem:superposition_shift}
 that $\zeta_k  = \xi_k + \pi_{V^+(t)}(\delta_k(t))$
 holds on $[t, t+\varepsilon]$,
 and from \cref{lemma:the_little_box},
 $\xi_0 = \pi_{V^+(t)^\perp}(\delta_k(t))$,
 and \cref{lem:normal_cone_polyhedral} that 
 $\xi_k(s) \in \spann(\{ \nu_i \mid i \in \AA_\strict^+(t)\}) = V^+(t)^\perp$.
 Combining all of this with \eqref{eq:1projLip} yields
 \begin{equation*}
 \begin{aligned}
 \| \pi_{V^+(t)}(\delta_k)  - \pi_{V^+(t)}(\delta_k(t)) \|_{C([t, t + \varepsilon];\R^d)}
 &=
  \| \pi_{V^+(t)}(\delta_k)  - 
  \pi_{V^+(t)}(\delta_k(t)) - \pi_{V^+(t)}(\xi_k) \|_{C([t, t + \varepsilon];\R^d)}
  \\
   &=
  \| \pi_{V^+(t)}(\delta_k)  - 
  \pi_{V^+(t)}(\zeta_k) \|_{C([t, t + \varepsilon];\R^d)}
  \\
  &\leq
  \|  \delta_k - \zeta_k \|_{C([t, t + \varepsilon];\R^d)}
  \\
  &\leq \gamma
 \end{aligned}
 \end{equation*}
 for all large enough $k$ and, after passing to the limit $k \to \infty$,
 \[
    \left | \pi_{V^+(t)}(\delta(s))  - \pi_{V^+(t)}(\delta(t)) \right |
    \leq \gamma \qquad \forall s \in [t, t+\varepsilon].
 \]
 Due to the definition of $\delta(t+)$,
 this shows that $\left | \pi_{V^+(t)}(\delta(t+))  - \pi_{V^+(t)}(\delta(t)) \right |
    \leq \gamma$ and, 
    since $\delta(t+) \in V^+(t)$ holds by \cref{lem:delta_critical,th_pointwise_criticality},
    that $\left | \delta(t+)   - \pi_{V^+(t)}(\delta(t)) \right |
    \leq \gamma$. As $\gamma > 0$ was arbitrary, the assertion of the theorem now follows.
 \end{proof}
 
\subsection{Temporal polyhedricity}
\label{subsec:3.5}

With \cref{th:left_jumps,th:jump_right}
at hand, 
we can revisit the notion of criticality 
from \cref{subsec:3.2}
to arrive at the following 
important corollary. 

\begin{corollary}[criticality of right limits]
\label{cor:right_tangential}
If 
$\delta \in \DD$ holds, 
then $\delta_+$ is a critical direction. 
\end{corollary}
\begin{proof}
Let $\delta \in \DD$ be given. 
We first check that 
$\delta_+(t) \in \ZZ(y(t))$
holds for all $t \in [0,T]$.
For $t=T$, this inclusion is trivial 
since our conventions for the right limit and 
\cref{lem:delta_critical} imply 
$\delta_+(T) = \delta(T+) = \delta(T) \in \ZZ(y(T))$. 
For $t \in [0,T)$, 
we obtain from 
\cref{lem:delta_critical,th:jump_right}
that 
$\delta(t+) = \pi_{V^+(t)}(\delta(t))$
with $\delta(t) \in \ZZ(y(t))$
and 
$V^+(t) := \spann(\{ \nu_i \mid i \in \AA_\strict^+(t)\})^\perp$.
In view of \cref{lem:nonobtuse_properties}\ref{lem:nonobtuse_properties:iv}, 
this implies 
$\delta(t+) \in \ZZ(y(t))$
for all $t \in [0,T)$. We thus indeed have 
$\delta_+(t) \in \ZZ(y(t))$
for all $t \in [0,T]$.
From \eqref{eq:Monteiro4.1.9},
\cref{lem:delta_critical},
and
\cref{th_pointwise_criticality},
we further obtain that 
$\delta_+(t-) = \delta(t-) \in V^-(t)$
for all $t \in (0,T]$
and 
$\delta_+(t+) = \delta(t+) \in V^+(t)$
for all $t \in [0,T)$.
Since $V^-(0) = V^+(T)= \R^d$ holds,
this shows that $\delta_+$
satisfies the conditions in \eqref{eq:pointwise_criticality}.
By invoking
\cref{th_pointwise_criticality},
it now follows
that $\delta_+$ is critical and the proof is complete.
\end{proof}

\Cref{cor:right_tangential} motivates the following definition.

\begin{definition}[critical cone in {$G_r([0,T];\R^d)$}]
\label{def:critical_cone}
We define 
\[
\KK_{G_r}^{\crit}(y,u)
:=
\left \{
\left.
z \in G_r([0,T];\R^d)
~\right |~
z \text{ is a critical direction}
\right \}.
\]
\end{definition}

Note that, by 
\cref{def:critical_direction},
\cref{lem:crit_subinterval},
and
\cref{th_pointwise_criticality},
we have 
\begin{equation}
\label{eq:GR-crit_cone_alts}
\begin{aligned}
\KK_{G_r}^{\crit}(y,u)
&=
\left \{
z \in G_r([0,T]; \R^d)
~\left |~
z(t) \in \ZZ(y(t))
~\forall t \in [0,T]
\text{ and }
\int_0^T \langle z, \dd w \rangle = 0
\right.
\right \}
\\
&=
\left \{
z \in G_r([0,T]; \R^d)
~\left |~
z(t) \in \ZZ(y(t))
~\forall t \in [0,T]
\text{ and }
\int_{s_1}^{s_2} 
\langle z, \dd w \rangle = 0
~\forall\,0 \leq s_1 < s_2\leq T
\right.
\right \}
\\
&=
\left \{
\left.
z \in G_r([0,T]; \R^d)
~\,\right |~
 z(t) \in \ZZ(y(t)),~
z(t+) \in V^+(t),~
z(t-) \in V^-(t)~
\forall t \in [0,T]
\right \}.
\end{aligned}
\end{equation}
From \cref{cor:right_tangential}, we further obtain:

\begin{corollary}[critical cone and right limits]
For every $\delta \in \DD$, it holds $\delta_+ \in \KK_{G_r}^{\crit}(y,u)$.
\label{cor:DDsubsetCritCone}
\end{corollary}

The main goal of this subsection is to demonstrate that
elements of the cone $\KK_{G_r}^{\crit}(y,u)$
can be approximated (in an appropriate sense)
by directions that are (a) more regular, (b) critical, and 
(c) admissible in the EVI \eqref{eq:DQ-EVI} satisfied
by the difference quotients $\delta_\tau$. 
As we will see in \cref{subsec:3.6}, this approximation property
is crucial for the derivation of the system \eqref{eq:char_sys_1}
that implies that 
$\DD$ is a singleton and, ultimately,
characterizes the directional derivatives of $\SS$. 
We remark that, in the context of elliptic variational inequalities,
density results of this type are typically referred 
to as \emph{polyhedricity} conditions;
see \cite{ChristofPhd2018,Haraux1977,Harder2017,Wachsmuth2019}
and the references therein. 
Accordingly, we call the resulting
approximation property \emph{temporal polyhedricity}. 
The set of critical radial directions that we 
use for the approximation of the elements 
of $\KK_{G_r}^{\crit}(y,u)$ is defined as follows:

\begin{definition}[critical radial $BV_r$-directions with zero left limits at jump points]
We define 
\begin{equation*}
\begin{aligned}
\KK_{BV_r}^{\rad, \crit, 0}(y,u)
:=
\Big \{ 
z \in 
\KK_{G_r}^{\crit}(y,u)
\cap BV_r([0,T];\R^d)
\,\Big |&
\text{ there exists a number } \tau > 0 \text{ independent of $s$ such that} 
\\ &~y(s) + \tau z(s) \in Z \text{ is true for all } s \in [0,T] \text{ and it holds}
\\
&~z(s-) = 0 \text{ for all $s \in [0,T]$ with } z(s-) \neq z(s)
\Big \}.
\end{aligned}
\end{equation*}
\end{definition}

We can now 
prove the main result of this subsection.

\begin{theorem}[temporal polyhedricity]
\label{theorem:tempoly}%
Let $z \in \KK_{G_r}^{\crit}(y,u)$ be given.
Then there exist 
$z_{k,l}, z_l \in G_r([0,T]; \R^d)$, $k,l \in \mathbb{N}$, such that the following is true:
\begin{equation}
\label{eq:temp_poly}
\begin{gathered}
z_{k,l} \in \KK_{BV_r}^{\rad, \crit, 0}(y,u), 
\qquad \|z_{k,l}\|_\infty \leq \|z\|_\infty~\forall k,l,
\\
z_{l} \in \KK_{G_r}^{\crit}(y,u) \cap BV_r([0,T]; \R^d),
\qquad \|z_{l}\|_\infty \leq \|z\|_\infty~\forall l,
\\
z_{k,l} \to z_l \text{ pointwise in } [0,T] \text{ for } k \to \infty \text{ for all }l,
\\
z_l \to z \text{ uniformly on $[0,T]$ for } l \to \infty.
\end{gathered}
\end{equation}
\end{theorem}

\begin{proof}
Let $z \in \KK_{G_r}^{\crit}(y,u)$ be given.
To construct the sequence 
$\{z_l\}$, we proceed along the lines 
of the proof of \cref{th_pointwise_criticality}.
Let $l \in \N$ be arbitrary but fixed.
As $z$ is right-continuous and regulated, 
for every $t \in [0,T]$, there 
exists a number $\varepsilon_t > 0$ such that 
the step function
\[
v_t\colon [0,T] \to \R^d,
\qquad
v_t(s)
:=
\begin{cases}
z(t-) & \text{ if } s \in [0,t),
\\
z(t+) & \text{ if } s \in [t,T],
\end{cases}
\]
satisfies $\| z - v_t\|_{\infty, [t-\varepsilon_t, t + \varepsilon_t] \cap [0,T]} \leq 1/l$, such that the 
assertions of \cref{lemma:monotone_play}
hold for $t$ on 
$[t-\varepsilon_t, t + \varepsilon_t] \cap [0,T]$ with functions 
$\lambda_i^t$, $i \in \AA(y(t))$,
such that 
$\lambda_i^t = \mathrm{const}$ holds on $[t-\varepsilon_t, t]\cap [0,T]$ for all $i \in \BB^-(t)$,
such that 
$\lambda_i^t = \mathrm{const}$ holds on $[t, t+ \varepsilon_t]\cap [0,T]$ for all $i \in \BB^+(t)$,
such that 
$ I \setminus \AA^+(t) \subset \II(y(s))$ holds for all 
$s \in (t, t + \varepsilon_t] \cap [0,T]$,
and such that 
$I \setminus \AA^-(t) \subset \II(y(s))$ holds for all 
$s \in [t - \varepsilon_t, t) \cap [0,T]$.
Due to its compactness,
we can cover the interval $[0,T]$ by finitely many 
of the intervals $(t - \varepsilon_t, t + \varepsilon_t)$.
We denote the times $t$ of this cover and 
the associated $\varepsilon_t$, $v_t$, and $\lambda_i^t$ by 
$0\leq t_1 \leq ... \leq t_M \leq T$
and $\varepsilon_m$, $v_m$, $\lambda_i^m$,
$i \in \AA(y(t_m))$, $m=1,...,M$, $M \in \N$,
respectively,  and again consider a 
smooth partition of unity $\{\psi_m\}_{m=1}^M$ on $[0,T]$ 
subordinate 
to the cover 
$\{(t_m - \varepsilon_m, t_m + \varepsilon_m)\}_{m=1}^M$,
as in \eqref{eq:POU_properties}.
We assume w.l.o.g.\ 
that the points $t_m$ are distinct and that 
$\psi_m(t_m) = 1$ holds for all $m=1,...,M$
and $\psi_m(t_j) = 0$ for all $m \neq j$.
For this choice of the partition of unity, we 
define 
\begin{equation}
\label{eq.zl_def}
z_l(s) := \sum_{m=1}^M \psi_m(s) v_{m}(s)\qquad s \in [0,T].
\end{equation}
The function $z_l$ obtained along these lines
is clearly in $BV_r([0,T]; \R^d)$ and
satisfies 
\[
\int_0^T \langle z_l, \dd w \rangle = 0
\qquad
\text{and}
\qquad
\sup_{s \in [0,T]}
|z_l(s) - z(s)|
\leq\frac1l
\]
by the exact same arguments as in the proof of
\cref{th_pointwise_criticality}. 
From the definition of 
the left and the right limit,
the definitions of $\AA^+(t_m)$ and $\AA^-(t_m)$,
our choice of $\varepsilon_m$,
and $z(t) = z(t+) \in \ZZ(y(t))$ for all $t \in [0,T]$,
it follows further that 
$\langle \nu_i, z(t_m -) \rangle \leq 0$
holds for all $i \in \AA^-(t_m)$,
that
$\langle \nu_i, z(t_m+) \rangle \leq 0$
holds for all $i \in \AA(y(t_m))$,
that $\AA(y(s)) \subset \AA^-(t_m)$
holds for all $s \in [t_m - \varepsilon_m, t_m) \cap [0,T]$,
and  
that $\AA(y(s)) \subset \AA(y(t_m))$
holds for all $s \in [t_m, t_m + \varepsilon_m] \cap [0,T]$.
By combining all of this, we obtain that 
$z(t_m -) \in \ZZ(y(s))$ holds for all
$s \in [t_m - \varepsilon_m, t_m) \cap [0,T]$,
that $z(t_m+) \in \ZZ(y(s))$ holds for all
$s \in [t_m, t_m + \varepsilon_m] \cap [0,T]$,
and---due to the definition of $z_l$
and the properties of $\{\psi_m\}$---that 
$z_l(s) \in \ZZ(y(s))$ holds for all $s \in [0,T]$.
In total and in view of \eqref{eq:GR-crit_cone_alts}, 
this shows that the functions $\{z_l\}$ satisfy
$\{z_l\} \subset \KK_{G_r}^{\crit}(y,u) \cap BV_r([0,T]; \R^d)$
and 
$\| z_l - z \|_\infty \leq 1/l \to 0$
for $l \to \infty$ as desired. 
That $\|z_l\|_\infty \leq \|z\|_\infty$
holds for all $l$ follows immediately 
from \eqref{eq.zl_def}. 
This completes the construction of the sequence 
$\{z_l\}$.

It remains to construct the functions
$\{z_{k,l}\}$. To this end, 
we consider a function $\varphi \in C^\infty(\R)$ 
satisfying 
\[
0 \leq \varphi(s) \leq 1~\forall s \in \R,\quad 
\varphi(s) = 1~\forall s \in (-\infty, -1],\quad 
\varphi(s) = 0~\forall s \in [0,\infty),
\]
and define 
\begin{equation}
\label{eq:randomeq3838dn_zkl}
z_{k,l}(s) 
:= \sum_{m=1}^M 
\left [
\varphi\left ( \frac{s - t_m + 1/k}{1/k} \right )
+ \mathds{1}_{[t_m, T]}(s)
\right ]
\psi_m(s) v_{m}(s)
\qquad s \in [0,T]
\qquad
k,l\in \N.
\end{equation}
Note that these functions
trivially satisfy 
$\|z_{k,l}\|_\infty \leq \|z\|_\infty $
for all $l,k$ and 
$z_{k,l}(s) \to z_l(s)$ for all $s \in [0,T]$
and $k \to \infty$. Moreover, 
it holds $\{z_{k,l}\} \subset \KK_{G_r}^{\crit}(y,u) \cap BV_r([0,T]; \R^d)$
for the exact same reasons as 
in the case of the sequence $\{z_l\}$
and 
$z_{k,l}(s-) = 0$ for all
$s \in [0,T]$ with $z_{k,l}(s-) \neq z_{k,l}(s)$
(namely, the points $t_m$ with $v_m(t_m)\neq 0$)
by construction. 
To see that 
$\{z_{k,l}\} \subset  \KK_{BV_r}^{\rad, \crit, 0}(y,u)$
holds, it remains to show that 
there exists $\tau > 0$
such that 
$y(s) + \tau z_{k,l}(s) \in Z$
holds for all $s \in [0,T]$. 
To this end, let 
$m \in \{1,...,M\}$ be  fixed. 
From our choice of $\varepsilon_m$,
we obtain that 
$\AA(y(s)) \subset \AA(y(t_m))$
holds for all $s \in [t_m, t_m + \varepsilon_m] \cap [0,T]$
and that $z(t_m) = z(t_m+) \in \ZZ(y(t_m))$.
This implies that 
\[
\langle \nu_i, y(s) + \tau z(t_m+) \rangle \leq \alpha_i
\quad \forall i \in \AA(y(t_m))
\quad \forall s \in [t_m, t_m + \varepsilon_m] \cap [0,T]
\quad \forall \tau > 0
\]
and
\[
\langle \nu_i, y(s) \rangle < \alpha_i
\quad \forall i \in \II(y(t_m))
\quad \forall s \in [t_m, t_m + \varepsilon_m] \cap [0,T].
\]
Due to the continuity of $y$, the above shows that 
there exists $\hat \tau_0 > 0$ satisfying 
$y(s) + \tau z(t_m+) \in Z$
for all $0 < \tau < \hat \tau_0$ and 
all $s \in [t_m, t_m + \varepsilon_m] \cap [0,T]$.
Let us next consider the interval 
$[t_m - \varepsilon_m, t_m) \cap [0,T]$.
Due to our choice of $\varepsilon_m$
and the definitions of 
$\AA^-(t_m)$ and $z(t_m -)$,
it holds 
\[
\langle \nu_i, y(s) + \tau z(t_m-) \rangle \leq \alpha_i
\quad \forall i \in \AA^-(t_m)
\quad \forall s \in [t_m - \varepsilon_m, t_m] \cap [0,T]
\quad \forall \tau > 0
\]
and
\[
\langle \nu_i, y(s) \rangle < \alpha_i
\quad \forall i \in I \setminus \AA^-(t_m)
\quad \forall s \in [t_m - \varepsilon_m, t_m - 1/k] \cap [0,T].
\]
Due to the continuity of $y$, this implies 
that 
there exists $\tilde \tau_0 \in (0, \hat \tau_0)$ satisfying 
$y(s) + \tau z(t_m-) \in Z$
for all $0 < \tau < \tilde \tau_0$ and 
all $s \in [t_m - \varepsilon_m, t_m - 1/k] \cap [0,T]$.
In total, we have now shown that 
\[
y(s) + \tau\left (\mathds{1}_{[t_m - \varepsilon_m, t_m - 1/k] \cap [0,T]}(s)z(t_m-) +
\mathds{1}_{[t_m, t_m + \varepsilon_m] \cap [0,T]}(s)z(t_m+)
\right ) \in Z
\qquad \forall s\in [0,T]\qquad \forall \tau \in (0, \tilde \tau_0).
\]
Due to the properties of $\psi_m$ and $\varphi$,
the definition of $v_m$,
and the convexity of $Z$,
this yields
\[
y(s) + \tau
\left [
\varphi\left ( \frac{s - t_m + 1/k}{1/k} \right )
+ \mathds{1}_{[t_m, T]}(s)
\right ]
\psi_m(s) v_{m}(s) \in Z
\qquad \forall s\in [0,T]\qquad \forall \tau \in (0, \tilde \tau_0).
\]
Due to the convexity of $Z$ and the definition \eqref{eq:randomeq3838dn_zkl},
it now follows immediately that
there exists $\tau > 0$
such that 
$y(s) + \tau z_{k,l}(s) \in Z$
holds for all $s \in [0,T]$. 
This completes the proof. 
\end{proof}

\begin{remark}~
\begin{itemize}
\item In the special case $d=1$ and $Z = [-r,r]$, $r>0$, 
a density result
analogous to that in \cref{theorem:tempoly} has already been proven 
in \cite[Theorem 6.5]{BrokateChristof2023} (with roughly the same techniques). Note that, in the latter theorem, 
the approximation is 
even accomplished with elements of the set
$\KK_{BV_r}^{\rad, \crit, 0}(y,u) \cap C^\infty([0,T])$. 
The price that one pays for this improved smoothness 
(that is essential in the optimal control context of 
 \cite{BrokateChristof2023})
 is that density is only obtained for a 
 certain subset of the critical cone $\KK_{G_r}^{\crit}(y,u)$,
 which, however, still contains the right-limit functions 
 of the directional derivatives of $\SS$ in the one-dimensional setting.
 Whether \cref{theorem:tempoly} can be refined in a similar manner for $d>1$
 is an open problem. 
 For further remarks  on this topic, we refer to \cref{sec:5}.
\item As already mentioned, 
in the elliptic setting, density properties 
similar to that in \cref{theorem:tempoly} are typically referred 
to as \emph{polyhedricity}-results and 
well established; see \cite{ChristofPhd2018,Haraux1977,Harder2017,Wachsmuth2019}. 
A major difference between our notion of temporal polyhedricity 
and the classical concept of polyhedricity is that, 
in \eqref{eq:temp_poly}, the approximation is not 
performed w.r.t.\ a mode of convergence 
that is ``natural''
for the underlying EVI \eqref{eq:EVI}
(this would be convergence in $CBV([0,T];\R^d)$ or $C([0,T];\R^d)$ 
as these are the spaces in which Lipschitz estimates are available 
for $\SS$)
but w.r.t.\ pointwise convergence. Working with this weaker 
notion of convergence is crucial, given that the 
directional derivatives of $\SS$ are typically
neither left- nor right-continuous and, thus, 
cannot be approximated with elements 
of the spaces $BV_r([0,T];\R^d)$, 
$CBV([0,T];\R^d)$, or $C([0,T];\R^d)$
when convergence w.r.t.\ the norm $\|\cdot\|_\infty$
is considered.

\end{itemize}
\end{remark}

\subsection{Unique characterization of the limit of the difference quotients}
\label{subsec:3.6}

We are now in the position to prove that the set 
$\DD$ is a singleton. 

\begin{lemma}[EVI for limits of difference quotients]
\label{lem:DD_system}
For every $\delta \in \DD$, it holds
\begin{equation}
\label{eq:DD-VI_prot}
\int_{s_1}^{s_2}
\langle  z , \dd (\delta-h) \rangle
+
\int_{s_1}^{s_2}
\langle \delta, \dd h \rangle
- \frac12 | \delta(s_2) |^2
+ \frac12 | \delta(s_1) |^2
\geq 0
\qquad 
\forall\, 0\leq s_1 < s_2 \leq T
\qquad
\forall z \in \KK_{G_r}^{\crit}(y,u).
\end{equation}
\end{lemma}

\begin{proof}
Let $\delta \in \DD$
be given,
with a sequence 
$\{\tau_k\}$ and difference 
quotients $\delta_{k} := \delta_{\tau_k}$
as in \cref{def:DD}.
Let 
$0\leq s_1 < s_2 \leq T$ and 
$z \in \KK_{G_r}^{\crit}(y,u)$
be fixed. 
Choose sequences 
$\{z_{m,l}\} \subset \KK_{BV_r}^{\rad, \crit, 0}(y,u)$
and
$ \{z_l\} \subset \KK_{G_r}^{\crit}(y,u)$
for $z$ as in \cref{theorem:tempoly}.
Then it follows from the 
definition of 
the set $\KK_{BV_r}^{\rad, \crit, 0}(y,u)$
and the convexity 
of $Z$ that, for all $m,l$,
we have 
$y(t) + \tau_k z_{m,l}(t) \in Z$
for all $t \in [0,T]$
for all sufficiently large $k$. 
This means that $z_{m,l}$
is admissible in the variational inequality 
\eqref{eq:DQ-EVI} for $\delta_k$
for $k$ large enough
and that we have 
\begin{equation}
\label{eq:randomeq2738398-djdje-2}
\int_{s_1}^{s_2}
 \left \langle z_{m,l} - \delta_k, \dd (\delta_k - h) \right \rangle 
 - \frac{1}{\tau_k}
 \int_{s_1}^{s_2}
 \left \langle z_{m,l} - \delta_k, \dd w \right \rangle 
 \geq 0.
\end{equation}
Due to the inclusion
$\KK_{BV_r}^{\rad, \crit, 0}(y,u) \subset \KK_{G_r}^{\crit}(y,u)$,
\eqref{eq:GR-crit_cone_alts},
\cref{lem:CBV-Stampacchia},
\eqref{eq:proto_criticality},
the $CBV$-regularity of $\delta_k$,
and the integration by parts formula 
in \cref{lem:KS-IBP}, \eqref{eq:randomeq2738398-djdje-2} implies 
\begin{equation*}
- \int_{s_1}^{s_2}
\left \langle z_{m,l} - \delta_k, \dd h \right \rangle 
 +
 \left \langle z_{m,l}(s_2), \delta_k(s_2) \right
 \rangle
 -
  \left \langle z_{m,l}(s_1), \delta_k(s_1) \right
 \rangle
 -
 \int_{s_1}^{s_2}
  \left \langle \delta_k, \dd  z_{m,l} \right \rangle 
   -
   \frac12 |\delta_k(s_2)|^2
   +
   \frac12 |\delta_k(s_1)|^2
 \geq 0
\end{equation*}
for all large enough $k$.
Using \cref{th:KS-bounded_convergence},
\cref{lem:EVI_diff_quot}, 
and the pointwise convergence $\delta_k(t) \to \delta(t)$
for all $t \in [0,T]$,
we can pass to the limit $k \to \infty$
in the above. This yields 
\begin{equation*}
- \int_{s_1}^{s_2}
\left \langle z_{m,l} - \delta, \dd h \right \rangle 
 +
 \left \langle z_{m,l}(s_2), \delta(s_2) \right
 \rangle
 -
  \left \langle z_{m,l}(s_1), \delta(s_1) \right
 \rangle
 -
 \int_{s_1}^{s_2}
  \left \langle \delta, \dd  z_{m,l} \right \rangle 
   -
   \frac12 |\delta(s_2)|^2
   +
   \frac12 |\delta(s_1)|^2
 \geq 0.
\end{equation*}
Integrating by parts again 
(keeping in mind that both 
the integrand and the integrator may now possess 
discontinuities), 
it follows that 
\begin{equation*}
\begin{aligned}
&- \int_{s_1}^{s_2}
\left \langle z_{m,l} - \delta, \dd h \right \rangle 
 +
\int_{s_1}^{s_2}
  \left \langle z_{m,l}, \dd \delta  \right \rangle 
   -
   \frac12 |\delta(s_2)|^2
   +
   \frac12 |\delta(s_1)|^2
   \\
   & \qquad\qquad-
  \sum_{s_1 < t \leq s_2}
  \left \langle
  \delta(t) - \delta(t-) 
  ,
  z_{m,l}(t) - z_{m,l}(t-) 
  \right \rangle
  +
  \sum_{s_1 \leq t < s_2}
  \left \langle
   \delta(t) - \delta(t+)
  ,
   z_{m,l}(t) - z_{m,l}(t+)
  \right \rangle
 \geq 0.
 \end{aligned}
\end{equation*}
Since $z_{m,l}$
is right-continuous with 
zero left limits at all points of discontinuity,
the last inequality simplifies to
\begin{equation*}
\begin{aligned}
&- \int_{s_1}^{s_2}
\left \langle z_{m,l} - \delta, \dd h \right \rangle 
 +
\int_{s_1}^{s_2}
  \left \langle z_{m,l}, \dd \delta  \right \rangle 
   -
   \frac12 |\delta(s_2)|^2
   +
   \frac12 |\delta(s_1)|^2
   \\
   & \qquad\qquad-
  \sum_{s_1 < t \leq s_2,~z_{m,l}(t) \neq z_{m,l}(t-)}
  \left \langle
  \delta(t) - \delta(t-) 
  ,
  z_{m,l}(t) 
  \right \rangle
 \geq 0.
 \end{aligned}
\end{equation*}
Recall that $\delta(t) = \pi_{\ZZ(y(t))}(\delta(t-))$
holds for all $t \in [0,T]$ by 
\cref{th:left_jumps} 
and that 
$z_{m,l}(t) \in \ZZ(y(t))$
by the definition of $\KK_{G_r}^{\crit}(y,u)$.
In view of \eqref{eq:randomeq_z_Zcirc},
this yields
\[
\sum_{s_1 < t \leq s_2,~z_{m,l}(t) \neq z_{m,l}(t-)}
  \left \langle
  \delta(t) - \delta(t-) 
  ,
  z_{m,l}(t) 
  \right \rangle
  =
  \sum_{s_1 < t \leq s_2,~z_{m,l}(t) \neq z_{m,l}(t-)}
  \left \langle
  \pi_{\ZZ(y(t))}(\delta(t-)) - \delta(t-) 
  ,
  z_{m,l}(t) 
  \right \rangle
  \geq 0.
\]
We may thus conclude that 
\begin{equation*}
\begin{aligned}
&- \int_{s_1}^{s_2}
\left \langle z_{m,l} - \delta, \dd h \right \rangle 
 +
\int_{s_1}^{s_2}
  \left \langle z_{m,l}, \dd \delta  \right \rangle 
   -
   \frac12 |\delta(s_2)|^2
   +
   \frac12 |\delta(s_1)|^2
 \geq 0.
 \end{aligned}
\end{equation*}
Letting $m$ go to infinity in the above
and then $l$, 
using \cref{th:KS-bounded_convergence}
and the properties of $\{z_{m,l}\}$
and $\{z_l\}$, we arrive 
at the desired variational inequality \eqref{eq:DD-VI_prot}.
\end{proof}

Note that, if $\delta$ is continuous on $[0,T]$,
then \cref{lemma:KS_lin_add,lem:CBV-Stampacchia}
allow to rewrite 
\eqref{eq:DD-VI_prot} as 
\begin{equation}
\label{eq:randome_strong_EVI}
\int_{s_1}^{s_2}
\langle  z -\delta , \dd (\delta-h) \rangle
\geq 0
\qquad 
\forall\, 0\leq s_1 < s_2 \leq T
\qquad
\forall z \in \KK_{G_r}^{\crit}(y,u).
\end{equation}
The EVI \eqref{eq:DD-VI_prot} can thus be 
interpreted as a weak formulation of \eqref{eq:randome_strong_EVI};
cf.\ \cite[Theorem 4.1, System (4.2)]{Christof2019parob}. 
By combining \eqref{eq:DD-VI_prot} 
with the results of the previous subsections, 
we now arrive at:

\begin{theorem}[uniqueness of limits of difference quotients]
\label{th:DD_unique}
The set $\DD$ is a singleton and its sole element 
is uniquely 
characterized by the system 
\begin{equation}
\label{eq:unique_DD_system}
\begin{aligned}
&\delta \in BV([0,T];\R^d),
\qquad
\delta_+ \in \KK_{G_r}^{\crit}(y,u),
\qquad
\delta(0) = h_0,\\
&\delta(t) = \pi_{\ZZ(y(t))}(\delta(t-))
~\forall t \in [0,T],
\qquad
\delta(t+) = \pi_{V^+(t)}(\delta(t))
~\forall t \in [0,T],\\
&\int_{s_1}^{s_2}
\langle  z , \dd (\delta-h) \rangle
+
\int_{s_1}^{s_2}
\langle \delta, \dd h \rangle
- \frac12 | \delta(s_2) |^2
+ \frac12 | \delta(s_1) |^2
\geq 0
\quad 
\forall 0\leq s_1 < s_2 \leq T
\quad
\forall z \in \KK_{G_r}^{\crit}(y,u).
\end{aligned}
\end{equation}
\end{theorem}
\begin{proof}
From \cref{lem:DD_nonempty}, 
it follows that $\DD$ is nonempty, 
and 
from 
\cref{th:left_jumps,th:jump_right},
\cref{cor:DDsubsetCritCone}, \eqref{eq:DQ-EVI},
and \cref{lem:DD_system},
we obtain that all elements of $\DD$
satisfy the conditions in \eqref{eq:unique_DD_system}.
To establish the assertion of the 
theorem, it remains to prove that the system 
\eqref{eq:unique_DD_system}
can have at most one solution. 
To this end, let us assume that 
$\delta, \eta \in BV([0,T]; \R^d)$
both solve  \eqref{eq:unique_DD_system}
and that a number $\gamma > 0$ is given. 
As $\delta$ and $\eta$ are of bounded variation, 
we can find points 
$0 =: t_0 < t_1 <...<  t_{M-1} < t_M := T$, $M \in \N$,
such that 
\begin{equation}
\label{eq:randomeq2z8h989ih2w}
4 \|\eta\|_\infty
\sum_{t \in [0,T] \setminus \{t_0,...,t_M\}}
|\delta(t+) - \delta(t-)|
\leq \gamma.
\end{equation}
Let $j \in \{1,...,M\}$ be arbitrary 
and let $s_1, s_2 \in [0,T]$ 
be points such that 
$t_{j-1} < s_1 < s_2 < t_{j}$ holds 
and such that both $\delta$ and $\eta$
are continuous at $s_1$ and $s_2$. 
(Recall that the set of points, where one of the functions $\delta$ and $\eta$ is discontinuous,
is at most countable.)
Then, by choosing $\eta_+$ as the test function 
in the variational inequality for $\delta$,
by choosing $\delta_+$ as the test function 
in the variational inequality for $\eta$,
by adding the resulting inequalities, 
by exploiting the $CBV$-regularity of $h$
and \cref{lem:KS-int_limits}, 
by invoking \cref{lem:KS-sub_int}, 
and by using the integration by parts 
formula in \cref{lem:KS-IBP}, we obtain that
\begin{equation*}
\begin{aligned}
0 &\leq 
\int_{s_1}^{s_2}
\langle  \eta_+ , \dd  \delta  \rangle
- \frac12 | \delta(s_2) |^2
+ \frac12 | \delta(s_1) |^2
+
\int_{s_1}^{s_2}
\langle  \delta_+ , \dd \eta  \rangle
- \frac12 | \eta(s_2) |^2
+ \frac12 | \eta(s_1) |^2
\\
&= 
\int_{s_1}^{s_2}
\langle  \eta_+ , \dd  \delta_+  \rangle
+
\int_{s_1}^{s_2}
\langle  \delta_+ , \dd \eta_+  \rangle
- \frac12 | \delta(s_2) |^2
+ \frac12 | \delta(s_1) |^2
- \frac12 | \eta(s_2) |^2
+ \frac12 | \eta(s_1) |^2
\\
&= 
\langle \delta(s_2), \eta(s_2) \rangle
-
\langle \delta(s_1), \eta(s_1) \rangle
+
\sum_{s_1 < t < s_2}
\left \langle \delta(t+) - \delta(t-), \eta(t+) - \eta(t-) \right \rangle 
\\
&\qquad\qquad
- \frac12 | \delta(s_2) |^2
+ \frac12 | \delta(s_1) |^2
- \frac12 | \eta(s_2) |^2
+ \frac12 | \eta(s_1) |^2.
\end{aligned}
\end{equation*}
Using the binomial identities and trivial estimates in the above yields 
\begin{equation*}
\begin{aligned}
\left | \delta(s_2 ) - \eta(s_2)\right |^2
\leq
\left | \delta(s_1 ) - \eta(s_1)\right |^2
+
4\|\eta\|_\infty
\sum_{s_1 < t < s_2}
\left | \delta(t+) - \delta(t-) \right |
\end{aligned}
\end{equation*}
and, after passing to the limits 
$s_1 \to t_{j-1}$ and $s > s_2 \to s \leq t_j$
for an arbitrary $t_{j-1} < s \leq t_j$
(which is possible due to the density of points of continuity),
\begin{equation}
\label{eq:randomeqw78wh2hoi2w}
\begin{aligned}
\left | \delta(s-) - \eta(s-)\right |^2
\leq
\left | \delta(t_{j-1}+) - \eta(t_{j-1}+)\right |^2
+
4\|\eta\|_\infty
\sum_{t_{j-1} < t < t_{j}}
\left | \delta(t+) - \delta(t-) \right |
\qquad 
\forall s \in (t_{j-1}, t_j].
\end{aligned}
\end{equation}
By plugging the projection identities in 
\eqref{eq:unique_DD_system} into \eqref{eq:randomeqw78wh2hoi2w} and by exploiting 
\eqref{eq:1projLip},
it now follows that 
\begin{equation*}
\begin{aligned}
\left | \delta(s-) - \eta(s-)\right |^2
\leq
\left | \delta(t_{j-1}-) - \eta(t_{j-1}-)\right |^2
+
4\|\eta\|_\infty
\sum_{t_{j-1} < t < t_{j}}
\left | \delta(t+) - \delta(t-) \right |
\quad 
\forall s \in (t_{j-1}, t_j]
\quad
\forall j=1,...,M,
\end{aligned}
\end{equation*}
which, after a trivial induction, yields
\begin{equation*}
\begin{aligned}
\left | \delta(s-) - \eta(s-)\right |^2
\leq
\left | h_0 - h_0\right |^2
+
4\|\eta\|_\infty
\sum_{j=1}^M
\left (
\sum_{t_{j-1} < t < t_{j}}
\left | \delta(t+) - \delta(t-) \right |
\right )
\leq \gamma
\qquad \forall s \in (0, T].
\end{aligned}
\end{equation*}
Here, in the last step, we have used \eqref{eq:randomeq2z8h989ih2w}.
As $\gamma > 0$ was arbitrary, 
the above shows that $\delta_- = \eta_-$ holds in $(0,T]$.
Due to the initial condition and the projection identities in 
\eqref{eq:unique_DD_system}, this 
implies $\eta = \delta$ as desired. 
This shows that $\DD$ is indeed a singleton and 
completes the proof. 
\end{proof}

As already mentioned
in \cref{subsec:3.1}, 
with \cref{th:DD_unique}
at hand, assertion \ref{main_sum:item:I}
of \cref{th:main_summary} follows immediately. 
For the sake of completeness, we state 
this result in the following theorem
in a formulation that is independent of \cref{ass:sec:3}.

\begin{theorem}[point \ref{main_sum:item:I}
of \cref{th:main_summary}]
\label{th:dir_diff_final_contra}
Let $T>0$ be given and 
let $Z \subset \R^d$, $d \in \N$, be a 
full-dimensional convex 
non-obtuse polyhedron.
Then the solution map 
$\SS\colon CBV([0,T];\R^d) \times Z \to CBV([0,T];\R^d)$,
$(u, y_0) \mapsto y$, 
of \eqref{eq:EVI}
is directionally differentiable 
in the sense that, for all tuples 
$(u, y_0) \in CBV([0,T]; \R^d) \times Z$
and all directions 
$(h, h_0) \in CBV([0,T]; \R^d) \times \R^d$ 
satisfying $y_0 + \tau_0 h_0 \in Z$ for 
some $\tau_0 >0$,
there exists a unique 
$\delta := \SS'((u, y_0);(h, h_0))  \in BV([0,T];\R^d)$
satisfying
\begin{equation}
\label{eq:randomeq_diff_quot_convergence-42}
\lim_{(0,\tau_0) \ni \tau \to 0}
\frac{\SS(u + \tau h, y_0 + \tau h_0)(t) - \SS(u, y_0)(t)}{\tau}
= \delta(t)\qquad \forall t \in [0,T].
\end{equation}
Moreover, the
directional derivative 
$\delta$ of $\SS$ at $(u,y_0)$ in direction $(h, h_0)$
is uniquely characterized by the 
system \eqref{eq:unique_DD_system},
with the set $\KK_{G_r}^{\crit}(y,u)$
defined 
by \eqref{eq:GR-crit_cone_alts}
with $y := \SS(u, y_0)$ and $w := \PP(u, y_0)$
and the spaces $V^\pm(t)$ 
defined 
as in \cref{def:temporal_indices}
w.r.t.\ 
an arbitrary but fixed standard description $\{(\nu_i, \alpha_i)\}_{i \in I}$ of $Z$.
\end{theorem}

\begin{proof}
If $(u, y_0)$ and 
$(h, h_0)$ are tuples 
as in the assertion of the theorem, 
then  \cref{ass:sec:3} holds
and it follows from 
the definition of the set $\DD$
in \eqref{eq:DD_def}, a trivial contradiction argument
based on 
\cref{th:HellySelect,th:CBV-Lipschitz},
and \cref{th:DD_unique} that 
$\SS$ is pointwisely directionally differentiable 
at $(u, y_0)$ in direction $(h, h_0)$
in the sense of \eqref{eq:randomeq_diff_quot_convergence-42}. That the directional derivative 
$\SS'((u, y_0);(h, h_0))$ is uniquely characterized by 
the system \eqref{eq:unique_DD_system}
is also a straightforward consequence 
of \cref{th:DD_unique}.
This completes the proof. 
\end{proof}

\subsection{Relation to the known one-dimensional results}
\label{subsec:3.7}

We conclude \cref{sec:3} by 
discussing how 
\cref{th:dir_diff_final_contra} relates 
to the results for the
one-dimensional play and stop
proved
in 
\cite{Brokate2015,Brokate2020,Brokate2021,BrokateChristof2023}.
In all what follows, 
we still assume that 
$d$, $T$, $Z$, $\{(\nu_i, \alpha_i)\}_{i \in I}$,
$y_0$, $u$, $y$, $w$, $h_0$, and $h$
are as in \cref{ass:sec:3}.
Note that, in the case $d=1$, 
every full-dimensional convex polyhedron $Z$
is of the form 
$[a,b]$, $(-\infty, b]$, $[a, \infty)$, or $\R$
with $-\infty < a  < b < \infty$. 
In particular, every such polyhedron is 
non-obtuse with a standard description 
involving (at most) the outer unit normals
$1$ and $-1$; see \cref{def:properties_descript}.
We begin with an observation 
on the jumps of the directional derivatives 
in the one-dimensional case.

\begin{corollary}[jumps in 1d]
\label{cor:1d_jumps}
If $d=1$ holds, then the 
directional derivative 
$\delta := \SS'((u, y_0);(h, h_0))$ satisfies the following
for all $t \in [0,T]$. 
\begin{enumerate}[label=\roman*)]
\item\label{cor:1d_jumps:i}
$\delta(t) \in \{\delta(t+), \delta(t-)\}$;
\item\label{cor:1d_jumps:ii}
$\delta(t+) = 0$ if $\delta(t+) \neq \delta(t-)$;
\item\label{cor:1d_jumps:iii}
$\delta(t+)(\delta(t+) - \delta(t-)) = 
\delta(t+)(\delta(t+) - \delta(t))   = 
\delta(t)(\delta(t) - \delta(t-)) = 
0$;
\item\label{cor:1d_jumps:iv}
$z(\delta(t) - \delta(t-))  \geq 0$ for all $z \in \ZZ(y(t))$;
\item\label{cor:1d_jumps:v}
$z(\delta(t+) - \delta(t))  = 0$ for all $z \in  V^+(t)$;
\item\label{cor:1d_jumps:vi}
$z(\delta(t+) - \delta(t-))  \geq 0$ for all $z \in \ZZ(y(t)) \cap V^+(t)$.
\end{enumerate}
\end{corollary}
\begin{proof}
For $d=1$, we have 
$\ZZ(y(t)) \in \{(-\infty, 0], [0, \infty), \R\}$
and
$V^+(t) \in \{\{0\}, \R\}$
for all $t \in [0,T]$.
This implies that,
for all $t \in [0,T]$, all $x \in \R$,
and all $F \in \{\pi_{\ZZ(y(t))}, \pi_{V^+(t)}\}$,
we have $F(x) \in \{0, x\}$ and $F(0)=0$.
If we combine these observations with \eqref{eq:randomeq_z_Zcirc}
and the jump 
conditions 
$\delta(t) = \pi_{\ZZ(y(t))}(\delta(t-))$
and
$\delta(t+) = \pi_{V^+(t)}(\delta(t))$
from 
\cref{th:left_jumps,th:jump_right},
then the assertions in
\ref{cor:1d_jumps:i} to 
\ref{cor:1d_jumps:vi}
follow immediately. 
\end{proof}

\Cref{cor:1d_jumps}\ref{cor:1d_jumps:i}
demonstrates that our results are consistent 
with \cite[Theorem 2.1]{Brokate2021} which establishes
that $ \SS'((u, y_0);(h, h_0)) \in BV_{rl}([0,T])$
holds for the one-dimensional stop $\SS$.
Points \ref{cor:1d_jumps:ii},
\ref{cor:1d_jumps:iii},
and \ref{cor:1d_jumps:vi}
of \Cref{cor:1d_jumps}
further reproduce the observations 
made in \cite[Corollary 4.12]{BrokateChristof2023};
cf.\ also the proof of \cite[Theorem 2.1]{Brokate2021}.
We remark that $BV_{rl}$-regularity cannot be expected for 
$\SS'((u, y_0);(h, h_0))$ if $d$ is greater than one;
see \cref{ex:left_righ_jump} below. 
Next, we study how the system 
\eqref{eq:unique_DD_system} is related  
to the EVIs
that have been 
derived for the directional derivatives 
of the one-dimensional stop 
$\SS$ and their right limits 
in \cite[Theorem 2.1]{Brokate2021} and
\cite[Corollary 4.13]{BrokateChristof2023}.

\begin{corollary}[EVI for the right limits of the 
directional derivatives in 1d]
\label{cor:right_lim_EVI}
If $d=1$ holds, then 
the right-limit function 
$\eta := \SS'((u, y_0);(h, h_0))_+ \in BV_r([0,T])$
is the unique solution of the variational inequality  
\begin{equation}
\label{eq:eta_EVI}
\begin{gathered}
\eta \in BV_r([0,T]),\qquad\eta \in \KK_{G_r}^{\crit}(y,u),\qquad \eta(0) = \pi_{V^+(0)}(h_0),
\\
\int_0^T (z  - \eta)\dd(\eta- h) \geq 0 
\quad
\forall z \in \KK_{G_r}^{\crit}(y,u).
\end{gathered} 
\end{equation}
Further, for all $s\in (0,T]$, it holds 
\begin{equation}
\label{eq:eta_EVI_loc}
\int_0^s (z  - \eta)\dd(\eta- h) \geq 0\quad \forall
z \in \KK_{G_r}^{\crit}(y,u).
\end{equation}
\end{corollary}
\begin{proof}
Define $\delta := \SS'((u, y_0);(h, h_0))$
and $\eta := \delta_+$.
We first prove that $\eta$ solves \eqref{eq:eta_EVI} and \eqref{eq:eta_EVI_loc}.
The first line of \eqref{eq:eta_EVI} 
is a straightforward consequence of \eqref{eq:unique_DD_system}.
To establish the second line and \eqref{eq:eta_EVI_loc},
we proceed similarly to the proof of 
\cref{th:DD_unique}. Let $s \in (0,T]$
and $z \in \KK_{G_r}^{\crit}(y,u)$
be given and let $\gamma > 0$ be arbitrarily small.
As $\eta$ has bounded variation, we can
find points 
$0 =: t_0 < t_1 <...<  t_{M-1} < t_M := s$, $M \in \N$,
such that 
\begin{equation}
\label{eq:randomeq2z8h989ih2w-42}
\|\eta\|_\infty
\sum_{t \in [0,s] \setminus \{t_0,...,t_M\}}
|\eta(t) - \eta(t-)|
\leq \gamma.
\end{equation}
Let $j \in \{1,...,M\}$ be arbitrary 
and let $s_1, s_2 \in [0,s]$ 
be points such that 
$t_{j-1} < s_1 < s_2 < t_{j}$ holds 
and such that both $\eta$ and $\delta$
are continuous at $s_1$ and $s_2$.
Then it follows from 
\eqref{eq:unique_DD_system},
\cref{lem:KS-int_limits,lem:KS-sub_int},
the identities 
$\delta(s_1) = \delta(s_1+) = \eta(s_1)$
and
$\delta(s_2) = \delta(s_2+) = \eta(s_2)$,
and 
\cref{lem:KS-IBP} (with $f=g=\eta$)
that
\begin{equation}
\label{eq:randomeq27383i}
\begin{aligned}
0 &\leq 
\int_{s_1}^{s_2}
\langle  z , \dd (\delta-h) \rangle
+
\int_{s_1}^{s_2}
\langle \delta, \dd h \rangle
- \frac12 | \delta(s_2) |^2
+ \frac12 | \delta(s_1) |^2
\\
&=
\int_{s_1}^{s_2}
\langle  z , \dd (\eta-h) \rangle
+
\int_{s_1}^{s_2}
\langle \eta, \dd h \rangle
- \frac12 | \eta(s_2) |^2
+ \frac12 | \eta(s_1) |^2
\\
&=
\int_{s_1}^{s_2}
\langle  z - \eta, \dd (\eta-h) \rangle
+
\frac12 \sum_{s_1 < t < s_2}
|\eta(t) - \eta(t-)|^2
\\
&=
\int_{t_{j-1}}^{t_j}
\mathds{1}_{[s_1, s_2]}
(z - \eta)\dd (\eta-h)
+
\frac12 \sum_{s_1 < t < s_2}
|\eta(t) - \eta(t-)|^2.
\end{aligned}
\end{equation}
As 
$\delta$ and $\eta$ possess at most countably many
points of discontinuity, we may let 
$s_1$ and $s_2$ go to $t_{j-1}$ and $t_j$, respectively,
and invoke the bounded convergence theorem to 
obtain the following from \eqref{eq:randomeq27383i}:
\begin{equation}
\label{eq:randomeq27828sni3i}
0 \leq
\int_{t_{j-1}}^{t_j}
\mathds{1}_{(t_{j-1}, t_j)}
(z - \eta)\dd (\eta-h)
+
\|\eta\|_\infty \sum_{t_{j-1}< t < t_j}
|\eta(t) - \eta(t-)|.
\end{equation}
Note that 
\cref{lem:KS-point-measures}, the right-continuity of 
$\eta$, 
\eqref{eq:Monteiro4.1.9},
points \ref{cor:1d_jumps:iii} and \ref{cor:1d_jumps:vi}
of \cref{cor:1d_jumps},
and the inclusion 
$z(t_j) = z(t_j+) \in \ZZ(y(t_j))\cap V^+(t_j)$
obtained from \eqref{eq:GR-crit_cone_alts}
and $z \in \KK_{G_r}^{\crit}(y,u)$ 
imply
\begin{equation}
\label{eq:important_here_d_1_needed}
\begin{aligned}
\int_{t_{j-1}}^{t_j}
(\mathds{1}_{\{t_{j}\}}
+
\mathds{1}_{\{t_{j-1}\}})
(z - \eta)\dd (\eta-h)
&=
z(t_j)(\eta(t_j) - \eta(t_j-))
-
\eta(t_j)(\eta(t_j) - \eta(t_j-))
\\
&= 
z(t_j)(\delta(t_j+) - \delta(t_j-))
-
\delta(t_j+)(\delta(t_j+) - \delta(t_j-))
\\
&\geq 0.
\end{aligned}
\end{equation}
In combination with \eqref{eq:randomeq27828sni3i}
and \cref{lemma:KS_lin_add}, this yields
\[
0 \leq
\int_{t_{j-1}}^{t_j}
(z - \eta)\dd (\eta-h)
+
\|\eta\|_\infty \sum_{t_{j-1}< t < t_j}
|\eta(t) - \eta(t-)|
\]
and, after summing over $j$ and using \eqref{eq:randomeq2z8h989ih2w-42} and \cref{lemma:KS_lin_add},
\begin{equation}
\label{eq:adding_together_82uhs2ois2}
0 \leq 
\sum_{j=1}^M
\int_{t_{j-1}}^{t_j}
(z - \eta)\dd (\eta-h)
+
\|\eta\|_\infty \sum_{t_{j-1}< t < t_j}
|\eta(t) - \eta(t-)|
\leq
\int_{0}^{s}
(z - \eta)\dd (\eta-h)
+
\gamma.
\end{equation}
Since $\gamma$ was arbitrary, this proves that 
$\eta$ satisfies \eqref{eq:eta_EVI_loc}
for all $s \in (0, T]$ and the second line of 
\eqref{eq:eta_EVI}.

It remains to prove that \eqref{eq:eta_EVI}
has at most one solution. 
To this end, let us assume that $\eta_1$ and $\eta_2$
solve \eqref{eq:eta_EVI}. By choosing $\eta_1$ as the 
test function in the EVI for $\eta_2$ and vice versa,
by adding the resulting inequalities, 
by again invoking \cref{lem:KS-IBP} 
(with $f=g=\eta_1 - \eta_2$), and by exploiting 
that $\eta_1(0) = \eta_2(0)$, we obtain that 
\[
0
\leq
- \int_0^T (\eta_1  - \eta_2)\dd( \eta_1 -\eta_2) 
=
- \frac12 \left (
|(\eta_1  - \eta_2)(T)|^2 + 
\sum_{0 < t \leq T}
| (\eta_1  - \eta_2)(t) - (\eta_1  - \eta_2)(t-)|^2
\right ).
\]
The above implies $\eta_1(T) = \eta_2(T)$ and,
since $\eta_1 - \eta_2$ is in $BV_r([0,T])$, 
$\eta_1 - \eta_2 \in CBV([0,T])$.
Next, we choose 
for an arbitrary but fixed $s \in (0,T)$
the function
$z := \mathds{1}_{[0, s)} \eta_2 + \mathds{1}_{[s, T]}\eta_1$
as the test function in the EVI for $\eta_1$ and 
$z := \mathds{1}_{[0, s)} \eta_1 + \mathds{1}_{[s, T]}\eta_2$
as the test function in the EVI for $\eta_2$
and add the resulting inequalities.
(Note that these functions are indeed elements 
of $\KK_{G_r}^{\crit}(y,u)$, see the last equality in 
\eqref{eq:GR-crit_cone_alts}.)
In combination with 
the continuity of $\eta_1 - \eta_2$,
again \cref{lemma:KS_lin_add,lem:KS-IBP,lem:KS-sub_int},
and the initial conditions,
this yields
\begin{equation*}
\begin{aligned}
0 &\leq
\int_0^T \mathds{1}_{[0, s)} (\eta_2 - \eta_1) \dd(\eta_1- \eta_2)
=
\int_0^s  (\eta_2 - \eta_1) \dd(\eta_1- \eta_2)
=
-\frac12 |\eta_1(s)- \eta_2(s)|^2.
\end{aligned}
\end{equation*}
In combination with what we have proven before and 
the initial condition in \eqref{eq:eta_EVI},
this shows that 
$\eta_1  = \eta_2 $ holds on $[0,T]$.
Thus, \eqref{eq:eta_EVI} has precisely 
one solution, namely
$\SS'((u, y_0);(h, h_0))_+$, and the proof is complete. 
\end{proof}

Note that 
the formulas for $\KK_{G_r}^{\crit}(y,u)$
in
\eqref{eq:GR-crit_cone_alts} and 
\cite[Proposition 4.9]{BrokateChristof2023}
imply that, if 
$d=1$ holds and $K^{\mathrm{ptw}}_{\crit}(y,u) \colon [0, T] \rightrightarrows \R$ is 
defined as in 
\cite[Definition 4.7]{BrokateChristof2023} (i.e., $K^{\mathrm{ptw}}_{\crit}(y,u)(t) := \ZZ(y(t)) \cap V^+(t)$), 
then we have 
\begin{equation}
\label{eq:crit_cone_equivalence}
z \in \KK_{G_r}^{\crit}(y,u)
\quad
\iff
\quad
z \in 
G\left ([0, T]; K^{\ptw}_{\crit}(y,u) \right) \cap 
G_r([0,T]).
\end{equation}
This shows that \cref{cor:right_lim_EVI} 
reproduces the EVI derived in \cite[Corollary 4.13]{BrokateChristof2023},
albeit with the modification that
only right-continuous test functions are considered
and that $h_0$ is also allowed to be nonzero.
(We expect that, by exploiting \eqref{eq:crit_cone_equivalence} and approximation arguments analogous to those in the 
proofs of \cref{th_pointwise_criticality,theorem:tempoly},
one can also get rid of the right-continuity of the test functions  in \cref{cor:right_lim_EVI}. We omit discussing this topic in detail in this paper.)

To obtain a system of EVIs 
that characterizes
the directional derivative 
$\delta := \SS'((u, y_0);(h, h_0))$ itself,
we can exploit the evolutionary nature of 
\eqref{eq:EVI}. This nature 
implies that, by restricting the 
function $\SS(u, y_0)$ to the
interval $[0,s]$, $0 < s < T$, 
we obtain precisely the solution 
of the variational inequality that 
arises when the terminal time $T$
in \eqref{eq:EVI} is replaced by $s$. 
For the directional derivative $\delta$, 
this means that 
$\delta_+$ 
not only satisfies 
the EVI \eqref{eq:eta_EVI} in the case $d=1$,
but also analogue EVIs 
formulated 
on the intervals $[0,s]$,
$0 < s < T$. 
As \eqref{eq:eta_EVI}
uniquely characterizes $\delta_+(T) = \delta(T)$
(due to the conventions for the right limit at the interval endpoint), the corresponding EVIs 
on the intervals 
$[0,s]$ uniquely characterize the values 
$\delta(s)$ for all $0 < s < T$ and, thus, 
$\delta$ in its entirety. Rigorously, we can 
formulate this result as follows:

\begin{corollary}[EVI-system for the 
directional derivatives in 1d]
\label{cor:dir_dif_EVI_1d}
If $d=1$ holds, then 
the directional derivative 
$\delta:= \SS'((u, y_0);(h, h_0))$
is the unique solution of the EVI-system 
\begin{equation}
\label{eq:delta_EVI_1d}
\begin{aligned}
&\delta \in BV([0,T]),
&&
\delta(0) = h_0,
\\
&
\delta_{+,[0,s]} \in \KK_{G_r}^{\crit}(y,u)[0,s],
&&
\int_0^s (z  - \delta_{+,[0,s]})\dd(\delta- h) \geq 0
\quad
\forall z \in \KK_{G_r}^{\crit}(y,u)[0,s]
\quad \forall s \in (0,T].
\end{aligned} 
\end{equation}
Here, $\delta_{+,[0,s]}$ and 
$\KK_{G_r}^{\crit}(y,u)[0,s]$ denote 
the right-limit function and the critical 
cone 
from 
\eqref{eq:GR-crit_cone_alts}
associated with the interval $[0,s]$,
respectively, with the corresponding conventions
for the interval endpoint $s$, i.e., 
\begin{equation}
\label{eq:def_+_subinterval}
\begin{aligned}
\delta_{+,[0,s]}(t)
&:=
\begin{cases}
\delta_+(t) &\text{ if } 0 \leq t < s,
\\
\delta(s) &\text{ if } t=s,
\end{cases}
\\
\KK_{G_r}^{\crit}(y,u)[0,s]
&:=
\left \{
z \in G_r([0,s])
~\left |~
z(t) \in \ZZ(y(t))
~\forall t \in [0,s]
\text{ and }
\int_0^s   z\, \dd w  = 0
\right.
\right \}.
\end{aligned}
\end{equation}
\end{corollary}
\begin{proof}
From \cref{cor:right_lim_EVI},
the arguments outlined above, 
and \eqref{eq:unique_DD_system},
we obtain that $\delta$
satisfies 
$\delta \in BV([0,T])$,
$\delta(0) = h_0$, and 
\begin{equation*}
\label{eq:randomeq-28727-dj7}
\delta_{+,[0,s]} \in \KK_{G_r}^{\crit}(y,u)[0,s],
\quad
\int_0^s (z  - \delta_{+,[0,s]})\dd(\delta_{+,[0,s]} - h) \geq 0
\quad
\forall z \in \KK_{G_r}^{\crit}(y,u)[0,s]
\quad \forall s \in (0,T].
\end{equation*}
Using 
\cref{lemma:KS_lin_add,lem:KS-point-measures,lem:KS-sub_int},
\eqref{eq:Monteiro4.1.9}, \eqref{eq:def_+_subinterval},
\eqref{eq:proj_VI}, \eqref{eq:eta_EVI},
and \eqref{eq:GR-crit_cone_alts},
one further easily checks that 
\begin{equation*}
\label{eq:randomeq-28727-dj7-2}
\begin{aligned}
\int_0^s (z  - \delta_{+,[0,s]})\dd(\delta_{+,[0,s]} - \delta)
&=
\int_0^s (z  - \delta_{+,[0,s]})
(\mathds{1}_{\{0\}} + \mathds{1}_{\{s\}})\dd(\delta_{+,[0,s]} - \delta)
\\
&=
\left (
z(0)  - \delta_{+,[0,s]}(0)
\right )
\left (
\delta_{+,[0,s]}(0+) - \delta(0+)
-
\delta_{+,[0,s]}(0) + \delta(0)
\right )
\\
&\qquad\qquad
+
\left (
z(s)  - \delta_{+,[0,s]}(s)
\right )
\left (
\delta_{+,[0,s]}(s) - \delta(s)
-
\delta_{+,[0,s]}(s-) + \delta(s-)
\right )
\\
&=
\left (
z(0)  - \delta(0+)
\right )
\left (
\delta(0) - \delta(0+)
\right )
\\
&=
\left (
z(0)  - \pi_{V^+(0)}(h_0)
\right )
\left (
h_0 - \pi_{V^+(0)}(h_0)
\right )
\\
&= 0
\qquad \forall z \in \KK_{G_r}^{\crit}(y,u)[0,s].
\end{aligned}
\end{equation*}
By combining all of the above and by exploiting \cref{lemma:KS_lin_add},
we obtain that $\delta$ is indeed a 
solution of \eqref{eq:delta_EVI_1d}.
To see that the system \eqref{eq:delta_EVI_1d} 
possesses at most one solution, 
we can proceed along the lines of 
the proof of \cref{cor:right_lim_EVI}.
If $\delta$ and $\hat \delta$ both solve \eqref{eq:delta_EVI_1d},
then we trivially have 
$\delta(0) = \hat \delta (0)$
and, for every $s \in (0,T]$,
we obtain from the EVIs satisfied by 
$\delta$ and $\hat \delta$,
the initial condition in \eqref{eq:delta_EVI_1d},
and 
\cite[Proposition 6.2]{Brokate2021}
(which is an easy consequence of 
\cref{lem:KS-point-measures,lem:KS-sub_int,lem:KS-IBP})
that 
\begin{equation}
\label{eq:random_estimate_uniqueness:262gs}
0
\leq
\int_0^s (\hat \delta_{+,[0,s]}  - \delta_{+,[0,s]})\dd(\delta- \hat \delta)
\leq
-\frac{1}{2}| \hat \delta_{+,[0,s]}(s)  - \delta_{+,[0,s]}(s)  |^2
=
-\frac{1}{2}| \hat \delta(s)  - \delta(s)  |^2.
\end{equation}
The above implies $\hat \delta(s)  = \delta(s)$
and, consequently, $\hat \delta = \delta$ on $[0,T]$.
This completes the proof of the corollary. 
\end{proof}

The system \eqref{eq:delta_EVI_1d} corresponds 
to the one in 
\cite[Theorem 4.11]{BrokateChristof2023} and 
\cite[Theorem 2.1, Equation (2.6)]{Brokate2021},
as one may easily check by means of the 
equivalence in \eqref{eq:crit_cone_equivalence}.
We would like to emphasize that one has to 
state the system \eqref{eq:delta_EVI_1d}
very carefully to get a unique characterization 
for $\delta$. In particular, it is not 
enough to merely demand 
$\delta_+ \in \KK_{G_r}^{\crit}(y,u)$
in \eqref{eq:delta_EVI_1d},
analogously to the EVI \eqref{eq:eta_EVI},
as this condition misses the 
information $\delta(s) \in \ZZ(y(s))$
for all $s \in (0,T]$
that is encoded in the inclusions
$\delta_{+,[0,s]} \in \KK_{G_r}^{\crit}(y,u)[0,s]$
for all $s \in (0,T]$ via the endpoint 
convention in \eqref{eq:def_+_subinterval}
and that is essential for the estimate 
\eqref{eq:random_estimate_uniqueness:262gs}.

As a final remark, we would like to point out that 
the proof of \cref{cor:right_lim_EVI}
cannot be extended (easily) to the case 
$d>1$. The problem is that, for $d>1$,
it is not necessarily true that 
the directional derivative
$\delta:= \SS'((u, y_0);(h, h_0))$
satisfies
$\langle \delta(t+), \delta(t+) - \delta(t-) \rangle \leq 0$
for all $t \in [0,T]$, 
as used in the final step of \eqref{eq:important_here_d_1_needed}; 
cf.\ \cref{cor:1d_jumps}\ref{cor:1d_jumps:iii}. This makes 
it (seemingly) impossible to correct 
the function values of the integrand 
in \eqref{eq:randomeq27828sni3i} at the 
right endpoints $t_j$ of the decomposition 
intervals $[t_{j-1}, t_j]$ used in \eqref{eq:randomeq2z8h989ih2w-42}
and, thus, to combine the integrals over the 
subintervals $[t_{j-1}, t_j]$ to a meaningful 
global EVI, as done in \eqref{eq:adding_together_82uhs2ois2}.
Note that this problem is also intrinsically related 
to the fact that the directional derivatives of $\SS$
are typically not in $BV_{rl}([0,T];\R^d)$
if $d>1$ holds. 
In view of these effects, 
it seems likely that the characterization of 
$\delta$ by means of the system 
\eqref{eq:unique_DD_system} is the best that 
one can expect for dimensions $d>1$ and 
that it is not possible in general to 
characterize $\delta$ via a classical EVI
without separate jump conditions. 
We conclude \cref{sec:3} with an example 
which demonstrates that situations with  
$\langle \delta(t+), \delta(t+) - \delta(t-) \rangle > 0$
and 
$\delta \not \in BV_{rl}([0,T];\R^d)$
are indeed possible for $d>1$.

\begin{example}[jumps in a 2d situation]
\label{ex:left_righ_jump}
Consider the case
\begin{equation*}
\label{eq:ex_situation}
\begin{gathered}
d=2,\qquad T=2,\qquad
Z = \{x \in \R^2 \mid 0 \leq x_2 \leq x_1\},
\qquad 
h \equiv 0,
\\
y_0 = \begin{pmatrix}
2 \\ 1
\end{pmatrix},
\qquad
h_0 = \begin{pmatrix}
0 \\ 2
\end{pmatrix},
\qquad
u \in W^{1,\infty}((0,2);\R^2),
\quad 
u(t) = 
\mathds{1}_{[0,1]}(t)
\begin{pmatrix}
2 \\ 1
\end{pmatrix}
(1-t)
+
\mathds{1}_{(1,2]}(t)
\begin{pmatrix}
1 \\ -1
\end{pmatrix}
(t-1).
\end{gathered}
\end{equation*}
 For the above data, one easily verifies that 
 the conditions in 
 \cref{ass:sec:3} hold with $\tau_0 = 1/2$ 
 and that 
 the solution $y := \SS(u, y_0)$
 of \eqref{eq:EVI} and the associated 
 play $w := u - y$ are given by 
 \[
y, w \in W^{1,\infty}((0,2);\R^2),
\qquad
y(t)
=
\mathds{1}_{[0,1]}(t)
\begin{pmatrix}
2 \\ 1
\end{pmatrix}
(1-t)
+
\mathds{1}_{(1,2]}(t)
\begin{pmatrix}
1 \\ 0
\end{pmatrix}
(t-1),
\qquad
w(t)
=
\mathds{1}_{(1,2]}(t)
\begin{pmatrix}
0\\ - 1
\end{pmatrix}
(t-1).
 \]
Using the calculus rules from \cref{subsec:2.4}, 
one also easily checks that the system 
\eqref{eq:unique_DD_system} is solved by the function
\begin{equation}
\label{eq:randomeq2828ej}
\delta(t)
:= 
\mathds{1}_{[0,1)}(t)
\begin{pmatrix}
0 \\ 2
\end{pmatrix}
+
\mathds{1}_{\{1\}}(t)
\begin{pmatrix}
1 \\ 1
\end{pmatrix}
+
\mathds{1}_{(1,2]}(t)
\begin{pmatrix}
1 \\ 0
\end{pmatrix}
\end{equation}
in this situation.
As \eqref{eq:unique_DD_system} 
is uniquely solvable by \cref{th:DD_unique},
this implies $\delta = \SS'((u, y_0);(h, h_0))$.
As $\delta$ is clearly 
neither left- nor right-continuous at $t=1$,
this demonstrates that there are indeed 
situations with $\SS'((u, y_0);(h, h_0)) \not \in BV_{rl}([0,T];\R^d)$.
From \eqref{eq:randomeq2828ej}, we further obtain that 
$
\langle \delta(1+), \delta(1+) - \delta(1-) \rangle
= 1 > 0.
$
The inequality 
$\langle \delta(t+), \delta(t+) - \delta(t-) \rangle \leq 0$
that is valid for $d=1$
by \cref{cor:1d_jumps}\ref{cor:1d_jumps:iii}
and that entered the proof 
of \cref{cor:right_lim_EVI} crucially in \eqref{eq:important_here_d_1_needed} is thus indeed 
in general wrong for $d>1$, as mentioned after 
\cref{cor:dir_dif_EVI_1d}.
\end{example}

\section{Proof of \ref{main_sum:item:II}: non-differentiability in the obtuse case}\label{sec:4}

It remains to prove part \ref{main_sum:item:II}
of \cref{th:main_summary}. To establish this second main 
result of our analysis, we 
use \cref{lem_counterexample}
to reduce the situation to two dimensions and 
then explicitly construct a counterexample.

\begin{theorem}[point \ref{main_sum:item:II}
of \cref{th:main_summary}]\label{th:non_dir_diff}
Let $T>0$ be given and 
let $Z \subset \R^d$, $d \in \N$, be a 
full-dimensional convex polyhedron
that is not non-obtuse. 
Let $\vartheta \in (0,T)$ be an arbitrary number. 
Then there exist
$y_0 \in Z$, $u \in W^{1,\infty}((0,T);\R^d)$,
$h_0 \in \ZZ(y_0)$, $\tau_0 > 0$, and $h \in C^\infty([0,T];\R^d)$
such that 
$y_0 + \tau h_0 \in Z$ holds for all $ \tau \in [0,\tau_0]$,
such that the solution operator 
$\SS\colon CBV([0,T];\R^d) \times Z \to CBV([0,T];\R^d)$
of \eqref{eq:EVI} satisfies 
\[
\left. \frac{\SS(u, y_0 + \tau h_0)  - \SS(u, y_0) }{\tau}\right |_{[\vartheta, T]}
=
\left. \frac{\SS(u + \tau h, y_0) - \SS(u, y_0) }{\tau} \right |_{[\vartheta, T]}
= \mathrm{const}
= c_\tau \qquad \forall\,0 < \tau \leq \tau_0
\]
for some $c_\tau \in \R^d$,
and such that the limit $\lim_{(0, \tau_0) \ni \tau \to 0} c_\tau$
does not exist.
\end{theorem}

\begin{proof}
Let $T$ and $Z$ 
be as in the theorem, let $\{(\nu_i, \alpha_i)\}_{i \in I}$
be a standard description of $Z$,
and let $\vartheta \in (0,T)$ be given.
From \cref{lem_counterexample}, we obtain that 
there exists $w \in Z$ such that $\AA(w) = \{i_1, i_2\}$ holds for some 
$i_1, i_2\in I$ with $\langle \nu_{i_1}, \nu_{i_2}\rangle > 0$.
We assume w.l.o.g.\ that $w=0$, $\alpha_{i_1} = \alpha_{i_2} = 0$,
and $i_1 = 1$, $i_2 = 2$.
(This can always be achieved by 
translating $Z$ and reordering $\{(\nu_i, \alpha_i)\}_{i \in I}$.)
As $\AA(0) = \{1,2\}$, we can find $\varepsilon \in (0, \vartheta)$
such that $\langle \nu_i, x\rangle < \alpha_i$ holds for all
$i \in I \setminus \{1,2\}$ and all $x \in B_\varepsilon(0)$. 
Since $\langle \nu_{1}, \nu_{2}\rangle > 0$ and 
$\nu_1 \neq \nu_2$, there further exist
unique $e_1, e_2 \in \R^d$ satisfying 
$|e_1| = |e_2| = 1$, 
$e_1, e_2 \in \spann(\{\nu_1, \nu_2\})$,
$\langle \nu_1, e_1\rangle = \langle \nu_2, e_2 \rangle = 0$,
$\langle \nu_1, e_2\rangle < 0$, $\langle \nu_2, e_1\rangle < 0$,
and $\langle e_1, e_2\rangle < 0$
(cf.\ \cref{lem:ZZproperties}\ref{lem:ZZproperties:iv}).
Consider now 
an arbitrary but fixed 
number $\gamma \in [0,1]$
and choose functions 
$\varphi \in C_c^\infty(\R)$
and 
$\beta_1 \in L^\infty(0,T)$
such that
\begin{equation*}
\begin{gathered}
\supp(\varphi) \subset \left (0, \frac\varepsilon2\right ),
\qquad
\varphi \geq 0,
\qquad
\int_0^T \varphi\,\dd s = 1,
\qquad
\beta_1 = 0 \text{ a.e.\ in }(\varepsilon, T),
\qquad
0 \leq
\beta_1 
\leq 
\frac{\langle e_1, e_2\rangle}{\langle \nu_1, e_2\rangle}
\text{ a.e.\ in } (0, \varepsilon),
\end{gathered}
\end{equation*}
and such that the limit 
\begin{equation}
\label{eq:beta1limit}
\lim_{(0, \varepsilon) \ni \tau \to 0}
\frac1\tau
\int_{\varepsilon - \tau}^\varepsilon
\beta_1(s)
\dd s
\end{equation}
does not exist. (Functions
$\varphi$ and $\beta_1$ with these properties 
can be 
constructed easily.) We define 
\begin{gather*}
y_0 := \varepsilon e_1,
\qquad
\tau_0 := \frac{\varepsilon}{2},
\qquad h_0 := -\gamma e_1,
\\
\lambda_1(t) := 
\int_{0}^t
\beta_1(s) \dd s,
\qquad
u(t)
:= 
 \max(0, \varepsilon - t) 
 e_1
+
\lambda_1(t) \nu_1,
\qquad h(t) := 
-(1-\gamma)\int_0^t \varphi(s) \dd s\,e_1.
\end{gather*}
For this data, 
one easily checks that 
$y_0 + \tau h_0 \in Z$
for all $\tau \in [0, \tau_0]$,
$u \in W^{1,\infty}((0,T);\R^d)$,
and $h \in C^\infty([0,T];\R^d)$.
We claim that
$\SS(u + \tau h, y_0 + \tau h_0)$
is equal to the function
\begin{equation}
\label{eq:ytaudef}
y^\tau(t) 
:=
\begin{cases}
\displaystyle
(\varepsilon  -\gamma \tau - t) 
e_1
+
\tau h(t)
& \text{ if } t \in [0, \varepsilon -\tau],
\\
\displaystyle
\int_{\varepsilon - \tau}^t 
\left \langle \dot u(s), e_2 \right \rangle 
\dd s\, e_2
& \text{ if } t \in  (\varepsilon - \tau, T],
\end{cases}
\end{equation}
for all $\tau \in [0,\tau_0]$.
To see this, 
let $\tau \in [0,\tau_0]$ be fixed. 
Due to \eqref{eq:ytaudef}
and $0 \leq \tau \leq \tau_0 < \varepsilon$,
we trivially have $y^\tau(0) = 
(\varepsilon  -\gamma \tau) 
e_1 = y_0 + \tau h_0$.
From the properties 
of $\varphi$ and the definition
$\tau_0:= \varepsilon/2$,
we further obtain that 
\begin{equation*}
\begin{aligned}
(\varepsilon  -\gamma \tau - (\varepsilon - \tau))
e_1
+
\tau h(\varepsilon - \tau)
&=
(1-\gamma)\tau e_1
-(1-\gamma) \tau\int_0^{\varepsilon - \tau} \varphi(s) \dd s\,e_1
\\
&=
(1-\gamma)\tau e_1
-(1-\gamma) \tau\int_0^{\varepsilon/2} \varphi(s) \dd s\,e_1
=
0
=
\int_{\varepsilon - \tau}^{\varepsilon - \tau} 
\left \langle \dot u(s), e_2 \right \rangle 
\dd s\, e_2.
\end{aligned}
\end{equation*}
In combination with \eqref{eq:ytaudef},
this yields $y^\tau \in W^{1,\infty}((0,T);\R^d)$.
Similarly, one also checks that 
\[
0 \leq 
\varepsilon  -\gamma \tau -t 
 - (1-\gamma) \tau \int_0^t \varphi(s)\dd s
 \leq 
 \varepsilon - \gamma \tau \leq \varepsilon 
 \qquad 
 \forall t \in [0, \varepsilon - \tau]
 \]
and, thus, 
\[
y^\tau(t)
=
 \left (\varepsilon  -\gamma \tau - t\right ) e_1
+
\tau h(t)
=
 \left ( \varepsilon  -\gamma \tau - t
 - (1-\gamma)\tau  \int_0^t \varphi(s)\dd s
 \right ) e_1
 \in \conv(\{0, \varepsilon e_1\})
 \subset Z
\qquad
\forall t \in [0, \varepsilon - \tau].
\]
For the derivative $\dot y^\tau$, we obtain 
\[
\dot y^\tau(t) - (\dot u + \tau \dot h)(t)
= 
-e_1 + \tau \dot h(t)
-
\left (-e_1 + \dot \lambda_1(t)\nu_1 
+
\tau 
\dot h(t)
\right )
=
-\beta_1(t) \nu_1
\in
-\NN_Z(y^\tau(t))
\text{ for a.a.\ } t \in (0, \varepsilon - \tau).
\]
Thus,  
$y^\tau(0) = y_0 + \tau h_0$,
$y^\tau(t) \in Z$ for all $t \in [0, \varepsilon-\tau]$,
and 
$\dot y^\tau(t) - (\dot u + \tau \dot h)(t)
 \in -\NN_Z(y^\tau(t))$
for a.a.\ $t \in (0, \varepsilon-\tau)$.
To prove that
$y^\tau = \SS(u + \tau h, y_0 + \tau h_0)$ holds,
it remains to check that the last
inclusions also hold on the interval $(\varepsilon - \tau,T]$; see 
\cref{prop:unique_solvability_stop_play}.
To this end, we note that 
our assumptions on $\beta_1$ yield
\begin{equation}
\label{eq:randome2329839e3jp-2}
\left \langle \dot u(t), e_2 \right \rangle 
=
\mathds{1}_{[0, \varepsilon]}(t)
\left \langle 
-  e_1
, e_2 \right \rangle 
+
\beta_1(t)
\left \langle 
 \nu_1
, e_2 \right \rangle 
=
\left (\mathds{1}_{[0, \varepsilon]}(t)
-
\frac{\left \langle 
 \nu_1
, e_2 \right \rangle }{\left \langle 
e_1
, e_2 \right \rangle } \beta_1(t) \right )
\left \langle 
-  e_1
, e_2 \right \rangle 
\in [0,1]
~
\text{ for a.a.\ } t \in (\varepsilon - \tau, T).
\end{equation}
This implies 
\begin{equation}
\label{eq:randomeq8hw2982wh-z6}
\begin{aligned}
y^\tau(t)
=
\int_{\varepsilon - \tau}^t 
\left \langle \dot u(s), e_2 \right \rangle 
\dd s\, e_2
=
\int_{\varepsilon - \tau}^{\min(\varepsilon, t)} 
\left \langle \dot u(s), e_2 \right \rangle 
\dd s\, e_2
 \in \conv(\{0, \tau e_2\})
 \subset \conv(\{0, \varepsilon e_2\})
\subset Z
\quad 
\forall t \in (\varepsilon - \tau, T].
\end{aligned}
\end{equation}
Similarly, we may also compute that 
\[
\left \langle 
\dot u(t), \nu_2
\right \rangle
=
\mathds{1}_{[0, \varepsilon]}(t)
\left \langle 
-e_1, \nu_2
\right \rangle
+
\beta_1(t)
\left \langle 
 \nu_1, \nu_2
\right \rangle
\geq 0
\text{ for a.a.\ } t \in (\varepsilon - \tau, T).
\]
In combination with the 
properties of $\varphi$,
the estimate $\varepsilon - \tau \geq \varepsilon/2$,
the inclusion 
$y^\tau(t) \in \conv(\{0, \varepsilon e_2\})$
for all $t \in (\varepsilon - \tau, T]$ in 
\eqref{eq:randomeq8hw2982wh-z6},
and the fact that $e_2$ and $\nu_2$
form an orthonormal basis of 
$\spann(\{e_1, e_2\}) = \spann(\{\nu_1, \nu_2\})$,
this implies 
\begin{equation*}
\dot y^\tau(t) - (\dot u + \tau \dot h)(t)
=
\dot y^\tau(t) - \dot u(t)
=
\left \langle \dot u(t), e_2 \right \rangle e_2
- \dot u(t)
=
-\left \langle \dot u (t), \nu_2 \right \rangle \nu_2
\in 
- \NN_Z(y^\tau(t))
\text{ for a.a.\ } t \in (\varepsilon - \tau, T).
\end{equation*}
We thus indeed have
$y^\tau(0) = y_0 + \tau h_0$,
$y^\tau(t) \in Z$ for all $t \in [0, T]$,
and 
$\dot y^\tau(t) - (\dot u + \tau \dot h)(t) \in -\NN_Z(y^\tau(t))$
for a.a.\ $t \in (0, T)$
and, as a consequence, 
$y^\tau = \SS(u + \tau h, y_0 + \tau h_0)$ as claimed. 
Note that, from \eqref{eq:ytaudef}, \eqref{eq:randome2329839e3jp-2}, 
and
$\beta_1 = 0$ a.e.\ in $(\varepsilon, T)$,
we obtain that 
\begin{equation}
\label{eq:randomeq3636-111}
\left. \frac{y^\tau - y^0}{\tau} \right |_{[\varepsilon, T]}
=
\left. \frac{\SS(u + \tau h, y_0 + \tau h_0)  - \SS(u, y_0) }{\tau}\right |_{[\varepsilon, T]}
= \mathrm{const}
\qquad \forall\,0 < \tau \leq \tau_0
\end{equation}
and
\begin{equation}
\label{eq:randomeq3636-222}
\frac{y^\tau(\varepsilon) - y^0(\varepsilon)}{\tau}
=
\frac1\tau
\int_{\varepsilon - \tau}^\varepsilon
\left \langle \dot u(s), e_2 \right \rangle 
\dd s\, e_2
=
\left \langle 
-  e_1
, e_2 \right \rangle e_2
+
\frac1\tau
\int_{\varepsilon - \tau}^\varepsilon
\beta_1(s)
\dd s\,\left \langle 
 \nu_1
, e_2 \right \rangle  e_2
\qquad \forall\,0 < \tau \leq \tau_0. 
\end{equation}
As $\beta_1$ was chosen such that the limit 
in \eqref{eq:beta1limit} does not exist,
since $\left \langle 
 \nu_1
, e_2 \right \rangle \neq 0$ and $0 < \varepsilon < \vartheta$,
and since, for $\gamma \in \{0,1\}$,
we obtain situations in which only the 
initial value $y_0$ and the forcing term $u$
are perturbed, the assertion of the theorem 
now follows immediately
from \eqref{eq:ytaudef},
\eqref{eq:randomeq3636-111}, and 
\eqref{eq:randomeq3636-222}.
\end{proof}
\smallskip

\begin{remark}~
\begin{itemize}
\item
The construction in the 
proof of \cref{th:non_dir_diff}
suggests that it might be possible 
to obtain directional differentiability results 
for the operator $\SS$ in the presence of 
obtuse interior angles
if one 
restricts the attention to 
situations in which the 
time-derivative of the function $\PP(u, y_0)\colon [0,T] \to \R^d$
possesses left and right limits everywhere 
in $[0,T]$. We leave this topic for future research. 

\item 
As \eqref{eq:randomeq3636-222}
shows, in the situation considered 
in the proof of \cref{th:non_dir_diff},
limits (of subsequences) of difference quotients 
of $\SS$ do not only depend on 
where the multiplier maps $\lambda_i$
in \eqref{eq:mon_play} are nonconstant,
but also on how fast the maps $\lambda_i$
grow. This is a major difference to the 
non-obtuse case; see
\eqref{eq:unique_DD_system},
\eqref{eq:GR-crit_cone_alts},
and \cref{def:temporal_indices}.
In the elliptic setting, 
dependencies of this type typically occur 
when the admissible set of the 
considered variational inequality 
(interpreted as a subset of the underlying function space)
is not (extended) polyhedric
but 
possesses proper nonzero curvature at its boundary;
cf.\ \cite[Lemmas 5.3, 5.8]{Christof2018SSC}
and the references therein. 
The observations in 
\cref{sec:3} and the proof of \cref{th:non_dir_diff}
thus suggest that
the admissible set 
$\{v \in CBV([0,T];\R^d) \mid v(0) = y_0,\;v(t) \in Z \text{ for all }t\in [0,T]\}$
of the EVI \eqref{eq:EVI},
considered as a subset of the space $BV([0,T];\R^d)$,
possesses proper nonzero curvature 
in the sense of the second-order epi-derivatives 
studied in \cite{Christof2018SSC}
when $Z$ has an obtuse ridge 
and is 
extended polyhedric (in a suitably defined sense) when $Z$ is non-obtuse. 
Note that such structural differences 
between non-obtuse and  obtuse polyhedra
are completely unheard of in the 
context of elliptic variational inequalities; see \cite{Wachsmuth2018pointwiseElliptic}. 
We again leave this topic for future research. 
\end{itemize}
\end{remark}

\section{Remarks on applications in optimal control}
\label{sec:5}

We conclude this paper 
by briefly commenting on applications of \cref{th:main_summary}
in the field of optimal control.
For illustration purposes, 
we restrict our attention to the following 
model problem governed by \eqref{eq:EVI}:
\begin{equation*}
\label{eq:P}
\tag{P}
	\begin{aligned}
		\text{Minimize} 
		\quad &  \JJ(y, y(T), u)  \\
        \text{w.r.t.}
        \quad &y \in CBV([0, T];\R^d), \quad  u \in U,\\
		\text{s.t.}
		\quad & 
  \left \{
\begin{aligned}
&y(t) \in Z\quad \forall t \in [0, T], && y(0) = y_0,
\\
&\int_0^T \left \langle v - y, \dd (y - u) \right \rangle \geq 0 &&\forall v \in C([0, T]; Z).
\end{aligned} 
\right.
\end{aligned}
\end{equation*}
Our standing assumptions on the quantities in \eqref{eq:P} are as follows
(cf.\ \cite[Assumption 3.1, Corollary 5.1]{BrokateChristof2023}):
\medskip

\begin{assumption}[standing assumptions for \cref{sec:5}]~\label{ass:standing:sec:5}
\begin{enumerate}[label=\roman*)]
\item $T>0$ and $d \in \N$ are given and fixed.
\item 
$(U, \|\cdot\|_U)$ is a reflexive real Banach space
such that  $U \subset CBV([0,T];\R^d)$ holds and 
such that $U$ is continuously, densely, and compactly embedded into 
$C([0,T];\R^d)$.
\item $\JJ\colon L^\infty((0, T);\R^d) \times \R^d \times U \to \R$ 
is a function with the following properties:
\begin{enumerate}
\item 
$\JJ$ is Fr\'{e}chet differentiable
and 
$\partial_1 \JJ(y,  y(T), u) \in L^1((0, T);\R^d)
$ for all $(y, u) \in CBV([0,T];\R^d) \times U$. 
Here, $ L^1((0, T);\R^d)$ is identified with a subset 
of $L^\infty((0,T);\R^d)^*$ in the canonical way.
\item $\JJ$ is lower semicontinuous 
in the sense that, for 
all $\{(y_k, z_k, u_k) \} \subset C([0, T];\R^d) \times \R^d \times U$
satisfying $y_k \to y$ in $C([0, T];\R^d)$, 
$z_k \to z$ in $\R^d$, and $u_k \weakly u$ in $U$,
we have 
\[
\liminf_{k \to \infty} \JJ(y_k, z_k, u_k) \geq \JJ(y,z,u).
\]
\item $\JJ$ is radially unbounded in the sense that 
there exists a function $\rho\colon [0, \infty) \to \R$
satisfying  
\[
\qquad\qquad\qquad\rho(s) \to \infty \text{ for } s \to \infty
\qquad
\text{and}
\qquad
\JJ(y, z, u) \geq \rho\left ( \|u\|_U \right ) \quad \forall (y, z, u) \in C([0, T];\R^d) \times \R^d \times U.
\]
\end{enumerate}
\item 
$Z\subset \R^d$ is a full-dimensional convex non-obtuse polyhedron.
\item $y_0 \in Z$ is a given and fixed starting value.
\end{enumerate}
\end{assumption}

A prototypical example covered by the 
setting in \cref{ass:standing:sec:5} is 
\[
U := H^1((0,T);\R^d),
\qquad 
\JJ(y, z, u)
:=
\frac{\omega_1}2
\|y  - y_D \|_{L^2((0,T);\R^d)}^2
+
\frac{\omega_2}2
|z - y_T |^2
+
\frac{\omega_3}{2}
\|u\|_{H^1((0,T);\R^d)}^2,
\]
with given $\omega_1,\omega_2,\omega_3 > 0$,
$y_D \in L^2((0,T);\R^d)$,
and $y_T \in \R^d$. From standard arguments, 
we obtain:

\begin{theorem}[solvability of \eqref{eq:P}]
Problem \eqref{eq:P} possesses
a solution $(\bar u, \bar y)
\in U \times CBV([0,T];\R^d)$.
\end{theorem}

\begin{proof}
The solvability of \eqref{eq:P} follows 
straightforwardly from the properties of $U$ and $\JJ$
in \cref{ass:standing:sec:5}, the Lipschitz estimate 
\eqref{eq:C_S_continuity}, and the direct method of the calculus of 
variations. 
\end{proof}

The main insight that 
\cref{th:main_summary} provides in the context of the 
minimization problem \eqref{eq:P} is that 
it is possible to formulate a classical Bouligand stationarity condition 
for local minimizers, i.e., a standard first-order necessary 
optimality condition in terms of the directional
derivatives of the reduced objective function
$U \ni u \mapsto \JJ(\SS(u, y_0), \SS(u, y_0)(T), u) \in \R$.

\begin{corollary}[Bouligand stationarity condition]
\label{cor:bouligand_condition}
Suppose that $\bar u \in U$ is a locally optimal 
control of \eqref{eq:P} with associated 
state $\bar y := \SS(\bar u, y_0)$. 
Then the map $\SS$ is directionally differentiable 
at $(\bar u, y_0)$ in the sense of 
part \ref{main_sum:item:I} of 
\cref{th:main_summary} and it holds 
\begin{equation}
\label{eq:Bouligand_condition}
\begin{aligned}
\int_0^T
\langle 
\partial_1 \JJ(\bar y, \bar y(T), \bar u),
\SS'((\bar u, y_0);(h, 0))
\rangle \dd s
&+
\left\langle
\partial_2 \JJ(\bar y, \bar y(T), \bar u)
,
\SS'((\bar u, y_0);(h, 0))(T)\right \rangle
\\
&+
\left \langle \partial_3 \JJ(\bar y, \bar y(T), \bar u), h \right \rangle_{U^*,U}
\geq 
0\quad 
\forall h \in U.
\end{aligned}
\end{equation}
\end{corollary}
\begin{proof}
If $\bar u$ is a locally optimal control of 
the problem \eqref{eq:P}
with state $\bar y := \SS(\bar u, y_0)$ and $h \in U$ is fixed, then 
our assumptions,
\cref{th:main_summary},
the Lipschitz estimate \eqref{eq:C_S_continuity},
and the local optimality of $\bar u$
imply that $\SS$ is directionally differentiable 
at $(\bar u, y_0)$ in the sense of \eqref{eq:pointwise_dir_diff} and that 
\begin{equation*}
\begin{aligned}
0
&\leq \JJ(\SS(\bar u + \tau h, y_0), \SS(\bar u + \tau h, y_0)(T), \bar u + \tau h)
- \JJ(\bar y, \bar y(T), \bar u)
\\
&=
\int_0^T
\langle 
\partial_1 \JJ(\bar y, \bar y(T), \bar u),
 \SS(\bar u + \tau h, y_0) - \bar y 
\rangle \dd s
+
\left\langle
\partial_2 \JJ(\bar y, \bar y(T), \bar u),
\SS(\bar u + \tau h, y_0)(T) - \bar y(T)
\right \rangle
\\
&\hspace{6.8cm}+
\left \langle \partial_3 \JJ(\bar y, \bar y(T), \bar u), \tau h \right \rangle_{U^*,U}
+ o(\tau)
\end{aligned}
\end{equation*}
holds for all sufficiently small $\tau > 0$, 
where the Landau notation refers to the limit $\tau \to 0$.
If we divide by 
$\tau$ in the above and pass to the limit $\tau \to 0$, 
using \eqref{eq:Lipschitz_estimate_42},
the pointwise convergence in \eqref{eq:pointwise_dir_diff},
and the dominated convergence theorem, then the assertion 
of the corollary follows.
\end{proof}

Note that it is typically 
hard to use the Bouligand stationarity 
condition \eqref{eq:Bouligand_condition}
directly as a point of departure for the design of numerical solution algorithms.
The main appeal of \eqref{eq:Bouligand_condition} is that 
it constitutes the most natural first-order necessary optimality condition 
for the problem \eqref{eq:P} and, thus, 
can be used as a point of reference when comparing 
stationarity systems that have been derived 
via regularization or discretization techniques;
see \cite{Adam2014,Cao2016,Colombo2020,Cao2019,Colombo2016,Henrion2023,Cao2022,Arroud2018,dePinho2019,Stefanelli2017,Brokate2013,Herzog2022,Nour2022}
and the comparisons in \cite{Harder2017}. 
A very interesting question in this context is whether 
the characterization of
the directional derivatives of $\SS$
by means of \eqref{eq:char_sys_1} 
can be used 
to reformulate the  Bouligand stationarity 
condition \eqref{eq:Bouligand_condition} 
as an equivalent primal-dual optimality system in qualified form. 
In the case $d=1$ and $Z = [-r,r]$, $r>0$, it has been demonstrated 
in \cite{BrokateChristof2023}
that 
this is indeed possible and that
a control $\bar u \in U$ with state $\bar y := \SS(\bar u, y_0)$ 
satisfies \eqref{eq:Bouligand_condition} if and only if 
there exist an adjoint state 
$\bar p \in BV([0,T])$ and a multiplier $\bar \mu \in G_r([0,T])^*$
such that the following \emph{strong stationarity system}
is satisfied:
\begin{equation}
\label{eq:strongstatsys}
\begin{gathered}
\bar p(0) = \bar p(T) = 0,
\qquad 
\bar p(t) = \bar p(t-)~\forall t \in [0,T),
\qquad 
\bar p(t-) \in K^\ptw_{\crit}(\bar y, \bar u)(t)~\forall t \in [0,T],
\\
\left \langle \bar \mu, z \right \rangle_{G_r([0,T])^*,G_r([0,T])} \geq 0\quad \forall z \in \KK_{G_r}^{\red,\crit}(\bar y, \bar u)  ,
\qquad 
\int_0^T h \,\dd \bar p = \left \langle \partial_3 \JJ(\bar y, \bar y(T), \bar u), h \right \rangle_{U^*,U} ~\forall h \in U,
\\
-\int_0^T z \,\dd \bar p  
=
\int_0^T \partial_1 \JJ(\bar y, \bar y(T), \bar u) z \, \dd s
+
\partial_2 \JJ(\bar y, \bar y(T), \bar u)z(T)
-
\left \langle \bar \mu, z \right \rangle_{G_r([0,T])^*,G_r([0,T])}\qquad \forall z \in G_r([0,T]).
\end{gathered}
\end{equation}
Here, $K^{\mathrm{ptw}}_{\crit}(\bar y,\bar u) \colon [0, T] \rightrightarrows \R$ is the same map as in \eqref{eq:crit_cone_equivalence}
and $\KK_{G_r}^{\red,\crit}(\bar y, \bar u) $ denotes a 
suitably defined subset of the critical cone 
$\KK_{G_r}^{\crit}(\bar y,\bar u)$;
see \cite[Definitions 4.7, 6.1]{BrokateChristof2023}.
The main instrument for establishing \eqref{eq:strongstatsys}
in \cite{BrokateChristof2023} is a version 
of \cref{theorem:tempoly} 
which shows that
the space $BV_r([0,T])$
in the first line of \eqref{eq:temp_poly} 
can be replaced by $C^\infty([0,T])$
when functions $z \in \KK_{G_r}^{\red,\crit}(\bar y, \bar u) $ 
are considered and $d=1$ holds; see \cite[Theorem 6.5]{BrokateChristof2023}.
Whether it is possible to refine the temporal polyhedricity 
result in \cref{theorem:tempoly} 
in a similar way for $d>1$ 
and to obtain a strong stationarity system for \eqref{eq:P}
analogous to \eqref{eq:strongstatsys} for arbitrary dimensions 
is an open problem.


\bibliographystyle{plain}
\bibliography{references}

\begin{thebibliography}{10}

\bibitem{Acary2010}
V.~Acary, O.~Bonnefon, and B.~Brogliato.
\newblock {\em Nonsmooth Modeling and Simulation for Switched Circuits}.
\newblock Lecture Notes in Electrical Engineering. Springer, 2010.

\bibitem{Adam2014}
L.~Adam and J.~Outrata.
\newblock On optimal control of a sweeping process coupled with an ordinary
  differential equation.
\newblock {\em Discrete Contin. Dyn. Syst. Ser. B}, 19(9):2709--2738, 2014.

\bibitem{Adly2018}
S.~Adly and L.~Bourdin.
\newblock Sensitivity analysis of variational inequalities via twice
  epi-differentiability and proto-differentiability of the proximity operator.
\newblock {\em SIAM J. Optim.}, 28(2):1699--1725, 2018.

\bibitem{Arroud2018}
C.~E. Arroud and G.~Colombo.
\newblock A maximum principle for the controlled sweeping process.
\newblock {\em Set-Valued Var. Anal.}, 26(3):607--629, 2018.

\bibitem{Attouch2006}
H.~Attouch, G.~Buttazzo, and G.~Michaille.
\newblock {\em Variational Analysis in {S}obolev and {BV} Spaces}.
\newblock SIAM, Philadelphia, 2006.

\bibitem{Berger2009}
M.~Berger, M.~Cole, and S.~Levy.
\newblock {\em Geometry II}.
\newblock Universitext. Springer, Berlin, Heidelberg, 2009.

\bibitem{Betz2019}
L.~Betz.
\newblock Strong stationarity for optimal control of a nonsmooth coupled
  system: Application to a viscous evolutionary variational inequality coupled
  with an elliptic {PDE}.
\newblock {\em SIAM J. Optim.}, 29(4):3069--3099, 2019.

\bibitem{BonnansCominettiShapiro1998}
J.~F. Bonnans, R.~Cominetti, and A.~Shapiro.
\newblock Sensitivity analysis of optimization problems under second order
  regular constraints.
\newblock {\em Math. Oper. Res.}, 23(4):806--831, 1998.

\bibitem{BonnansShapiro2000}
J.~F. Bonnans and A.~Shapiro.
\newblock {\em Perturbation Analysis of Optimization Problems}.
\newblock Springer Series in Operations Research. Springer, New York, 2000.

\bibitem{Brokate2020}
M.~Brokate.
\newblock {N}ewton and {B}ouligand derivatives of the scalar play and stop
  operator.
\newblock {\em Math. Model. Nat. Phenom.}, 15, 2020.
\newblock Art.~51.

\bibitem{BrokateChristof2023}
M.~Brokate and C.~Christof.
\newblock Strong stationarity conditions for optimal control problems governed
  by a rate-independent evolution variational inequality.
\newblock {\em SIAM J. Control Optim.}, 61(4):2222--2250, 2023.

\bibitem{Brokate2013}
M.~Brokate and P.~Krej\v{c}\'{\i}.
\newblock Optimal control of {ODE} systems involving a rate independent
  variational inequality.
\newblock {\em Discrete Contin. Dyn. Syst. Ser. B}, 18(2):331--348, 2013.

\bibitem{Brokate2015}
M.~Brokate and P.~Krej\v{c}\'{\i}.
\newblock Weak differentiability of scalar hysteresis operators.
\newblock {\em Discrete Contin. Dyn. Syst.}, 35(6):2405--2421, 2015.

\bibitem{Brokate2021}
M.~Brokate and P.~Krej\v{c}\'{\i}.
\newblock A variational inequality for the derivative of the scalar play
  operator.
\newblock {\em J. Appl. Numer. Optim.}, 3(2):263--283, 2021.

\bibitem{Bruns2009}
W.~Bruns and J.~Gubeladze.
\newblock {\em Polytopes, Rings, and K-Theory}.
\newblock Springer Monographs in Mathematics. Springer, New York, 2009.

\bibitem{Cao2022}
T.~H. Cao, N.~T. Khalil, B.~S. Mordukhovich, D.~Nguyen, T.~Nguyen, and F.~L.
  Pereira.
\newblock Optimization of controlled free-time sweeping processes with
  applications to marine surface vehicle modeling.
\newblock {\em IEEE Control Syst. Lett.}, 6:782--787, 2022.

\bibitem{Cao2016}
T.~H. Cao and B.~S. Mordukhovich.
\newblock Optimal control of a perturbed sweeping process via discrete
  approximations.
\newblock {\em Discrete Contin. Dyn. Syst. Ser. B}, 21(10):3331--3358, 2016.

\bibitem{Cao2019}
T.~H. Cao and B.~S. Mordukhovich.
\newblock Optimal control of a nonconvex perturbed sweeping process.
\newblock {\em J. Differential Equations}, 266(2):1003--1050, 2019.

\bibitem{Cao2022-2}
T.~H. Cao, B.~S. Mordukhovich, D.~Nguyen, and T.~Nguyen.
\newblock Applications of controlled sweeping processes to nonlinear crowd
  motion models with obstacles.
\newblock {\em IEEE Control Syst. Lett.}, 6:740--745, 2022.

\bibitem{ChristofPhd2018}
C.~Christof.
\newblock {\em Sensitivity Analysis of Elliptic Variational Inequalities of the
  First and the Second Kind}.
\newblock PhD thesis, Technische Universit{\"a}t Dortmund, 2018.

\bibitem{Christof2019parob}
C.~Christof.
\newblock Sensitivity analysis and optimal control of obstacle-type evolution
  variational inequalities.
\newblock {\em SIAM J. Control Optim.}, 57(1):192--218, 2019.

\bibitem{Christof2022}
C.~Christof, C.~Meyer, B.~Schweizer, and S.~Turek.
\newblock Strong stationarity for optimal control of variational inequalities
  of the second kind.
\newblock In M.~Hinterm{\"u}ller, R.~Herzog, C.~Kanzow, M.~Ulbrich, and
  S.~Ulbrich, editors, {\em Non-Smooth and Complementarity-Based Distributed
  Parameter Systems: Simulation and Hierarchical Optimization}, pages 307--327.
  Springer, Cham, 2022.

\bibitem{Christof2018SSC}
C.~Christof and G.~Wachsmuth.
\newblock No-gap second-order conditions via a directional curvature
  functional.
\newblock {\em SIAM J. Optim.}, 28(3):2097--2130, 2018.

\bibitem{ChristofWachsmuth2020}
C.~Christof and G.~Wachsmuth.
\newblock Differential sensitivity analysis of variational inequalities with
  locally {L}ipschitz continuous solution operators.
\newblock {\em Appl. Math. Optim.}, 81:23--62, 2020.

\bibitem{ChristofQVI2023}
C.~Christof and G.~Wachsmuth.
\newblock Lipschitz stability and {H}adamard directional differentiability for
  elliptic and parabolic obstacle-type quasi-variational inequalities.
\newblock {\em SIAM J. Control Optim.}, 60(6):3430--3456, 2022.

\bibitem{Colombo2016}
G.~Colombo, R.~Henrion, D.~H. Nguyen, and B.~S. Mordukhovich.
\newblock Optimal control of the sweeping process over polyhedral controlled
  sets.
\newblock {\em J. Differential Equations}, 260(4):3397--3447, 2016.

\bibitem{Colombo2020}
G.~Colombo, B.~S. Mordukhovich, and D.~Nguyen.
\newblock Optimization of a perturbed sweeping process by constrained
  discontinuous controls.
\newblock {\em SIAM J. Control Optim.}, 58(4):2678--2709, 2020.

\bibitem{dePinho2019}
M.~d.~R. de~Pinho, M.~M.~A. Ferreira, and G.~V. Smirnov.
\newblock Optimal control involving sweeping processes.
\newblock {\em Set-Valued Var. Anal.}, 27:523--548, 2019.

\bibitem{Desch1999}
W.~W. Desch and J.~Turi.
\newblock The stop operator related to a convex polyhedron.
\newblock {\em J. Differential Equations}, 157(2):329--347, 1999.

\bibitem{Do1992}
C.~N. Do.
\newblock Generalized second-order derivatives of convex functions in reflexive
  {B}anach spaces.
\newblock {\em Trans. Amer. Math. Soc.}, 334(1):281--301, 1992.

\bibitem{Dupuis1991}
P.~Dupuis and H.~Ishii.
\newblock On {L}ipschitz continuity of the solution mapping to the {S}korokhod
  problem, with applications.
\newblock {\em Stochastics and Stochastic Reports}, 35(1):31--62, 1991.

\bibitem{Evans2010}
L.~C. Evans.
\newblock {\em Partial Differential Equations}.
\newblock AMS, Providence, RI, 2nd edition, 2010.

\bibitem{Faigle2002}
U.~Faigle, W.~Kern, and G.~Still.
\newblock {\em Algorithmic Principles of Mathematical Programming}.
\newblock Texts in the Mathematical Sciences. Springer, 2002.

\bibitem{FitzpatrickPhelps1982}
S.~Fitzpatrick and R.~R. Phelps.
\newblock Differentiability of the metric projection in {H}ilbert space.
\newblock {\em Trans. Amer. Math. Soc.}, 270(2):483--501, 1982.

\bibitem{Haraux1977}
A.~Haraux.
\newblock How to differentiate the projection on a convex set in {H}ilbert
  space. {S}ome applications to variational inequalities.
\newblock {\em J. Math. Soc. Japan}, 29(4):615--631, 1977.

\bibitem{Harder2017}
F.~Harder and G.~Wachsmuth.
\newblock Comparison of optimality systems for the optimal control of the
  obstacle problem.
\newblock {\em GAMM-Mitt.}, 40(4):312--338, 2018.

\bibitem{Henrion2023}
R.~Henrion, A.~Jourani, and B.~S. Mordukhovich.
\newblock Controlled polyhedral sweeping processes: Existence, stability, and
  optimality conditions.
\newblock {\em J. Differential Equations}, 366:408--443, 2023.

\bibitem{Herzog2022}
R.~Herzog, D.~Knees, C.~Meyer, M.~Sievers, A.~St{\"o}tzner, and S.~Thomas.
\newblock Rate-independent systems and their viscous regularizations: Analysis,
  simulation, and optimal control.
\newblock In M.~Hinterm{\"u}ller, R.~Herzog, C.~Kanzow, M.~Ulbrich, and
  S.~Ulbrich, editors, {\em Non-Smooth and Complementarity-Based Distributed
  Parameter Systems: Simulation and Hierarchical Optimization}, pages 121--144.
  Springer, Cham, 2022.

\bibitem{Hintermueller2009}
M.~Hinterm\"{u}ller and I.~Kopacka.
\newblock Mathematical programs with complementarity constraints in function
  space: {C-} and strong stationarity and a path-following algorithm.
\newblock {\em SIAM J. Optim.}, 20(2):868--902, 2009.

\bibitem{Jarusek2003}
J.~Jaru\v{s}ek, M.~Krbec, M.~Rao, and J.~Soko\l{}owski.
\newblock Conical differentiability for evolution variational inequalities.
\newblock {\em J. Differential Equations}, 193(1):131--146, 2003.

\bibitem{Krasnoselskii2012}
M.~A. Krasnosel'skii and A.~V. Pokrovskii.
\newblock {\em Systems with Hysteresis}.
\newblock Springer, Berlin, Heidelberg, 1st edition, 1989.

\bibitem{Krejci1996}
P.~Krej\v{c}\'{\i}.
\newblock {\em Hysteresis, Convexity and Dissipation in Hyperbolic Equations},
  volume~8 of {\em Mathematical Sciences and Applications}.
\newblock Gakk\={o}tosho, Tokyo, 1996.

\bibitem{Krejci1999}
P.~Krej\v{c}\'{\i}.
\newblock Evolution variational inequalities and multidimensional hysteresis
  operators.
\newblock In P.~Dr\'{a}bek, P.~Krej\v{c}\'{\i}, and P.~Tak\'{a}\v{c}, editors,
  {\em Nonlinear Differential Equations}, volume 404 of {\em Research Notes in
  Mathematics}, pages 47--110. Chapman \& Hall/CRC, London, 1999.

\bibitem{Krejci2014Compare}
P.~Krej\v{c}\'{\i} and V.~Recupero.
\newblock Comparing {BV} solutions of rate independent processes.
\newblock {\em J. Convex Anal.}, 21(1):121--146, 2014.

\bibitem{Mangasarian1994}
O.~L. Mangasarian.
\newblock {\em Nonlinear Programming}.
\newblock Number~10 in SIAM's Classics in Applied Mathematics. SIAM, 2nd
  edition, 1994.

\bibitem{Mielke2015}
A.~Mielke and T.~Roub\'{\i}\v{c}ek.
\newblock {\em Rate-Independent Systems}.
\newblock Number 193 in Applied Mathematical Sciences. Springer, New York,
  2015.

\bibitem{Mignot1976}
F.~Mignot.
\newblock Contr\^{o}le dans les in\'{e}quations variationelles elliptiques.
\newblock {\em J. Funct. Anal.}, 22:130--185, 1976.

\bibitem{MignotPuel1984}
F.~Mignot and J.~P. Puel.
\newblock Optimal control in some variational inequalities.
\newblock {\em SIAM J. Control Optim.}, 22(3):466--476, 1984.

\bibitem{Monteiro2019}
G.~A. Monteiro, A.~Slav\'{\i}k, and M.~Tvrd\'{y}.
\newblock {\em Kurzweil-Stieltjes Integral: Theory and Applications}.
\newblock Number~15 in Series in Real Analysis. World Scientific, Singapore,
  2019.

\bibitem{Moreau1973}
J.-J. Moreau.
\newblock Probl\`eme d'\'{e}volution associ\'{e} \`a un convexe mobile d'un
  espace hilbertien.
\newblock {\em C. R. Acad. Sci. Paris S\'{e}r. A-B}, 276:A791--A794, 1973.

\bibitem{Moreau1977}
J.-J. Moreau.
\newblock Evolution problem associated with a moving convex set in a {H}ilbert
  space.
\newblock {\em J. Differential Equations}, 26(3):347--374, 1977.

\bibitem{Moreau2011}
J.-J. Moreau.
\newblock On unilateral constraints, friction and plasticity.
\newblock In G.~Capriz and G.~Stampacchia, editors, {\em New Variational
  Techniques in Mathematical Physics}, pages 171--322. Springer, Berlin,
  Heidelberg, 2011.

\bibitem{Noll1995}
D.~Noll.
\newblock Directional differentiability of the metric projection in {H}ilbert
  space.
\newblock {\em Pacific J. Math.}, 170(2):567--592, 1995.

\bibitem{Nour2022}
C.~Nour and V.~Zeidan.
\newblock Optimal control of nonconvex sweeping processes with separable
  endpoints: Nonsmooth maximum principle for local minimizers.
\newblock {\em J. Differential Equations}, 318:113--168, 2022.

\bibitem{Shapiro1994}
A.~Shapiro.
\newblock Existence and differentiability of metric projections in {H}ilbert
  spaces.
\newblock {\em SIAM J. Optim.}, 4(1):130--141, 1994.

\bibitem{Siddiqi2002}
A.~H. Siddiqi, P.~Manchanda, and M.~Brokate.
\newblock On some recent developments concerning {M}oreau's sweeping process.
\newblock In A.~H. Siddiqi and M.~Ko{\v{c}}vara, editors, {\em Trends in
  Industrial and Applied Mathematics: Proceedings of the 1st International
  Conference on Industrial and Applied Mathematics of the Indian Subcontinent},
  pages 339--354. Springer, Boston, MA, 2002.

\bibitem{Stefanelli2017}
U.~Stefanelli, D.~Wachsmuth, and G.~Wachsmuth.
\newblock Optimal control of a rate-independent evolution equation via viscous
  regularization.
\newblock {\em Discrete Contin. Dyn. Syst. Ser. S}, 10(6):1467--1485, 2017.

\bibitem{Wachsmuth2018pointwiseElliptic}
G.~Wachsmuth.
\newblock Pointwise constraints in vector-valued {S}obolev spaces: with
  applications in optimal control.
\newblock {\em Appl. Math. Optim.}, 77(3):463--497, 2018.

\bibitem{Wachsmuth2019}
G.~Wachsmuth.
\newblock A guided tour of polyhedric sets: basic properties, new results on
  intersections and applications.
\newblock {\em J. Convex Anal.}, 26(1):153--188, 2019.

\bibitem{Zarantonello1971}
E.~H. Zarantonello.
\newblock Projections on convex sets in {H}ilbert space and spectral theory {I}
  and {II}.
\newblock In {\em Contributions to nonlinear functional analysis ({P}roc.
  {S}ympos., {M}ath. {R}es. {C}enter, {U}niv. {W}isconsin, {M}adison, {W}is.,
  1971)}, pages 237--424. Academic Press, New York-London, 1971.

\end{thebibliography}
\end{document}